\documentclass[10pt,reqno]{amsart}
\usepackage[T1]{fontenc}
\usepackage{lmodern}
\usepackage{microtype}
\usepackage{eucal,fullpage,  amsmath,amsthm,amssymb,mathrsfs, stmaryrd,color,enumerate,accents}
\usepackage[all]{xy}
\usepackage{url}

\numberwithin{equation}{section}

\usepackage{mymacros}

\usepackage{tikz-cd}

\usepackage{relsize}
\usepackage[bbgreekl]{mathbbol}
\usepackage{amsfonts}
\usepackage[dvipsnames]{xcolor} % gives OliveGreen, etc.
\usepackage[pdfusetitle, colorlinks, pagebackref, hyperindex]{hyperref}

\hypersetup{
  linkcolor=RoyalBlue,        % brighter than Black
  citecolor=RawSienna,    % brighter than OliveGreen
  urlcolor=RoyalBlue, 
}

\newcommand{\cO}{{\mathcal{O}}}
\newcommand{\cL}{{\mathcal{L}}}

\newcommand{\cA}{{\mathcal{A}}}

\newcommand{\cD}{{\mathcal{D}}}

\newcommand{\cE}{\mathcal{E}}
\newcommand{\cF}{\mathcal{F}}

\newcommand{\cC}{\mathcal{C}}

\newcommand{\catinf}{\mathcal{C}\mathrm{at}_\infty}

\newcommand{\A}{\mathbf{A}}
\newcommand{\G}{\mathbf{G}}
\newcommand{\ZZ}{\mathbf{Z}}

\newcommand{\C}{\mathbf{C}}
\newcommand{\FF}{\mathbf{F}}
\newcommand{\QQ}{\mathbf{Q}}

\newcommand{\spec}{{\mathrm{Spec}}}

\newcommand{\Fun}{\mathrm{Fun}}
\newcommand{\Maps}{\mathrm{Maps}}
\newcommand{\End}{\mathrm{End}}

\newcommand{\Shv}{\mathrm{Shv}}

\newcommand{\Qcoh}{\mathrm{QCoh}}

\newcommand{\Sym}{\mathrm{Sym}}

\newcommand{\Alg}{\mathrm{Alg}}

\newcommand{\Gr}{\mathrm{Gr}}

\renewcommand{\L}{\mathrm{L}}

\newcommand{\Pic}{\mathrm{Pic}}

\newcommand{\id}{\mathrm{id}}
\newcommand{\ev}{\mathrm{ev}}

\newcommand{\fil}{\mathrm{Fil}}

\newcommand{\gr}{\mathrm{gr}}

\newcommand{\Perf}{\mathrm{Perf}}
\newcommand{\into}{\hookrightarrow}
\newcommand{\prl}{\mathrm{Pr}^L}

\renewcommand{\inf}{\mathrm{inf}}

\newcommand{\st}{\mathrm{st}}

\newcommand{\Tot}{\mathrm{Tot}}

\newcommand{\tensor}{\otimes}
\newcommand{\cofib}{\mathrm{cofib}}
\newcommand{\cn}{\mathrm{cn}}
\newcommand{\ccn}{\mathrm{ccn}}

\newcommand{\res}{\mathrm{Res}}
\newcommand{\pij}{(\pi_!)^\wedge_J}
\newcommand{\splt}{\mathrm{Split}}

\DeclareSymbolFontAlphabet{\mathbb}{AMSb} %to ensure that the meaning of \mathbb does not change
\DeclareSymbolFontAlphabet{\mathbbl}{bbold}
\newcommand{\Prism}{{\mathlarger{\mathbbl{\Delta}}}}

\newcommand{\comment}[1]{}

\newcommand{\cosimp}[3]{\xymatrix@1{#1 \ar@<.4ex>[r] \ar@<-.4ex>[r] & {\ }#2 \ar@<0.8ex>[r] \ar[r] \ar@<-.8ex>[r] & {\ } #3 \ar@<1.2ex>[r] \ar@<.4ex>[r] \ar@<-.4ex>[r] \ar@<-1.2ex>[r] & \cdots }}
\newcommand{\simp}[3]{%
  \xymatrix@1{%
    \cdots \ar@<1.2ex>[r] \ar@<.4ex>[r] \ar@<-.4ex>[r] \ar@<-1.2ex>[r] &
    {\ }#3 \ar@<0.8ex>[r] \ar[r] \ar@<-.8ex>[r] &
    {\ }#2 \ar@<.4ex>[r] \ar@<-.4ex>[r] &
    #1
  }%
}
\newcommand{\simpII}[2]{%
  \xymatrix@1{%
    \cdots \ar@<0.8ex>[r] \ar[r] \ar@<-0.8ex>[r] &
    {\ }#2 \ar@<0.4ex>[r] \ar@<-0.4ex>[r] &
    #1
  }%
}

\begin{document}
\newtheorem{theorem}{Theorem}[section]
\newtheorem*{theorem*}{Theorem}
\newtheorem*{condition*}{Condition}
\newtheorem*{definition*}{Definition}
\newtheorem{proposition}[theorem]{Proposition}
\newtheorem{lemma}[theorem]{Lemma}
\newtheorem{corollary}[theorem]{Corollary}
\newtheorem{claim}[theorem]{Claim}

\theoremstyle{definition}
\newtheorem{definition}[theorem]{Definition}
\newtheorem{question}[theorem]{Question}
\newtheorem{goal}[theorem]{Goal}
\newtheorem{remark}[theorem]{Remark}
\newtheorem{guess}[theorem]{Guess}
\newtheorem{example}[theorem]{Example}
\newtheorem{condition}[theorem]{Condition}
\newtheorem{warning}[theorem]{Warning}
\newtheorem{caution}[theorem]{Caution}
\newtheorem{notation}[theorem]{Notation}
\newtheorem{construction}[theorem]{Construction}
\newtheorem{recollection}[theorem]{Recollection}
\newtheorem{variant}[theorem]{Variant}

\title{Derived algebras on formal stacks and prismatic gauges}
\begin{abstract}
This paper studies how the theory of derived algebras (in the sense of Bhatt-Mathew and Raksit) interacts with formal derived geometry, specifically the formal derived stacks which show up in the theory of prismatization. As an application we prove some classification theorems for derived algebras in quasi-coherent sheaves on a certain class of filtered \emph{formal} stacks, which includes those whose quasi-coherent sheaves are prismatic gauges over a perfectoid ring. Along the way, among other things, we study the behavior of derived algebras along schematic quasi-affine morphisms in derived geometry, and for example, classify derived algebras on the source as precisely those derived algebras on the target which receive a map from the pushforward of the structure sheaf of the source. We also indicate how to extend some of our results to (formal) classifying stacks of diagonalizable group schemes.

As an aside, we also show some classification theorems even for quasi-coherent sheaves on formal stacks which (to our knowledge) weren't available in the literature on derived geometry previously.
These results are motivated by forthcoming work of the author \cite{Sah25Nyg} and \cite{Sah25Syn} but are hoped to be generally useful.
\end{abstract}
\author{Shubhankar Sahai}
\address{Department of Mathematics, University of California San Diego, La Jolla, CA 92093, USA}
\email{ssahai@ucsd.edu}
\subjclass[2020]{Primary $14\text{A}30$; Secondary $14\text{F}08$, $14\text{B}20$,$14\text{F}30$, $18\text{N}60$, $18\text{C}15$ }
\maketitle
%\noindent\hfill\textit{Draft version---not for distribution}\par
\tableofcontents
\section{Introduction}

The paper grew out of an attempt to understand the interaction of derived algebras of Bhatt-Mathew and Raksit \cite{Rak20}, which serve as a robust notion of non-connective animated (equivalently simplicial commutative) rings in many settings, with derived completions and derived formal geometry. As a consequence, in this note we develop some basic results on non-connective formal geometry which seem to not be in the literature.

Another motivation was to extend the classification, due to Simpson \cite{Simpson1997HodgeFiltrationNonabelian}, of the filtered derived category over a discrete ring as quasi-coherent sheaves on $\agm$, to derived algebras in the completed setting, following the modern techniques adopted by Moulinos \cite{Moulinos2021GeometryOfFiltrations}. As a consequence of this second motivation, we use the techniques developed in this note to classify derived algebras in quasi-coherent sheaves on a certain class of filtered formal stacks which live over $\bgm$ and factor over the canonical projection $\agm\to \bgm.$

The overarching motivation for proving these results stems from forthcoming work of the author \cite{Sah25Nyg} and \cite{Sah25Syn}, and are already used in \cite{Sah25DeRhamAffineness}. 

While we explain these works at some length in \S~\ref{sec: forthcoming work}, we are hopeful that the results presented herein will be more widely applicable. As an example, some of our results in \S~\ref{sec: jcomp setting} and \S~\ref{sec: dalg along qaff} and already seem useful in developing a Tannakian framework (in the sense of \cite{Lurie2011DAGVIII}) in the formal derived setting, which is the subject of ongoing investigations \cite{MundingerSahai_inprep} (for example see Remark \ref{remark: definition of formal stack}).

We will begin with some extended background in \S~\ref{sec: background} to contextualize the technical nature of our results. An explanation of our results themselves is in \S~\ref{sec: results}.

\subsection{Background.}\label{sec: background}
The notion of animated or simplicial commutative rings, introduced systematically in the pioneering work of Quillen \cite{Quillen1970HomologyCommutativeRings} and \cite{Quillen1970OnCohomologyCommutativeRings}, is now an essential tool in (derived) algebraic geometry. Indeed, the theory is now quite mature and standard; for example, see \cite{BraggOlsson2021RepresentabilityCohomologyFiniteFlat} for some beautiful recent applications to questions of representability of flat cohomology.

The difference between an animated ring and a general connective $\einf$-ring is that the latter, in positive or mixed characteristic, will not be equivalent to a simplicial commutative ring. In other words, the multiplication on a general connective $\einf$-ring is not \emph{strict} (for a precise modern explanation of this notion of strictness, see \cite[\S~25.1.5]{Lur18}).

While animated rings are sufficient for many purposes, in (even classical) algebraic geometry one has to often work with \emph{coconnective} $\einf$-rings which are not animated rings, but admit a notion of strictification in the sense that they can be modeled (along with the multiplication) by cosimplicial commutative rings.

\begin{example}\label{eg: intro k3}
  Let $k$ be a perfect field of positive characteristic and let $X/k$ be an ordinary $K3$-surface i.e. the Frobenius endomorphism $F:H^2(X,\cO_X)\to H^2(X,\cO_X)$ is an isomorphism. 

  The coconnective $\einf$-algebra given by the derived ring of global sections $\Gamma(X,\cO_X)$ then can be computed by a Cech complex giving rise to a cosimplicial model for its multiplication. 

  This cosimplicial ring is not formal in the sense that it is not isomorphic to the graded ring $H^0(X,\cO_X)\oplus H^2(X,\cO_X)[-2]$ as the following argument we learned from S. Petrov\footnote{All mistakes are, of course, our own.} shows. The non-formality can be checked at the level of $\einf$-rings and one notes that the power operation $P^0\colon H^i(X,\cO_X)\to H^i(X,\cO_X)$ of \cite[Remark 2.27]{LurieDAGXIII}, which corresponds to the Frobenius endomorphism of $\Gamma(X,\cO_X)$ by \cite[Proposition 1.6]{petrov_de_rham}, is not zero for $i=2$. 
\end{example}

As Example \ref{eg: intro k3} shows, even in classical algebraic geometry one has to work with non-formal cosimplicial rings. This theory has a very different feel from the theory of animated rings. For example if $A$ is an animated ring, then $A$ can be thought of as a pro-nilpotent thickening of $\pi_0(A)$. There is no natural substitute for this deformation theoretic interpretation in the cosimplicial world (at least to the author's knowledge). Another aspect, already implicit in Example \ref{eg: intro k3}, is that cosimplicial commutative rings carry a Frobenius, while animated rings do not \cite[Proposition 11.6]{WittVectorGrassmannian}.

The theory of derived rings invented by Bhatt-Mathew, Raksit \cite{Rak20} and Brantner-Mathew \cite{BrantnerMathewActa_toappear} provides a robust theory of `nonconnective' animated rings, gluing both the simplicial and cosimplicial degrees, in quite general contexts akin to commutative algebra objects in $\icats$; for example one has a notion of derived algebras in quasi-coherent sheaves on (classical) algebraic stacks \cite[Appendix A.2]{KubrakPrikhodko2024PadicHodgeArtinStacks}. We briefly recall the theory now.

\begin{recollection}[Classical monadicity of rings]\label{rec: classical monad intro}
    Let $\Mod_{\ZZ}^\heartsuit$ denote the abelian category of discrete $\ZZ$-modules and let $\mathrm{Rings}$ be the $1$-category of discrete rings. Then the forgetful functor $U\colon \mathrm{Rings}\to  \Mod_{\ZZ}^\heartsuit$ admits a left adjoint $\Sym_{\ZZ}(-):=\bigoplus_{i\geq 0} \Sym^i_{\ZZ}(-).$ 
    The composite $$U\circ \Sym_{\ZZ} \colon \Mod_{\ZZ}^\heartsuit\to \Mod_{\ZZ}^\heartsuit$$ is a monad wherein the multiplication is induced by the natural map $\Sym^i_{\ZZ}(\Sym_{\ZZ}^j(M))\to \Sym_{\ZZ}^{i.j}(M)$, where $M\in \Mod_{\ZZ}^\heartsuit.$
    It is not too hard to show that this adjunction is monadic by Barr-Beck, whence it induces an equivalence of $1$-categories 
    $$\mathrm{Rings}\simeq \Mod_{U\circ \Sym_{\ZZ}^*}(\Mod_{\ZZ}^\heartsuit).$$

\end{recollection}

The monadicity theorem above admits a robust generalization to the animated setting with \emph{almost the same proof}, but with Barr-Beck replaced by Barr-Beck-Lurie.

\begin{recollection}[Monadicity for animated rings]\label{rec: animated monadicity intro}
    Now let $\aring$ be the $\icat$ of animated (or simplicial commutative) rings and $(\Mod_\ZZ)_{\geq 0}$ the $\icat$ of connective $\ZZ$-modules. There is an obvious conservative and limit preserving forgetful functor $U\colon \aring \to (\Mod_\ZZ)_{\geq 0}$. By the adjoint functor theorem for presentable $\icats$, there exists a left adjoint $\L\Sym_{\ZZ}\colon (\Mod_\ZZ)_{\geq 0}\to \aring$ and can be described explicitly as $$\L\Sym_{\ZZ}=\bigoplus_{i\geq 0} \L\Sym^n_{\ZZ}(-),$$ where $\L\Sym^n_{\ZZ}$, the summand in the $n$th degree on the right hand side, is the \emph{animation} of the classical $\Sym_{\ZZ}^n$ functor of Recollection \ref{rec: classical monad intro}.
    It again follows by a simple application of the $\infty$-categorical Barr-Beck (a.k.a Barr-Beck-Lurie) that the forgetful functor $U$ is monadic and exhibits an equivalence of $\icats$

    $$\aring\simeq \Mod_{\L\Sym_{\ZZ}}((\Mod_\ZZ)_{\geq 0})$$ where by abuse of notation we identify $\L\Sym_\ZZ$ with the composite $U\circ \L\Sym_{\ZZ}.$
\end{recollection}

Note that $\L\Sym_{\ZZ}$ in Recollection \ref{rec: animated monadicity intro} uses the functors $\L\Sym_{\ZZ}$ on the connective part of $\Mod_\ZZ$, by animating (or left Kan extending) the values of $\Sym_\ZZ$ on finite free (hence discrete) $\ZZ$-modules. If one wants to incorporate the notion of global sections as in Example \ref{eg: intro k3} as well as animated rings into an ambient category of possible non-connective animated rings, then $\L\Sym_{\ZZ}$ as described above doesn't work.

However the idea going back to Illusie \cite[\S~I.4]{Illusie1971CotangentComplexI} is that if one can appropriately define a monad $$\widetilde{\L\Sym_{\ZZ}}\colon \Mod_\ZZ\to \Mod_\ZZ,$$ whose values extend those of $\L\Sym_\ZZ$ on the connective part and the right Kan extension of $\Sym_\ZZ$ on the coconnective part, then one would have a robust notion of non-connective animated rings.

The endofunctors $\L\Sym^n_\ZZ$ for $n>0$ were extended on $\Mod_\ZZ$ by using a notion of \emph{right-left} extensions of functors by Brantner-Mathew \cite[\S~3]{BrantnerMathewActa_toappear}, and the monad $\L\Sym_\ZZ$ was extended as a \emph{filtered monad} to the full category by Raksit in \cite[\S~4]{Rak20}. 

Thereafter, one defines the $\infty$-category of \emph{derived rings} as $$\dalg_{\ZZ}:=\Mod_{\L\Sym_{\ZZ}}(\Mod_\ZZ)$$ where, by abuse of notation, we also denote the extension $\widetilde{\L\Sym_{\ZZ}}$ of $\L\Sym_\ZZ$ to the unbounded category by the same notation.

This theory has several deep applications and we refer the reader to \cite{Antieau2025FiltrationsCohomologyI}, \cite{MM24}, \cite{MondalReinecke2026UnipotentHomotopySchemes}, \cite{petrov_de_rham} and \cite{Rak20} among others.

In \cite[Definition 4.2.1]{Rak20}, the author defines a notion of \emph{derived algebraic contexts} which, in the $p$-local case, are $\Mod_{\ZZ_p}$-linear objects in $\Pr^L$, the $\infty$-category of presentable $\infty$-categories, satisfying some additional conditions like being symmetric monoidal, having a right complete $t$-structure which is compatible with filtered colimits, the connective part being compactly generated by a subcategory of the heart, etc and shows that you can derive the filtered monad $\L\Sym$ on a derived algebraic context and you get a good working theory of non-connective animated rings in that context.

\begin{definition}\label{def: definition of pcomp in prl}
Let $\Mod_{\ZZ_p}^\pcomp$ denote the $\icat$ of $p$-complete $\ZZ_p$- modules. An $\icat$ $\cC\in \Pr^L$ is called $p$-complete if it admits a structure of a $\Mod_{\ZZ_p}^\pcomp$-module in $\Pr^L.$\footnote{This terse definition is a consequence of the idempotence of $\Mod_{\ZZ_p}^\pcomp$ as a $\Mod_{\ZZ_p}$-algebra in $\Mod_{\Mod_{\ZZ_p}}(\Pr^L$). In other words, membership in $(\Pr^L)^\pcomp$ is a property not a structure. See \cite[Appendix A]{AntieauKrauseNikolaus2023PrismaticDeltaRings}.} We write $(\Pr^L)^\pcomp$ for the $\infty$-category of $p$-complete presentable $\icats$.
\end{definition}

We now give an example of how $p$-complete $\icats$ arise in nature. 

\begin{example}[$p$-complete $\icats$ arise as categories of quasi-coherent sheaves on stacks.]\label{eg: pcomp cats as qcoh on stacks}
Let $X$ be a stack i.e. an fpqc sheaf of spaces on $\aring.$ We will say a stack is a $p$-adic formal stack if the structure morphism $X\to \spec(\ZZ)$ factors over $\spf(\ZZ_p)\to \spec(\ZZ).$ This induces a canonical action of $\Qcoh(\spf(\ZZ_p))=\Mod_{\ZZ_p}^{\pcomp}$ on $\Qcoh(X)$, whence $\Qcoh(X)\in (\Pr^L)^\pcomp.$ 
\end{example}

One of our underlying goals in this paper is to understand how the the theory of derived rings interacts with objects in $(\Pr^L)^\pcomp$ which arise as above. 

However one hits an obstruction even when the formal stack is affine, as shown by the example of  $\spf(\ZZ_p)$, the final object in the $\icat$ of $p$-adic formal stacks.
Indeed, $\Mod_{\ZZ_p}^\pcomp=\Qcoh(\spf(\ZZ_p))$, the unit in $(\Pr^L)^\pcomp$ is not a derived algebraic context as its $t$-structure is not compatible with filtered colimits in the sense of \cite[Definition 1.3.5.20]{HA} i.e. the coconnective part is not stable under filtered colimits.

\begin{example}\label{eg: failure of pcomp being context}
The $t$-structure on $\Mod_{\ZZ_p}^\pcomp$ is pulled back along the inclusion $\Mod_{\ZZ_p}^\pcomp\into \Mod_{\ZZ_p}$ i.e. a module is (co)connective if it is as a $\ZZ_p$-module. We will show that colimits of coconnective modules can become non-trivially connective.
Recall that colimits in $\Mod_{\ZZ_p}^\pcomp$ are computed in $\Mod_{\ZZ_p}$ and completed. By considering the cofiber 
$$\ZZ_p\to \QQ_p\to \QQ_p/\ZZ_p$$ we see that $(\colim_n \ZZ/p^n)^{\wedge}_p=(\QQ_p/\ZZ_p)^{\wedge}_p=\ZZ_p[1]$ which is not coconnective.\footnote{One can try to repair this by working with the \emph{torsion} $t$-structure. We do not know if this gives an equivalent notion.}
\end{example}

While Example \ref{eg: failure of pcomp being context} shows that it is not possible to intrinsically derive the $\L\Sym$ monad on $\Mod_{\ZZ_p}^{\pcomp}$ in the sense of Raksit, there are many instances in nature where $p$-complete derived algebras show up. 

\begin{example}
    Let $X/\FF_p$ be a smooth affine scheme. Then its crystalline cohomology $\Gamma_{\mathrm{cris}}(X/\ZZ_p)$ is a $p$-complete derived algebra in the sense that its underlying module lives in $\Mod_{\ZZ_p}^{\pcomp}$. The fact that it is a derived algebra can be seen because the cohomology can be computed using $p$-complete cosimplicial rings given by Cech-Alexander complex. See \cite[\S~2]{BhattDeJong2011CrystallineDeRham}\footnote{Note that by classical deformation theory, one may lift such a smooth affine scheme all the way to $\ZZ_p$, whence one may also recover the derived algebra structure by contemplating the $p$-complete de Rham complex of the lift, for example by the main result of \emph{loc.cit}.}.
\end{example}

Similar structures show up routinely in the theory of prismatic cohomology \cite{BS19}, which can be seen as a mixed characteristic lift of crystalline cohomology. Moreover, in the the theory of prismatization \cite{Drin20}, \cite{APC}, \cite{APCII} and \cite{Bhatt22} one has to contemplate not just individual $p$-complete derived algebras but in fact entire $p$-complete (in the sense of Definiton \ref{def: definition of pcomp in prl})  categories of coefficients. 

\begin{example}
    Let $X/\ZZ_p$ be a smooth $p$-adic formal scheme. Then there is a functorial assignment of a $p$-adic formal stack $X\mapsto X^\Prism$ so that $$\Gamma(X^\Prism, \cO_{X^\Prism})\simeq \Gamma_{\Prism}(X)$$ where the right hand side is the absolute prismatic cohomology of $X.$  In this setting, $\Qcoh(X^\Prism)\in (\Pr^L)^\pcomp$  and can be regarded as a category of coefficients for prismatic cohomology in the following sense: If $f\colon Y\to X$  is a smooth proper morphism then $f^\Prism_*\cO_{Y^\Prism}\in \Perf(X^\Prism)$ carries information about the absolute prismatic cohomology of $Y$ and, by derived base change, of the fibers of $f$. 
    
    The functor $f^\Prism_*$ carries derived algebras in quasi-coherent sheaves on $Y$ to derived algebras on $X$. In particular $f^\Prism_*\cO_{Y^\Prism}$ is a derived algebra on $X.$ We will show in \cite{Sah25Syn} that it is in fact fruitful to take this refined multiplicative structure on $f^\Prism_*\cO_{Y^\Prism}$ seriously.

    The $p$-adic formal stacks above are often formal not just in the $p$-adic direction, but also the \emph{prismatic direction}.\footnote{A perhaps simplistic viewpoint of the prismatic theory which is heuristically implicit in this note is the following. A prism is a pair $(A,I)$ of a (discrete) ring $A$ and an ideal $I$ with $I\in \Pic(A)$. One requires that $A$ be derived $(p,I)$-complete and treats the $I$-direction as a formal parameter, whence $A$ may be treated as a pro-nilpotent thickening of the $p$-complete ring $A/I$. Further one requires that $A$ carry a derived lift of the Frobenius on $A/p$, encapsulated in the notion of a $\delta$-ring.
The theory then provides a deformation of the (derived) de Rham cohomology functor $\dR_{-/(A/I)}$ to a commutative algebra $\Prism_{-/A}$ over $A$ along with a Frobenius coming from the $\delta$-structure. It is therefore this $I$-direction we treat as formal.} Indeed this happens when the structure morphism $X\to \spf(\ZZ_p)$ factors over $\spf(\cO_C)$, the ring of integers in a completed algebraic closure of $C$ of $\QQ_p$. This adds some complexity to their categories of quasi-coherent sheaves which can now be complete for two parameters (in this case the kernel of the canonical surjection $A_\inf(\cO_C)\to \cO_C$, where the source is Fontaine's famous period ring) and not just $p$-complete.
\end{example}

Motivated by the examples above, our  main goal now is to explain how the theory of derived algebras interacts with the formal completions in derived algebraic geometry.

\subsection{Results.}\label{sec: results}

.

To make the idea precise we fix some notation.

\begin{notation}\label{notation in intro.}
    For this subsection and the next, fix an animated ring $A$ and a finitely generated ideal $J\subset \pi_0(A).$ Let $\Mod_A^\jcomp$ be the $\icat$ of $J$-complete $A$-modules and let $\dalg_A$ denote the $\infty$-category of derived $A$-algebras, defined as $(\dalg_\ZZ)_{A/}.$
\end{notation}

In the sequel we will study how the notion of derived algebras interacts with $\icats$ of quasi-coherent sheaves on geometric objects over $A$ which are $J$-complete.

\subsubsection{$J$-complete derived algebras.}
Our first results correspond to the notion of $J$-complete derived algebras. See  \S~\ref{sec: jcomp setting} for more on the following thorem.

\begin{theorem}
    There is a localizing subcategory $\dalg_A^\jcomp\into \dalg_A$ of $J$-complete derived algebras with the following properties
    \begin{enumerate}
    \item $\dalg_A^\jcomp$ is a presentably symmetric monoidal $\icat$.
        \item The left adjoint is completion at the level of underlying modules.
        \item The forgetful functor $\dalg_A^\jcomp\to \Mod_A^\jcomp$ is monadic with left adjoint $(\L\Sym_A)\cj$, the $J$-adic completion of the monad $\L\Sym_A.$
        \item The forgetful functor $\dalg_A^\jcomp\to \Calg_A^\jcomp$ is symmetric monoidal, conservative and preserves all (co)limits.
        \item  When one right Kan extends $\dalg(-)$ on all prestacks, then $\dalg(\spf(A))=\dalg_A^\jcomp$.
        \item The assignment $B\mapsto \dalg^\jcomp_B$ on $A$-algebras is local for the $J$-completely flat topology.
    \end{enumerate}
\end{theorem}

To prove this we also give a proof of a folklore theorem which seems to not be available in the literature. Note that there is no completeness conditions on $A$, although the statement of course depends only on the $J$-completion of $A.$

\begin{theorem}\label{thm: intro spfA is modjcomp}
    There is a canonical equivalence of $\icats$ $\Qcoh(\spf(A))\simeq \Mod_A^\jcomp.$
\end{theorem}

See Remark \ref{remark: mod on spf ref} for some literature concerning the above theorem as far as the author is aware.

\subsubsection{Graded derived algebras and comonads.}

In \S~\ref{sec: graded dalg and comonads} we prove several comonadicity theorems for graded derived algebras over $A$, in both the $J$-complete and noncomplete setting. The results are generically of the following form (see Corollary \ref{cor: continuous comonad dalg}).

\begin{theorem}\label{thm: cts monadicity}
    Let $\Gr\dalg_A^\jcomp$ be the $\icat$ of graded derived algebras such that each graded piece is $J$-complete and let $\dalg_A^\jcomp$ be the $\icat$ described in the previous subsection. Then there's a `continuous' comonadic adjunction $$\pij\colon \Gr\dalg_A^\jcomp\rightleftarrows \dalg_A^\jcomp\colon \pi^*$$ and the `continuous' comonad $\pij\circ \pi^*$ can be identified with a `continuous' action of the completed bicommutative Hopf algebra $(-)\ctensor_A A[\ZZ].$
\end{theorem}

Here the functor $\pij\colon \Gr\Mod_A^\jcomp\to \Mod_A^\jcomp$ is a completed version of the compactly supported pushforward of local systems of spectra which shows up in the theory of parametrised spectra \cite{AndoBlumbergGepner2018ParametrizedThom}. 

In Remarks \ref{remark: remark on diagonalizable coactions} and \ref{remark: continuous coaction diagonalizable schemes} we explain how to generalize mutatis mutandis the result above to the case of a continuous coaction by the bicommutative Hopf algebra $\ZZ[M]$ where $M$ is an arbitrary abelian group.

\subsubsection{Derived algebras on geometric stacks.}

In this section we study the geometric theory of derived algebras. We give a sampling of the results proved in this section.

One of the first results we prove is the following

\begin{theorem}\label{main thm spfbgm}
     For $A$ as in Notation \ref{notation in intro.}, there is a canonical equivalence of $\icats$ $$\dalg(B\G_m\times \spf(A))\simeq \Gr\dalg^\jcomp_A$$ where $\spf(A)$ is taken in the $J$-adic direction. 
\end{theorem}

Note that Theorem \ref{main thm spfbgm} is not a formal consequence of the easier statement for $\Qcoh$. See Remark \ref{remark: grdalg is not formal}. Thus this is proved using the continuous comonadicity established in Theorem \ref{thm: cts monadicity}.

In Remark \ref{remark: classifying stacks diagonalizable group schemes}, we will explain how to prove Theorem \ref{main thm spfbgm} in the case one replaces $B\G_m$ with the classifying stack $B\DD(M)$ where $\DD(M)$ is the diagonalizable group scheme  associated to an arbitrary abelian group $M$, e.g. as explained in \cite[\S~2]{Tate1997FiniteFlatGroupSchemes}.

We will use this result along with the notion of quasi-affine morphisms in derived algebraic geometry, which we now recall.

Following \cite{GaitsgoryRozenblyum2017SAGI} we will say that a morphism $f\colon X\to Y$ of derived stacks is quasi-affine if for any animated ring $R$ with a morphism $\spec (R)\to Y$, the base change $X\times_Y \spec(R)$ is representable by a quasi-compact quasi-affine derived scheme (the latter condition meaning that it admits an open immersion into an affine scheme). Then our result (see Proposition \ref{cor: quasi-aff preserves dalg}) is the following:

\begin{theorem}\label{thm: intro quasiaffine embedding}
    Let $f\colon X\to Y$ be a quasi-affine morphism of derived stacks. Then there is an equivalence of $\icats$ $$\dalg(X)\simeq \dalg(Y)_{f_*\cO_X/}.$$
\end{theorem}

Similar techniques which go in the proof of Theorem \ref{thm: main theorem}, also allow us to deduce a (derived) base change theorem for derived algebras in Proposition \ref{prop: derived base change for algebras}.

We apply these result to quasi-affine stacks over $\bgm$ to obtain results of the following form. Recall that a prism $(A,I)$ is called orientable if the effective Cartier divisor $I\in \Pic(A)$ admits a global generator $d\in I$. See Example \ref{example main thm.}.

\begin{theorem}\label{thm: main theorem}
Let $(A,I)$ be an orientable bounded prism. Fix an orientation $d\in I$. Then there's an equivalence of $\icats$

    $$\fil\dalg^{(p,d)\mathrm{-comp}}_{\fil_{(d)}A}\simeq \dalg(\spf(A[u,t]/(ut-d))/\GG_m)$$ where the formal direction on the right is $(p,d)$-adic.
\end{theorem}

Note that when $(A,I)$ is a perfect prism, then there's a canonical isomorphism 
$$(A/I)^\nyg\simeq \spf(A[u,t]/(ut-d))/\GG_m $$ and since $\dalg$ is functorial in stacks we have 

\begin{corollary}
    There is an equivalence of $\icats$
  $$\fil\dalg^{(p,d)\mathrm{-comp}}_{\fil_{(d)}A}\simeq \dalg((A/I)^\nyg)$$
\end{corollary}

In \cite{Sah25Nyg}, for any bounded prism $(A,I)$ we will define the relative Nygaard filtered prismatization following the constructions implicit in \cite[\S~5]{Bhatt22}. We will show that the
relative Nygaard filtered prismatization provides a good theory of coefficients for Nygaard filtered prismatic cohomology. Theorem \ref{thm: main theorem} then plays the main role therein.

\begin{remark}
    From our perspective, the key issues solved in this note are how to work around the non-obvious functoriality of the notion of $\dalg$. As an example of what would be desirable, see Remark \ref{remark: base changes for contexts.}. 
\end{remark}

\subsection{Forthcoming work.}

The tools developed in this note are in service to the the authors' forthcoming work \cite{Sah25Nyg} and \cite{Sah25Syn}. We will explain the relation of our current paper with the ones cited above in  \S~\ref{sec: forthcoming work}.

\subsection{Related work.}
This work is closely related to several recent preprints with different authors, although the results are mostly orthogonal. We have tried our best to represent them herein, and all misconceptions and mistakes in representing the work lie with the current author.

In \cite[Definition 5.1]{Antieau2025FiltrationsCohomologyI} a generalization of the derived algebraic contexts of Raksit are described. Unfortunately they still require compatibility of the $t$-structure with filtered colimits. However, it seems those definitions can be used to simplify the proofs presented in \S~\ref{ssec: generalisation to graded algebras}.

A less precise version of Proposition \ref{theorem: graded dalg bgm} appears as \cite[Remark B.5]{LahotiManam2025CohomologyRingStacks}. The authors state that it follows from their Lemmas B$.2$ and B$.3$, but the current author was unable to reconstruct their reasoning. In any case, our version of Proposition \ref{theorem: graded dalg bgm} is more precise and compatible with \cite{Moulinos2021GeometryOfFiltrations} by construction.

The recent preprint \cite{IshizukaYoshikawa2026DerivedGradedModules}, studies results analogous to those presented in our \S~\ref{ssec: continuous comonad}. Their results are more general in the sense that they work with graded discrete rings which are complete for \emph{homogeneous} ideals. Moreover they work with comodules over the group ring of torsion free $\ZZ$-modules, encompassing examples like $\ZZ[1/p]$ which show up in their theory of graded perfectoid rings \cite{IshizukaYoshikawa2025GradedPerfectoidRings}. It should be possible to extend their results to animated rings using our techniques, but it is not clear to the author if there are any applications of such a theory. However, their paper was an inspiration to extend our classification to (formal) classifying stacks of diagonalizable group schemes.

After a preliminary draft of this work had been prepared, the author learned of the closely related but somewhat orthogonal work   \cite{GardnerHekking2025IdealsDerivedGeometries}. In that paper, the authors study a notion of ideals and completions in their `derived geometric contexts' of \cite{BenBassatHekking2025BlowUpsNormalBundles}, which (to the author's understanding) is a geometric generalization of the `derived algebraic contexts' of Raksit. The paper \cite{GardnerHekking2025IdealsDerivedGeometries}, in particular works with a notion of $\cC_{\geq 0}$-stacks, where $\cC$ is a derived algebraic context and $\cC_{\geq 0}$ is the connective part. They then study a notion of ideals and completeness in this setting and give several deep applications to their notion of \emph{derived geometries}. This is in contrast to our work, where we explicitly aim to set up a working theory of non-connective formal geometry with applications to prismatic cohomology. The main difference between their and our results is that they study quasi-coherent sheaves in this generality, which when specialized to the context of $\Mod_{\ZZ}$ studies completeness of modules over $\dalg_\ZZ$, while we study complete objects in $\dalg_{\ZZ}$ itself. We refer the reader to \S~$1$ of their paper for a more in depth explanation of their results and technique than the author is capable of giving.

\subsection{Organization of the note.}
The note is organized as follows. Fix an animated ring $A$ and an ideal $J\subset \pi_0(A)$. In \S~\ref{sec: jcomp setting} we study a notion of $J$-complete derived algebras over $A$. The main observation is that $\dalg_A^\jcomp$ is a localizaing subcategory of $\dalg_A$ and that when one defines $\spf(A)$ as a prestack on all animated rings, with the formal direction being $J$-adic, then $\dalg(\spf(A))=\dalg_A^\jcomp.$ In \S~\ref{sec: graded dalg and comonads} we study $J$-complete graded derived algebras over the animated ring $A$ and prove various results which show that the $\icat$ $\Gr\dalg_A$ and $\Gr\dalg_A^\jcomp$ are comonadic over $\dalg_A$ and $\dalg_A^\jcomp$ in various settings. 
In \S~\ref{sec: dalg on geomstacks} we study the notion of $\dalg$ on geometric (formal) stacks, by which we loosely mean those admitting a representable faithfully flat affine morphism from representable (formal) stacks. It is in this section where we apply the theory of the previous sections to prove the classification theorems.

\subsection{Notation and conventions.}\label{notations and conventions}
We will freely use the language of $\infty$-categories and higher algebra as developed in  \cite{HTT} and \cite{HA} respectively. Our notation and conventions will thus reflect this. 

\subsubsection{}Geometric functors like $f_*$ and $\Gamma(X,-)$ will always be derived. Similarly, the symmetric monoidal functor on any symmetric monoidal $\infty$-category will simply be denoted $\otimes$.

\subsubsection{}We denote by $\mathcal{S}$ the $\infty$-category of spaces, or equivalently $\infty$-groupoids, or equivalently animae. We denote by $\mathrm{Sp}$ the $\infty$-category of spectra, and by $\mathbf{S}$, the sphere spectrum .

\subsubsection{} \label{infty nonsense}For the descent results we will need $\catinf$, the $\infty$-category of all $\infty$-categories. We will also need $\mathrm{Pr}^L$ the $\infty$-catgory of all presentable $\infty$-categories with colimit preserving functors between them, and its opposite category $\Pr^R$ which has the same objects but limit preserving accessible functors. $\Pr^L$ is a symmetric monoidal $\infty$-category equipped with a Lurie tensor product $\cC\otimes \cD$ for any $\infty$-categories $\cC, \cD\in \mathrm{Pr}^L.$  Thus we will also contemplate $\Calg(\Pr^L)$, which is equivalently the $\icat$ of of presentably symmetric monoidal $\icats$.

\subsubsection{} $\aring$ denotes the $\infty$-category of animated rings and if $A$ is an animated ring then we denote $\aring_A$ the under category  $\aring_{A/}$. 

\subsubsection{}\label{forgetful} We denote by $\Calg_\ZZ$ denotes the $\infty$-category of $\einf$-algebras in $\Sp$. There is a functor $\Theta\colon \aring\to \Calg_\ZZ$ which sends an animated ring $A$ to its underlying $\einf$-ring $A^\circ.$ 

\subsubsection{}
The $\infty$-category of modules $\Mod_A$ is defined as the $\infty$-category $\Mod_{A^\circ.}$
The $\infty$-category $\Mod_A$ is symmetric monoidal and we write the tensor product as $\otimes_A$; in particular even when $A$ is discrete, the tensor product is always derived. $\Mod_A$ carries a natural $t$-structure compatible with $\otimes_A$ such that an $A$-module is (co)connective if and only if its image along the forgetful functor $\Mod_A\to \mathrm{Sp}$ is (co)connective. We denote by $(\Mod_A)_{\geq 0}$ (resp. $(\Mod_A)_{\leq 0}$) the connective (resp. coconnective) parts of $\Mod_A$. 

\subsubsection{} Let $\cC$ and $\cD$ be $\infty$-categories with $t$-structures. A functor $F\colon \cC\to \cD$ is \emph{left} $t$-exact (resp. \emph{right} $t$-exact) if it preserves coconnective (resp. connective) objects. It is $t$-exact if it is both left and right $t$-exact.

\subsubsection{}\label{quotients} Let $A$ be an animated ring and $f\in \pi_0A$. Then we write $A/\!\!/f$ for the tensor product $\ZZ\otimes_{\ZZ[x]} A$, where the latter is the algebra obtained by freely setting $f=0$ in $A$ \footnote{Note that this is not independent of the choice of generators. For example $A/\!\!/0=A\oplus A[1]\neq A$ unless $A$ itself is the zero ring.}. This is also equivalent as a module to $\cofib(A\xrightarrow{f} A).$ For a sequence of elements $(f_1,\ldots, f_r)$ we write $A/\!\!/(f_1,\ldots, f_r)$ for the tensor product $A \otimes_{\ZZ[x_1,\ldots,x_r]} \ZZ$. As a module this is equivalent to $\bigotimes_{i=1}^r \cofib(A\xrightarrow{f_i}A).$ For an $A$-module $M$, we denote by $M/\!\!/(f_1,\ldots, f_r)$ the tensor product $M\otimes_A A/\!\!/(f_1,\ldots, f_r).$

\subsubsection{}
On occasion it would be necessary to take quotients of discrete rings $A$ by sequences $f_1,\ldots, f_r\in A$. In this case we write $A/(f_1,\ldots, f_r)$ for the quotient. Similarly if $M$ is a discrete $A$-module and $(f_1,\ldots, f_r)$ is a sequence of elements of $A$ then the classical quotient is $M/(f_1,\ldots, f_r).$ Note that even if $A$ is discrete, $A/\!\!/(f_1,\ldots, f_r)\neq A/(f_1,\ldots, f_r)$ and the latter is the Koszul complex on the sequence $(f_1,\ldots, f_r).$

%\subsubsection{}\label{torsion} Let $A$ be an animated ring  and let $J\subset \pi_0(A)$ be a finitely generated ideal. A module $M\in \Mod_A$ is called $J$-torsion if $J$ acts locally nilpotently on $\pi_*(M).$ Let $\Mod_A^{J\text{-}\mathrm{tors}}$ denote the full subcategory of $\Mod_A$ spanned by $J$-torsion modules. This is a colocalizing subcategory of $\Mod_A$ and the inclusion $\Mod_A^{J\text{-}\mathrm{tors}}\into \Mod_A$ admits a right adjoint $\Gamma_J(-)\colon \Mod_A\to \Mod_A^{J\text{-}\mathrm{tors}}$ which we call local cohomology since it agrees with the classical notion. Note that this is a natural candidate for full subcategory of modules supported on the topological space $|\mathrm{Spec} (\pi_0(A)/J)|$.

\subsubsection{ }\label{complete} Let $A$ be an animated ring  and let $J\subset \pi_0(A)$ be a finitely generated ideal. We say a module $M\in \Mod_A$ is $J$-complete if after choosing generators $(f_1,\ldots, f_r)=J$, the canonical map $M\to \lim_n M/\!\!/(f_1^n,\ldots, f_r^n)$ is an isomorphism. This is independent of the choice of generators. We let  $\Mod_A^{J\text{-}\mathrm{comp}}$ the full subcategory $\Mod_A^{J\text{-}\mathrm{comp}}$ of $\Mod_A$ spanned by $J$-complete modules. This is a localizing subcategory of $\Mod_A$ and the left adjoint to the inclusion $\Mod_{A}^{J\text{-}\mathrm{comp}}\into \Mod_A$ is given by completion $(-)^\wedge_J\colon \Mod_A\to \Mod_{A}^{J\text{-}\mathrm{comp}}.$ %Note that this is another candidate for modules supported on $|\spec \pi_0(A)/J|$. 

%\subsubsection{}\label{mgm} The categories of \ref{torsion} and \ref{complete} are equivalent by May-Greenless-Matlis duality. The orthogonal complement of both these categories is given by $\Mod_A^{J\text{-}\mathrm{loc}}$, the full subcategory spanned by modules $M$ where each $x\in J$ acts invertibly on $\pi_*M.$

\subsubsection{}\label{flatness} Suppose $M\in \Mod_A$. Then recall from \cite[Definition 7.2.2.10]{HA} that  is called flat if $\pi_0M$ is a flat $\pi_0(A)$ module and the natural action map $\pi_*(A)\otimes_{\pi_0(A)}\pi_0(M)\to \pi_*(M)$ is an isomorphism of $\ZZ$-graded abelian groups (in particular, $M$ is connective). It is faithfully flat if it is flat and $\pi_0M$ is a faithfully flat $\pi_0A$-module.

\subsubsection{}\label{jcompflat}For $J\subset \pi_0(A)$ finitely generated, we say that a module $M\in \Mod_A$ is $J$-completely (faithfully) flat if after choosing generators $J=(f_1,\ldots, f_r)$ the module $M/\!\!/(f_1,\ldots, f_r)$ is a (faithfully) flat module over $A/\!\!/(f_1,\ldots, f_r)$. This is equivalent to $M\otimes_A(-)$ being $t$-exact on $\Mod_A^{J\text{-}\mathrm{tors}}$, the $\icat$ of $J$-torsion $A$ modules, for the induced $t$-structure from $\Mod_A$ and thus is independent of choice of generators.

\subsubsection{}
Faithfully flat maps over an animated ring $A$ form a Grothendieck topology on $(\aring_A)^\mathrm{op}$ which we call the fpqc topology and denote $(\aring_{A})_{\mathrm{fpqc}}.$ We will use that the assignment $A\mapsto \Mod_A$ is a (hyper)sheaf on $(\aring_{A})_{\mathrm{fpqc}}.$

\subsubsection{}
A prestack over an animated ring $A$ is simply an accessible functor $X\colon \aring_A\to \calS$. A stack is an accessible hypersheaf  $X\colon (\aring_A)_{\fpqc}\to \calS.$ While we will work in the generality of prestacks, the objects of interest will be genuine stacks.

\subsubsection{}\label{right kan qcoh.}
The category of quasi-coherent sheaves on a prestack $X\colon \aring_A\to \calS$ is defined by right Kan extension by the formula 

$$\Qcoh(X)=\lim_{\spec(R)\to X} \Mod_R$$ where the limit is over all animated rings with a map to $X.$ The $\icat$ $\Qcoh(X)$ carries a $t$-strucutre wherein a module is connective if and only if it is after pullback to $\Mod_R$ for all $R$ appearing in the limit diagram. The co-connective part is defined by orthogonality and does not yield an easy description except in specific cases. 

%\subsubsection{}\label{geometric stack t-structure.}
%We say that a stack $X$ is a geometric stack if there is a faithfully flat\footnote{We warn the reader that \emph{faithfully flat} here doesn't mean an effective epimorphism for the flat topology. A representable morphism of stacks $X\to Y$ is faithfully flat if, after base change from an animated ring $\spec(R)\to Y$, the induced map $\spec(R)\times_Y X\to \spec(R)$ is faithfully flat in the sense of usual derived geometry. } representable affine morphism $p\colon \spec(R)\to X$ from an animated ring $R.$ Then,it follows that a complex $\cF\in \Qcoh(X)$ is (co)connective if and only if its pullback $p^*\cF\in \Mod_R$ is (co)connective. This follows from item $(3)$ of \cite[Proposition 3.4.17]{Lurie2011DAGVIII} and with the same proof.

\subsubsection{}\label{set theory} We state our conventions on set theory. In the main body of the paper the only point at which we need to contend with set theory is when computing limits for right Kan extensions as in the formula in \ref{right kan qcoh.}. However this will not be an issue as in our cases of interest, the limits will always be computed on a small category indexed by the non-negative integers. On the other hand, in Appendix \ref{appendix: lsym on prestacks}, we need to work with arbitrary prestacks and in that case we will work with prestacks valued in the $\infty$-category $\widehat{\calS}$ of large spaces, defined as in \cite[Remark 1.2.16.4]{HTT}.

\subsection{Acknowledgements.}

We thank Benjamin Antieau, Ryo Ishizuka, Aise Johan de Jong, Dmitry Kubrak, Akhil Mathew, Tasos Moulinos, Joshua Mundinger, Sasha Petrov and Chris Xu for useful correspondences or conversations.
In particular, Dmitry Kubrak explained one of the results in their forthcoming work with Shizhang Li \cite{KubrakLiKGn_inprep}, which is then used in our Appendix \ref{appendix: lsym on prestacks}, also mostly explained to us by Kubrak, as well as gave some useful comments on an earlier draft.
We also thank Maxwell Johnson for several conversations about $\icats$ over the years, which gave the author the confidence to pursue the topics herein.

As always, we are very grateful to our advisor Kiran Kedlaya for support, encouragement and the freedom to pursue our interests without which this paper would not exist in any form. The author was partially supported by NSF grant DMS-$2401536$
and from S.E. Warschawski Professorship under Kiran Kedlaya.

\section{Derived algebras on formal spectra.}\label{sec: jcomp setting}

\begin{notation}\label{notation: jcomp}
    For the rest of this section, fix an animated ring $A$ and $J\subset \pi_0(A)$ a finitely generated ideal.
\end{notation}
\begin{goal} \label{goal: purpose of dalgj}
 We are interested in defining an $\infty$-category $\dalg_A^\jcomp$ of $J$-complete derived $A$-algebras with the following properties

    \begin{enumerate}
\item $\dalg_A^\jcomp$ is the full subcategory of $\dalg_A$ consisting of derived algebras whose underlying modules are $J$-complete. 
    \item $\dalg_A^\jcomp$ is presentably symmetric monoidal and the inclusion functor $\dalg_A^\jcomp\into \dalg_A$ is a localization and the right adjoint $(-)^\wedge_J\colon \dalg_A\to \dalg_A^\jcomp$ is given by completion at the level of modules in the sense of \ref{complete}.
    \item The forgetful functor $\dalg_A^\jcomp\to \Mod_A^\jcomp$ is monadic with left adjoint $(-)^\wedge_{J}\circ \L\Sym_A$. 
    \item The forgetful functor $\dalg_A^\jcomp\to \Calg_A^\jcomp$ preserves all colimits.
    \item When we consider $\spf(A)$, defined as the subfunctor of $\spec(A)$, sending $B\in \aring$ to the subspace of $\Maps(A,B)$ annihilating some power of $J\subset \pi_0(A)$, then $\dalg(\spf(A))$, defined via the general procedure of right Kan extensions, agrees with $\dalg_A^\jcomp$.
    \item On the category of animated $A$-algebras, the functor $B\mapsto \dalg_B^\jcomp$ satisfies $J$-completely flat descent in the sense of \ref{jcompflat}.
\end{enumerate}
\end{goal}

In the next few subsections we will show that the naive definition of $\dalg_A^\jcomp$ as those derived algebras which satisfy property $(1)$ above accomplishes all the goals highlighted in Goal \ref{goal: purpose of dalgj}. We prove the first $5$ in this section and the last one is established in the next one.

\subsection{$J$-complete derived algebras.}

\begin{definition}\label{def: dalgj 1}
    We define $\dalg_A^\jcomp$ as the full subcategory of $\dalg_A$ consisting of derived algebras where the underlying functor $U\colon \dalg_A\to \Mod_A$ admits a factorisation $\dalg_A\xrightarrow{U^J}\Mod_A^\jcomp\to \Mod_A$ where the last functor is the canonical inclusion of $J$-complete $A$-modules.
\end{definition}

We begin by giving another definition of $\dalg_A^\jcomp$ which makes the functorialities of the above constructions apparent.

\begin{lemma}\label{lem: pullback def dalgj}
    There is a canonical pullback diagram in $\Pr^R$ the $\infty$-category of presentable $\infty$-categories along with right adjoint functors
    \[\begin{tikzcd}[cramped]
	{\dalg_A^{\jcomp}} & {\dalg_{A}} \\
	{\Mod_A^{\jcomp}} & {\Mod_A}
	\arrow[from=1-1, to=1-2]
	\arrow[from=1-1, to=2-1]
	\arrow[from=1-2, to=2-2]
	\arrow[from=2-1, to=2-2]
\end{tikzcd}\] where the bottom horizontal arrow is the inclusion $\Mod_A^{\jcomp}\into \Mod_A$ and the right vertical arrow is $U\colon \dalg_A\to \Mod_A.$
\end{lemma}
\begin{proof}
     Let us write $\cC$ for the $\icat$ of Definition \ref{def: dalgj 1} and $\cC'$ for the pullback in the lemma. Note that the functor $\cC'\to \dalg_A$ is full faithful being pulled back along a fully faithful functor $\Mod_A^\jcomp\to \Mod_A$. 
    Therefore the canonical functor $\cC\to \cC'$ induced by the universal property of $\cC'$ is fully faithful as the composite $\cC\to \cC'\to \dalg_A$ is fully faithful. It remains to check that the functor is essentially surjective. To show this it remains to check that a pair $(R, M, \sigma)$ consisting of $R\in \dalg_A, M\in \Mod_A^\jcomp$ and an isomorphism $\sigma\colon M\to R$ in $\Mod_A$ comes from $\cC.$ But $\cC$ consists of derived algebras whose underlying modules are $J$-complete, thus upto isomorphism we may pick $(R, U(R), \id_R)$ where $U\colon \dalg_A\to \Mod_A$ is the forgetful functor.
\end{proof}

\begin{proposition}\label{prop: dalgj is presentable}
    The $\infty$-category $\dalg_A^\jcomp$ satisfies the following properties
    \begin{enumerate}
        \item $\dalg_A^\jcomp$ is a presentable $\icat$ and therefore admits all limits and colimits. 
        \item The inclusion $\dalg_A^\jcomp\into \dalg_A$ is a fully faithful right adjoint, and therefore admits a left adjoint $L\colon \dalg_A\to \dalg_A^\jcomp.$ In other words it is a localizing subcategory of $\dalg_A^\jcomp.$
        \item $\dalg_A^\jcomp$ inherits a symmetric monoidal structure from $\dalg_A$ given by $L(R\otimes_A S)$ and so that $L\colon \dalg_A\to \dalg_A^\jcomp$ is symmetric monoidal.
    \end{enumerate}
\end{proposition}
\begin{proof}
    \begin{enumerate}
        \item It follows from Lemma \ref{lem: pullback def dalgj} that $\dalg_A^\jcomp$ is presentable as the pullback diagram is taken in $\Pr^R$ and can be computed after forgetting to $\catinf.$ Since a presentable $\icat$ has all limits and colimits, the second conclusion follows.
        \item The first claim can be checked in two ways, either from Definition \ref{def: dalgj 1} or from the pullback diagram in Lemma \ref{lem: pullback def dalgj}. In the latter case note that pullbacks of fully faithful functors remain fully faithful. The inlcusion is a right adjoint since the pullback diagram is taken in $\Pr^R$ and therefore admits a left adjoint $L\colon \dalg_A\to \dalg_A^\jcomp.$

        \item Recall from \cite[Remark 5.2.7.5]{HTT} that colimits in $\dalg_A^\jcomp$ can be computed in $\dalg_A$ and then completed. This allows us to induce a cocartesian monoidal structure of \cite[Definition 2.4.0.1]{HA} on $\dalg_A^\jcomp$, which by the above is given by $L(R\otimes_A S).$ The part about symmetric monoidality follows from the fact that for any derived $A$-algebra $R$ the functor $R\to L(R)$ is an $J$-complete equivalence at the level of module. In particular the fiber (as modules) is $J$-local. But as explained in \cite[Remark 4.2.6]{Lur11}, if $M$ is an $J$-local $A$-module then $N\otimes_A M $ is $J$-local. it follows that the canonical morphism of derived algebras $L(R\otimes_A S)\to L(L(R)\otimes_A L(S))$ is an equivalence.
    \end{enumerate}
\end{proof}

Our goal now is to prove that the right adjoint $L\colon \dalg_A\to \dalg_A^\jcomp$ can be computed at the level of modules.

\begin{construction}\label{cons: completion def}[A formula for completion.]
Fix generators $J=(f_1,\ldots, f_r).$

For each $n>0$ consider the derived $A$-algebra $$A_n:=A\qq(f_1^n,\ldots, f_r^n)$$ of \ref{quotients}.

Now let $R$ be a derived $A$-algebra. We will define a $J$-complete derived $A$-algebra $R^\wedge_J$ as follows. 

Since tensor products of derived $A$-algebras remain derived $A$-algebras we set $$R_n=R\otimes_A A_n.$$
Then set
\begin{equation}\label{eq: completion def}
    R^\wedge_J=\llim_n R_n.
\end{equation}

Since $\dalg_A$ has all limits, and they are computed at the level of modules, we see that $R^\wedge_J$ is a derived algebra. Further from Definition \ref{def: dalgj 1} we see that $R^\wedge_J$ is clearly $J$-complete, for example from \ref{complete}.

There is a canonical morphism of derived $A$-algebras
\begin{equation}\label{eq: can map to for completion}
  R\to  R^\wedge_J  
\end{equation}

\end{construction}

We next show that the procedure in Construction \ref{cons: completion def} provides an explicit formula for the completion.

\begin{lemma}\label{lem: adjoint is completion}
    The canonical morphism $R\to R^\wedge_J$ of Equation \ref{eq: can map to for completion} exhibits the target as a left adjoint to the inclusion $\dalg_A^\jcomp\to \dalg_A.$
\end{lemma}
\begin{proof}
       Note that each ring $A_n$ in Construction \ref{cons: completion def} is $J$-nilpotent and therefore $J$-complete. A similar statement holds for $R\otimes_A A_n.$ Thus the algebras $R\otimes_A A_n$ are all $J$-complete. Now limits of $J$-complete derived algebras remain $J$-complete as explained in the proof of Proposition \ref{prop: dalgj is presentable}. Thus we see that $R^\wedge_J$ is a $J$-complete derived algebra. 

    Now as explained in the proof of \cite[Proposition 5.2.7.4]{HTT} we may check that for any $J$-complete derived algebra $R$, the canonical map $R\to R^\wedge_J$ is an isomorphism. For this we may forget to $\Mod_A$ in which case we are reduced to checking the following fact: Let $M$ be in $\Mod_A^\jcomp$, then the canonical morphism $M\to \llim_n M/\!\!/(f_1^n,\ldots, f_r^n)$ is an isomorphism. But this is well known and may be deduced, for example, from \cite[Lemma 3.4.12]{BhattScholze2015ProEtale}, after specializing to the punctual $\infty$-topos $\mathcal{S}$ of spaces and noting that the argument in \emph{loc.cit.} works formally for animated rings.
\end{proof}

\begin{notation}
    In light of Lemma \ref{lem: adjoint is completion}, we will henceforth denote the left adjoint $L\colon \dalg_A\to \dalg_A^\jcomp$ by $(-)^\wedge_J$ \emph{with the understanding} that the universal property of left adjoints shows that it is independent of the set of generators.
\end{notation}

We thank Ben Antieau for suggesting the statement of the next corollary.

\begin{corollary}\label{cor: dalgj is right adjointable}
 The diagram in Lemma \ref{lem: pullback def dalgj} is left adjointable. More precisely, the following diagram 
\[\begin{tikzcd}[cramped]
	{\dalg_A^\jcomp} & {\dalg_A} \\
	{\Mod_A^\jcomp} & {\Mod_A}
	\arrow["{U^J}"', from=1-1, to=2-1]
	\arrow["{(-)^\wedge_J}"', from=1-2, to=1-1]
	\arrow["U", from=1-2, to=2-2]
	\arrow["{(-)^\wedge_J}"', from=2-2, to=2-1]
\end{tikzcd}\]

canonically commutes.
\end{corollary}

\begin{proof}
    This follows immediately from Lemma \ref{lem: adjoint is completion} by examining the explicit formula for completion.
\end{proof}

\begin{corollary}\label{cor: dalgj to modj is symm monoidal}
    The functor $U^J\colon \dalg_A^\jcomp\to \Mod_A^\jcomp$ is symmetric monoidal.
\end{corollary}
\begin{proof}
    This follows again from Corollary \ref{cor: dalgj is right adjointable}, along with point $(3)$ in Proposition \ref{cor: dalgj to modj is symm monoidal} and the fact that $U\colon \dalg_A\to \Mod_A$ is symmetric monoidal.
\end{proof}

\begin{notation}
    In light of Corollary \ref{cor: dalgj to modj is symm monoidal} we will now denote the tensor product of two derived $A$-algebras $R$ and $S$ by $R\widehat{\otimes}_A S$ as no confusion can arise. 
\end{notation}

\begin{construction}\label{cons: functor to calg}
    Note that the $\icat$ of $J$-complete $\einf$-algebras over $A$ also sits in the pullback diagram in $\Pr^R$
\[\begin{tikzcd}[cramped]
	{\Calg_A^\jcomp} & {\Calg_A} \\
	{\Mod_A^\jcomp} & {\Mod_A}
	\arrow[from=1-1, to=1-2]
	\arrow[from=1-1, to=2-1]
	\arrow[from=1-2, to=2-2]
	\arrow[from=2-1, to=2-2]
.\end{tikzcd}\] 
Note that this gives $\Calg_A^\jcomp$ a cocartesian monoidal structure admitted by $J$-completing the usual tensor product in $\Calg_A.$
Using from point $(c)$ of \cite[Notation 4.2.28]{Rak20}, that there is a forgetful `underlying' $\einf$-algebra functor $\Theta\colon \dalg_A\to \Calg_A$ which is symmetric monoidal, conservative and preserves all limits and colimits.

Thus, by the pullback description of $\dalg_A^\jcomp$ in Lemma \ref{lem: pullback def dalgj}, there is an evident forgetful functor $\Theta^J\colon \dalg_A^\jcomp\to \Calg_A^\jcomp$ fitting into a commutative diagram

\[\begin{tikzcd}[cramped]
	{\dalg_A^{\jcomp}} & {\dalg_A} \\
	{\Calg_A^\jcomp} & {\Calg_A}
	\arrow[from=1-1, to=1-2]
	\arrow["{\Theta^J}"', from=1-1, to=2-1]
	\arrow["\Theta", from=1-2, to=2-2]
	\arrow[from=2-1, to=2-2]
\end{tikzcd}\]
where the horizontal arrows the canonical inclusions.
\end{construction}

\begin{proposition}\label{prop: map to calgj is symm monoidal}
    The functor $\theta^J\colon \dalg_A^\jcomp\to \Calg_A^\jcomp$ is conservative, preserves all limits and colimits and is symmetric monoidal. 
\end{proposition}
\begin{proof}
It is clear that the functor $\theta^J$ is conservative. The functor $\Theta^J\colon \dalg^\jcomp\to \Calg^\jcomp$ is right adjoint as it is induced entirely in $\Pr^R$ and hence preserves all limits. It remains to show the case of colimits. There is an evident commutative diagram 

\[\begin{tikzcd}[cramped]
	{\dalg_A^{\jcomp}} & {\dalg_A} \\
	{\Calg_A^\jcomp} & {\Calg_A}
	\arrow["{\Theta^J}"', from=1-1, to=2-1]
	\arrow["{(-)^\wedge_J}"', from=1-2, to=1-1]
	\arrow["\Theta", from=1-2, to=2-2]
	\arrow["{(-)^\wedge_J}"', from=2-2, to=2-1]
\end{tikzcd}\]
obtained from the right adjointability of the second diagram in Construction \ref{cons: functor to calg}. This right adjointability can be seen from Corollary \ref{cor: dalgj is right adjointable} and because completions of $\einf$-algebras can be computed at the level of modules by \cite[Remark 4.2.6]{Lur11}. 

For the purpose of this proof let us denote the inclusion $\dalg_A^\jcomp\into \dalg_A$ by $i_{\dalg}$ and the inclusion $\Calg_A^\jcomp\into \Calg_A$ by $i_{\Calg}.$
Now let $F\colon K\to \dalg_A^\jcomp$ be a colimit diagram. 
Then we have that the colimit in $\dalg_A^\jcomp$ is computed by $(\colim_K(i_\dalg(F))^{\wedge}_J$. 
On the other hand note that $$\Theta(\colim_K (i_\dalg(F)))=\colim_K(\Theta\circ i_{\dalg}(F))=\colim_K(i_{\Calg}\Theta^J(F))$$ where the first equality comes because, as recalled in Construction \ref{cons: functor to calg}, $\Theta$ commutes with all colimits, and the second equality follows because of the commutativity of the second diagram in Construction \ref{cons: functor to calg}. 

Completing the rightmost expression we obtain $\colim_K(i_\Calg\Theta^J(F))^\wedge$ which is the colimit of the diagram $\Theta^J\circ F\colon K\to \Calg_A^\jcomp$. But this agrees with $\Theta^J((\colim_K(i_\dalg(F))^{\wedge}_J)$ by commutativity of the displayed diagram in this proof.

Lasty, the symmetric monoidality of $\Theta^J$ follows by the fact that it preserves colimits and both the source and target are equipped with their respective cocartesian monoidal structure.
\end{proof}

\begin{remark}
    Let $R\leftarrow S\rightarrow T$ be a span in $\dalg_A^\jcomp$. Then Proposition \ref{prop: map to calgj is symm monoidal} implies that the pushout can be computed as $R\ctensor_S T$ and agrees with the completed tensor product of the underlying $\einf$-algebras. 
\end{remark}

We can now establish functoriality of the assignment on animated $A$-algebras which sends $B\mapsto \dalg_B^\jcomp.$

\begin{remark}\label{remark: functoriality of dalgj}
     Let $B\to C$ be a map of animated $A$-algebras. Then the functor which sends $\dalg_B^\jcomp\to \dalg_C^\jcomp$ is given by $R\mapsto R\ctensor_B C.$ Indeed it suffices to observe that the completed tensor product of a $J$-complete $A$-algebra $R$ with another (possibly non-complete) $A$-algebra $T$ is $J$-complete. But this follows because $J$-completion is symmetric monoidal as explained in point $(3)$ of Proposition \ref{prop: dalgj is presentable}.
\end{remark}

We end by establishing the monadicity of the forgetful functor $U^J\colon \dalg_A^\jcomp \to \Mod_A^\jcomp.$ 

\begin{proposition}\label{prop: dalgj is monadic.}
    The forgetful functor $U^J\colon \dalg_A^\jcomp\to \Mod_A^\jcomp$ is monadic and the adjoint is given by $(\L\Sym_A)\cj:=(-)^\wedge_J\circ \L\Sym_A.$
\end{proposition}
\begin{proof}
    We will use the coarse Barr-Beck-Lurie criteria of \cite[Theorem 4.7.0.3]{HA}. First note that as already established $U^J\colon \dalg_A^\jcomp\to \Mod_A^\jcomp$ is a right adjoint between presentable $\icats$ and so it admits a left adjoint. We will discuss the left adjoint soon. But first we verify the other criteria in the monadicity theorem. Note that $U^J$ is conservative by construction. $\dalg_A^\jcomp$ is presentable hence admits geometric realizations. It suffices to show that $U^J$ preserves all geometric realizations. We will follow the same strategy as in the proof of \ref{prop: map to calgj is symm monoidal}.
    To this end denote by $i_\dalg$ the inclusion $\dalg_A^\jcomp\to \dalg_A$ and by $i_\Mod$ the inclusion $\Mod_A^\jcomp\to \Mod_A.$
Now let $X_\bullet\in \dalg^\jcomp$ be a simplicial object and $|X_\bullet|$ be its geometric realiszation. Then note first that $$U^J|X_\bullet|=U^J((|i_\dalg(X)^\bullet|)^\wedge_J)=(U|i_\dalg(X)^\bullet|)\cj,$$ where the first equality is by the definition of colimits in $\dalg_A^\jcomp$ and the second by the commutativity of the diagram in Corollary \ref{cor: dalgj is right adjointable} wherein $U^J\circ (-)\cj\simeq (-)\cj\circ U.$
Further note that $U\colon \dalg_A\to \Mod_A$ commutes with all sifted colimits by point $(c)$ of \cite[Notation 4.2.28]{Rak20}. Thus $$(U|i_\dalg(X)^\bullet|)\cj=|U\circ i_{\dalg}(X)^\bullet|\cj.$$
Now by commutativity of the defining diagram in Lemma \ref{lem: pullback def dalgj} we have $U\circ i_\dalg \simeq i_\Mod\circ U^J$ whence we obtain that $$|U\circ i_{\dalg}(X)^\bullet|\cj=|i_\Mod \circ U^J(X)|\cj.$$
But the right hand side is precisely the definition of $|U^J(X^\bullet)|$ as desired.

It remains to compute the monad which means for us to compute the left adjoint. To this end note that if $R$ is a $J$-complete derived $A$-algebra then we have the following equalities for any $M\in \Mod_A^\jcomp$

$$\Maps_{\Mod_A^\jcomp}(M,R)=\Maps_{\Mod_A}(M,R)=\Maps_{\dalg_A}(\L\Sym(M),R)$$ where the first equality is because $\Mod_A^\jcomp$ is fully faithful inside $\Mod_A$, the second one is because $\L\Sym\colon \Mod_A\to \dalg_A$ is a left adjoint. Now since completion is left adjoint we conclude that the righmost mapping space is equivalent to $\Maps_{\dalg_A^\jcomp}(\L\Sym(M)\cj,R)$ and therefore we have a canonical equivalence of mapping spaces
    $$\Maps_{\Mod_A^\jcomp}(M,R)\simeq\Maps_{\dalg_A^\jcomp}(\L\Sym(M)\cj,R).$$
\end{proof}

\begin{remark}\label{remark: base changes for contexts.}
Note that the notion of a \emph{derived algebraic context} due to Raksit \cite[Defintion 4.2.1]{Rak20} while very useful is somewhat strict of a notion. For example in the sequel we would like to rely on results presented in this section but for $\Gr\Mod_A$ and $\filmod_A$ respectively. It would be very desireable to clarify the theory as follows
\begin{enumerate}
    \item A derived algebraic context $\cC$ is $\Mod_\ZZ$-linear object in $\Pr^L$. It would be very useful to make the notion flexible enough to obtain functoriality for the base changes $\cC\otimes_{\Mod_\ZZ}\Mod_A.$
    \item Moreover it would be desireable to obtain the results in this section via base changes of the $\Mod_A$-linear category $\cC\otimes_{\Mod_\ZZ}\Mod_A$ along $\Mod_A\to \Mod_A^\jcomp$ where the map is given by completion. 
\end{enumerate}
Both of the results above would imply that one can do the completed theory presented here rather formally even in the case of completed graded and filtered $A$-modules, a notion we will introduce piecemeal in the sequel.
\end{remark}

\subsection{Derived algebras on $J$-adic formal spectra.}

We continue with the notation established in Notation \ref{notation: jcomp}.
Our goal is the following

\begin{goal}\label{goal: prestack dalg}
    Consider the prestack $\spf(A)\colon \aring_\fpqc \to \mathcal{S}$ which sends any animated ring $B$ to the summand of $\maps_{\aring}(A,B)$ consisting of maps which annihilate some power of $J\subset \pi_0(A)$. 
    Our goal is to show that when we define $\dalg(\spf(A))$ by right Kan extensions, then there is an equivalence of $\icats$
    $$\dalg(\spf(A))=\dalg_{A}^\jcomp.$$
\end{goal}

We will achieve Goal \ref{goal: prestack dalg} by first establishing the analgous statement for (the right Kan extension of) $\Qcoh(-)$. The result will then follow from conservativity of the forgetful functor.

\begin{remark}\label{remark: two cons of dalg are equivalent}
In the forthcoming Construction \ref{cons: right kan extension of dalg}, we define $\dalg(X)$ on a prestack $X$ simply by Kan extensions. 
However, one may consider the coCartesian fibration $\Mod\to \aring$ and note that  \cite[Remark 4.2.29]{Rak20} also implies that one has an endo-transformation $\L\Sym$ of the coCartesian fibration $\Mod\to \aring$ i.e. it preserves all coCartesian arrows. This gives a natural endofunctor $\L\Sym_X\in \End(\Qcoh(X))$ for any prestack $X$, which one can check by hand to be monadic (or by using more precise results of Kubrak-Li presented in Construction \ref{cons: Kubrakli}). Then one may contemplate the difference between $\Mod_{\L\Sym_X}(\Qcoh(X))$ and $\dalg(X).$ Indeed the definition via the global monad $\L\Sym_X$ is useful in calculations, as done in \cite[\S~2]{petrov_de_rham}, while that of Construction \ref{cons: right kan extension of dalg} is useful in proving results like our Corollary \ref{cor: dalg on spf}.
In fact we will show that both these constructions are equivalent in our Appendix \ref{appendix: lsym on prestacks}, using ideas of D. Kubrak.
\end{remark}

\begin{construction}[Derived algebras on prestacks]\label{cons: right kan extension of dalg}
    For any animated ring $R\in \aring$, the assignment $R\mapsto \dalg_R$ forms a $\catinf$-valued functor by \cite[Remark 4.2.29]{Rak20}. It follows by applying the $\icat$ Grothendieck construction in the form of Lurie's unstraightening \cite[Theorem 3.2.0.1]{HTT} that one gets a coCartesian fibration $\dalg\to \aring$. Using the dual of \cite[Proposition 2.7.6.]{Lurie2011DAGVIII} in the form of Remark $2.7.8$ of \emph{loc.cit.}, we may \emph{define} $\dalg(X)$ for any prestack $X$ on $\aring$ via the formula
    $$\dalg(X)=\lim_{\spec(R)\to X} \dalg_R.$$

    This definition has the following features. 

    \begin{enumerate}
        \item For any stack $X$, there is a forgetful functor $U_X\colon \dalg(X)\to \Qcoh(X)$ which is monadic by Lemma \ref{lem: qcoh to dalg is monadic}.
        \item The functor $U_X\colon \dalg(X)\to \Qcoh(X)$ factors over $\Calg(X):=\Calg(\Qcoh(X))$ is conservative and preserves all limits and colimits.
        \item Let $L_X\colon \Qcoh(X)\to \dalg(X)$ be the left adjoint. Then the monad $U_X\circ L_X$ on $\Qcoh(X)$ is given by the global monad $\L\Sym_X$ explained in Remark \ref{remark: two cons of dalg are equivalent} (or more precisely in Construction \ref{cons: Kubrakli}) by Proposition \ref{prop: kubrak prop}.
    \end{enumerate}

\end{construction}

We now need an analogue of \cite[Lemma 8.1.2.2]{Lur18} for animated rings complete along a finitely generated ideal. In the setting of \emph{derived} algebraic geometry this lemma can be substantially simplified.

\begin{lemma}[\label{lem: luries lemma on spf}]
Let $A$ be an animated ring and $J\subset \pi_0(A)$ a finitely generated ideal. Then there exists an $\mathbf{Z}_{>0}$-indexed tower $\ldots \to A_4\to A_3\to A_2\to A_1$ in the $\infty$-category $\aring_A$ with the following properties
    \begin{enumerate}
        \item The maps $A_{i+1}\to A_i$ induce a surjection on $\pi_0(A_{i+1})\to \pi_0(A_i)$ with nilpotent kernel.
        \item For every animated ring $B$, the canonical map $$\colim_n\maps_{\aring}(A_n,B)\to \maps_{\aring}(A,B) $$ induces a homotopy equivalence of the source with the summand consisting of those maps $A\to B$ which annihilate some power of $J$ i.e. with $\spf (A)$.
        \item Each of the animated rings $A_n$ are perfect when regarded as an $A$-module.
    \end{enumerate}
\end{lemma}
\begin{proof}
     As in the proof of \cite[Lemma 8.1.2.2]{Lur18} it suffices to construct for each $x\in \pi_0(A)$ a tower of rings $\ldots \to A_4(x)\to A_3(x)\to A_2(x)\to A_1(x)$ satisfying the properties in the lemma i.e. prove the lemma for $J=(x).$ Then we can conclude the case of arbitrary finitely generated ideals by choosing generators $J=(x_1,\ldots, x_r)$ and setting $A_n=A_n(x_1)\otimes_A A_n(x_2)\otimes_A \ldots \otimes_A A_n(x_r).$

    In fact, since we are in the animated setting, our proof is far simpler since quotients exist as described for example in \ref{quotients} and for $x\in \pi_0(A)$ we may simply take $A_n(x)=A/\!\!/x^n.$ It remains to check that the the listed properties in the statement of the lemma are satisfied. The properties $(1)$ and $(3)$ are clearly true. Thus it remains to verify property $(2).$ The argument here matches the argument in \emph{loc.cit.}

    First fix an $n>0$ and consider the canonical map $A\to A\qq x^n.$  Suppose $B$ is an animated ring and $\varphi\colon A\to B$ a morphism. Then the space of factorizations of $\varphi\colon A\to B$ through $A\to A\qq x^n$ is identified with the space of paths in $\Maps_{\ZZ}(\ZZ[w],B)\simeq B$ between $\varphi(x^n)$ and $0.$ Indeed this follows immediately from the universal property of $A\qq x^n$ as described in \ref{quotients}. It follows that the homotopy fiber $P^n$ of $\Maps_{\aring}(A_n,B)\to \Maps_{\aring}(A,B)$ over $\varphi$ is identified with $\Maps_{B}(\varphi(x^n),0)$ inside the $\infty$-category of the underlying space of $B.$ Note that then the homotopy fiber of $$\colim_n\Maps_{\aring}(A_n,B)\to \Maps_{\aring}(A,B)$$ is identified with $P_\infty:=\colim_n P_n$ where the transition map is given by multiplication by $\varphi(x).$
    If $\varphi(x)$ is not nilpotent in $B$ then $P_n$ is empty for each $n$ and so there's nothing to prove. So assume that $P_n$ is non-empty for some $n$. Then we observe that $\pi_m(P_n)=\pi_{m+1}(B).$ Indeed this follows from noting that $\Maps_{B}(\varphi(x^n),0)$ is a torsor over $\Maps_{B}(0,0)=\Omega_{0}B$ where the right hand side is the space of loops based at $0$.
    So now we have $$\pi_m(P_\infty)=\colim_n\pi_m(P_n)=\colim( \pi_{m+1}(B)\xrightarrow{\varphi(x)} \pi_{m+1}(B)\xrightarrow{\varphi(x)}\pi_{m+1}(B)\to \ldots)$$ but the right hand side is $0$ since $\varphi(x)$ is nilpotent.
\end{proof}

\begin{remark}\label{remark: A_n for animated choose generators}
    In light of \emph{the proof of Lemma \ref{lem: luries lemma on spf}}, whenever we choose a tower $\ldots \to A_3\to A_2\to A_1$ for the animated ring $A$, we have \emph{implicitly chosen a set of generators for $J$}.
\end{remark}

\begin{corollary}\label{cor: limit of mod on spf}
    Let $A$ an animated ring and $J\subset \pi_0(A)$ be finitely generated then $\Qcoh(\spf A)=\llim_n \Mod_{A_n}$ and $\dalg(\spf A)=\llim_n \dalg_{A_n}$ where the $A_n$ are as constructed in Lemma \ref{lem: luries lemma on spf}.
\end{corollary}
\begin{proof}
   Keeping in mind Remark \ref{remark: A_n for animated choose generators}, both of these follow from noting that the category defined by the tower $\spec(A_1)\to \spec (A_2)\to \spec(A_3)\to \ldots$ are final in the indexing $\icat$ $\daff_X$ by Lemma \ref{lem: luries lemma on spf} and so the limit of Construction \ref{cons: right kan extension of dalg} and \ref{right kan qcoh.} can be computed on this category.
\end{proof}

\begin{construction}\label{cons: can map from spf to spec}
Note that by the explicit description of $\spf(A)$ as a prestack on $\aring_\fpqc$ there is a canonical morphism of prestacks
$$\eta_A\colon \spf(A)\to \spec(A).$$

This induces a pullback morphism 
$$(\eta_A^*)_{\Mod}\colon \Qcoh(\spec(A))\simeq \Mod_A\to \Qcoh(\spf(A))$$
and similarly 

$$(\eta_A^*)_{\dalg}\colon \dalg(\spec(A))\simeq \dalg_A\to \dalg(\spf(A)).$$

By construction, it follows that the following diagram commutes $$% https://tikzcd.yichuanshen.de/#N4Igdg9gJgpgziAXAbVABwnAlgFyxMJZABgBpiBdUkANwEMAbAVxiRAB12pGBzAfQCCIAL6l0mXPkIoAjOSq1GLNp24MeACk5w0AMw0CAlIZFiQGbHgJEyMhfWatEHdgFlog0+MtSicu9QOys6qvFrsOjAAxgbGIgowUDzwRKC6AE4QALZIZCA4EEhyio5s4TA4dIIAegBUhnyh6l4gGdm51AVIAEyBSk4gAKqe1Ax0AEYwDAAKElbSIOlYPAAWOC1tOYjFXYgAzH2lzsPA2nqxwhuZW735hfujE1OzPtbOS6vrh8Eg5ZU19Uabmg8WEQA
\begin{tikzcd}
\dalg_A \arrow[r, "(\eta_A^*)_\dalg"] \arrow[d, "U_A"'] & \dalg(\spf(A)) \arrow[d, "U_{\spf(A)}"] \\
\Mod_A \arrow[r, "(\eta_A^*)_\Mod"']                    & \dalg(\spec(A))                        
\end{tikzcd}$$
    
\end{construction}

Using Construction \ref{cons: can map from spf to spec} we now prove an analogue of \cite[Theorem 8.3.4.4]{Lur18} in the setting of formal derived algebraic geometry.

\begin{remark}\label{remark: mod on spf ref}
    As stated in \cite[Remark 8.3.4.5]{Lur18}, in this setting we do not need a connectivity hypothesis while working in derived algebraic geometry. Compare also \cite[Footnote 12]{BS19}. Note that this result appears in a restricted setting in \cite[\href{https:/\!\!/stacks.math.columbia.edu/tag/0H0L}{Tag 0H0L}]{stacks-project}\footnote{Note that the Noetherian hypothesis in the referenced tag is inessential and can be replaced by a notion of ideals generated by weakly proregular sequences. See \cite[\href{https:/\!\!/stacks.math.columbia.edu/tag/0H1E}{Tag 0H1E}]{stacks-project}, with the point being that for a weakly proregular sequence in any discrete ring $A$, the pro-systems $\{A/\!\!/(x_1^n,\ldots, x_r^n)\}$ and $\{A/(x^n_1,\ldots, x^n_r)\}$ are pro-isomorphic.} and \cite[Appendix A.]{AntieauKrauseNikolaus2023PrismaticDeltaRings} in the classical setting of bounded $p$-torsion rings, and should certainly be well known over animated rings. However, we do not know a reference in this generality, although note Remark \ref{remark: mundinger}.
    \end{remark}

\begin{theorem}\label{thm: mod on spf}
    The pullback $(\eta_A^*)_{\Mod}\colon \Qcoh(\spec(A))\simeq \Mod_A\to  \Qcoh(\spf(A))$ restricts to an equivalence on $\Mod_A^\jcomp$
\end{theorem}
\begin{proof}
    Choose a set of generators $J=(x_1,\ldots, x_r)$ and set $A_n=A/\!\!/(x_1^n,\ldots, x_r^n).$

    By Corollary \ref{cor: limit of mod on spf} we can identify $\Qcoh(\spf(A))\simeq \llim_n\Mod_{A_n}$ and the functor $(\eta_A^*)_{\Mod}$ may be identified with the functor $F\colon \Mod_A\to \llim_n \Mod_{A_n}$ sending $M\mapsto \{M\otimes_A A_n\}$. This functor admits a right adjoint $G\colon \llim_n \Mod_{A_n}\to \Mod_A$ sending $\{M_n\}\mapsto \llim_n M_n.$ Since each of the $M_n$ are modules over the $J$-nilpotent animated rings $A_n$  it follows that $\llim_n M_n$ is a $J$-complete module $A$ and lands in $\Mod_A^\jcomp.$
    
    By abuse of notation we denote the restriction of $F$ to $\Mod_A^\jcomp$ also by $F$ and so we have an adjunction 
    $$F \colon \Mod_A^\jcomp \rightleftarrows \llim_n\Mod_{A_n}\colon G .$$
    
    It follows that the unit $\eta_M \colon M\to G\circ F(M)=\llim_n M\otimes_A A_n$ is an isomorphism for all $J$-complete modules $M$. Thus we have a fully faithful embedding $F\colon \Mod_A\to \llim_n \Mod_{A_n}.$  For formal reasons,\footnote{Let $F\colon \cC\to \cD$ be a functor of $\icats$ with right adjoint $G\colon  \cD\to \cC$. If $F$ is fully faithful, and $G$ is conservative, then $F$ is an equivalence of $\infty$-categories. Indeed full faithfulness of $F$ gives that $\eta\colon \id_C\to  G\circ F$ is an equivalence. Now it suffices to show that the counit $\varepsilon\colon F\circ G\to \id_D$ is an equivalence. By the triangle identity we have that $G(\varepsilon)\circ \eta_G=\id_G$ and so since $\eta_G$ is an equivalence, we learn that $G(\varepsilon)$ is an equivalence. By conservativity of $G$ we conclude that $\varepsilon$ is an equivalence. A similar proof holds when $G$ is fully faithful and $F$ is conservative.} it suffices to show that $G$ is conservative.
    
    Since we are in a stable $\icat$ it suffices to show that if $G(\{M_n\})=\llim_n M_n=0$ then $M_n=0$ for all $n>0.$ This is a consequence of Lemma \ref{lem: BS lemma for arbitrary rings}. 
\end{proof}

In Lemma \ref{lem: BS lemma for arbitrary rings}, we will prove an animated analogue of Lemma \cite[3.5.4]{BhattScholze2015ProEtale}, proved there for discrete Noetherian rings in a classical topos. 
Before we do that we need a calculation of certain pro-systems defined by Koszul complexes.

\begin{lemma}\label{lem: pro koszul}
    In the situation of Theorem \ref{thm: mod on spf}, consider the pro-system given by $\{A_{n}\otimes_A A_1\}_{n>0}$ and the constant pro-system $\{A_1\}_{n>0}$. Then $\{A_{n}\otimes_A A_1\}_{n>0}$ is pro-isomorphic to $\{A_1\}_{n>0}$.
\end{lemma}
\begin{proof}
Note that for any $n>0$ we have $A_n=A\qq f_1^n,\ldots, f^n_r$ is obtained by freely setting the elements $f^n_i=0$ for $ 1\leq i\leq r$.

In particular if $M$ is any $A$-module, then the action of $x_i^n$ on $M\otimes_A A_n$ factors through $0.$
Since tensors are computed at the level of modules, we may also write the underlying as $\bigotimes_{i=1}^{i=r}\cofib(A\xrightarrow{x_i^n} A).$

Thus we compute 
\begin{align}\label{eq: formula for antimesa1}
A_n\otimes_A A_1
  &= (\bigotimes_{i=1}^{r}\cofib(A\xrightarrow{x_i^n} A))\otimes_A (\bigotimes_{i=1}^{r} A\qq x_i)\notag\\
  &= \bigotimes_{i=1}^{r}\cofib(A\qq x_i\xrightarrow{x_i^n} A\qq x_i)\notag\\
  &= \bigotimes_{i=1}^{r}(A\qq x_i\oplus A \qq x_i[1])\notag\\
  &= \bigoplus_{k=0}^{r} A_1^{\oplus \binom{r}{k} }[k].
\end{align}

Thus we learn that the terms in the tower $\{A_n\otimes_A A_1\}$ are independent of $n.$
Further from \ref{eq: formula for antimesa1} there is an obvious morphism of pro-systems $\{A_n\otimes_A A_1\}_{n>0}\to \{A_1\}_{n>0}$ corresponding to projection on the first factor and a section  $\{A_1\}_{n>0}\to \{A_n\otimes_A A_1\}_{n>0}.$ 

It remains to check that this induced morphism is actually a pro-isomorphism. For this it suffices to show that for any $n>1$ the morphism $A_n\otimes_A A_1\to A_{n-1}\otimes_A A_1$ is identity on the first factor of \ref{eq: formula for antimesa1} and $0$ on the other factors. This will prove the result as it is now clear that the morphism of pro-systems $\{A_n\otimes_A A_1\}_{n>0}\to \{A_1\}_{n>0}$ is an isomorphism. 

So now we prove the fact about the morphism being $\id$ on the first factor and $0$ on positive factors.

Note that the morphism $A_n\to A_{n-1}$ is obtained by base change along $\ZZ[x_1,\ldots, x_r]\xrightarrow{x_i\to f_i} A$ of the universal morphism 
\begin{equation}\label{eq: universal quotients}
  \ZZ[x_1,\ldots, x_r]/(x_1^n,\ldots, x_r^n)\to \ZZ[x_1,\ldots, x_r]/(x_1^{n-1},\ldots, x_r^{n-1})  
\end{equation}

The tensor product $A_n\otimes_A A_1\to A_{n-1}\otimes_A A_1$ is obtained by a further base change along $A\otimes_{\ZZ[x_1,\ldots, x_r]} \ZZ$. Thus to understand the morphism we may just base change \ref{eq: universal quotients} along $-\otimes_{\ZZ[x_1,\ldots, x_r]} \ZZ$.

Writing $$\ZZ[x_1,\ldots, x_r]/(x_1^n,\ldots, x_r^n)=\bigotimes_{i=1}^{i=r}\cofib(\ZZ[x_1,\ldots, x_r]\xrightarrow{x_i^n} \ZZ[x_1,\ldots, x_r])$$ and viewing the induced morphism of \ref{eq: universal quotients} by coming, after iterated tensoring, from the complex

$$% https://tikzcd.yichuanshen.de/#N4Igdg9gJgpgziAXAbVABwnAlgFyxMJZABgBpiBdUkANwEMAbAVxiRAB12AtL5ADwD6ARlKcGUCDjikABIIBOFEAF9S6TLnyEUIytXrNWiDt16CRYiVNkKlq9djwEiIoVVqMWbTj37DR7OKS0nICiipqIBiOWkRkbvqeRia+5gFB1qHhyu4wUADm8ESgAGbyEAC2SGQgOBBIIh6GbIJYAHqE9iBllQ3UdUgATInNxpxYUBGl5VWINQOIAMwjXsatU90zSMu19YjDTasgrW3AYAC0QsoqFMpAA
\begin{tikzcd}
{\ZZ[x_1,\ldots, x_r]} \arrow[r, "x_i^n"] \arrow[d, "x_i"] & {\ZZ[x_1,\ldots, x_r]} \arrow[d, "\id"] \\
{\ZZ[x_1,\ldots, x_r]} \arrow[r, "x_i^{n-1}"]              & {\ZZ[x_1,\ldots, x_r]}                 
\end{tikzcd},$$ we see that the derivation of \ref{eq: formula for antimesa1}, gives that the induces morphism $$\bigoplus_{k=0}^{r}\ZZ^{\oplus \binom{r}{k}}[k]\to \bigoplus_{k=0}^{r}\ZZ^{\oplus \binom{r}{k}}[k]$$ is $\id$ on the first factor and zero on all factors with a postive degree. After base change to $A_1$ we conclude the same for the induced morphism  

$$\bigoplus_{k=0}^{r} A_1^{\oplus \binom{r}{k} }[k]\to \bigoplus_{k=0}^{r} A_1^{\oplus \binom{r}{k} }[k].$$

This concludes the proof.

\end{proof}

\begin{corollary}
    Let $k>0$, then the pro-systems $\{A_n\otimes_A A_k\}_{n>0}\to \{A_k\}_{n>0}$ are pro-isomorphic.
\end{corollary}
\begin{proof}
    We may look at the sub-pro-system of $\{A_n\}_{n>0}$ consisting of the tower $\{A_{mk}\}_{m>0}$. Then by Lemma \ref{lem: pro koszul}, the pro-systems $\{A_{mk}\otimes_A A_k\}_{m>0}\to \{A_k\}_{m>0}$ are pro-isomorphic. 
    Note that $\{A_{mk}\otimes_A A_k\}$ is final in the co-filtered system $\{A_{n}\otimes_A A_k\}_{n>0}$, thus we can conclude the result.
\end{proof}

\begin{lemma}\label{lem: BS lemma for arbitrary rings}
    In the setting of Theorem \ref{thm: mod on spf}, any tower $\{M_n\}\in \llim_n \Mod_{A_n}$ we have $G(\{M_n\})\otimes_A A_k=(\llim_n M_n)\otimes_A A_k=M_k.$
\end{lemma}
\begin{proof}
        Note that each $A_k$ is a perfect (hence dualizable) $A$-module by point $(3)$ of Lemma \ref{lem: luries lemma on spf}. 
    Thus we have that $(\llim_n M_n)\otimes_A A_k=\llim_n(M_n\otimes_A A_k).$ Therefore it suffices to verify that the pro-system $\{M_n\otimes_A A_k\}$ is pro-isomorphic to the constant system $\{M_k\}.$

    Now note that we have a change of rings pro-isomorphism $\{M_n\otimes_A A_k\}\simeq \{M_n\otimes_{A_n} A_n\otimes_A A_k\}.$

    We first assume that $k=1.$ In this case we observe that the explicit proof of lemma \ref{lem: pro koszul}, in particular \ref{eq: formula for antimesa1}, shows that for each $n>0$ the algebra $A_n\otimes_A A_1$ is equivalent to $A_1\oplus C$ where $C$ is $\binom{r}{k}-1$ shifted copies of $A_1$ and the number of shifts. Thus tensoring $M_n$ with $A_n\otimes_A A_1$ gives $M_n\otimes_{A_n} A_1\oplus C'$ where $C'$ is  $\binom{r}{k}-1$ shifted copies of $M_n\otimes_{A_n} A_1$.
    Note that, after this identification, the induced map in the tower $M_n\otimes_{A_n} A_n\otimes_A A_1\to M_{n-1}\otimes_{A_{n-1}} A_{n-1}\otimes_A A_1 $ is the one coming from $$M_n\otimes_{A_n}A_1\to M_{n-1}\otimes_{A_{n-1}} A_1$$ on the first factor and zero on the rest of the factors. 

    We may apply similar reasoning for the tower $\{M_n\otimes_{A_n}A_1\}_{n>0}.$
    Thus the natural projection map again is a pro-isomorphism. 

    Now suppose $k>1$. In this case we again look at the final sub-pro-system given by the indices $\{M_{mk}\otimes_{A_{mk}}A_k\}_{m>0}$ and similarly for $\{M_{mk}\otimes_{A_{mk}}A_k\}_{m>0}$ and repeat the reasoning.

    Therefore $\llim_n (M_n\otimes_A A_k)=\llim_n(M_n\otimes_{A_n} A_k)$ and the second limit is eventually constant with value $M_k.$ So we conclude.
\end{proof}

\begin{remark}\label{remark: pro-iso stacksproject}
      In \cite[\href{https://stacks.math.columbia.edu/tag/0GZI}{Tag 0GZI}]{stacks-project} a different proof of Lemma \ref{lem: BS lemma for arbitrary rings} is given. We are not sure if it applies in our situation as the models of pro-systems of \emph{loc.cit.} seems slightly different as it is constructed using flat modules over explicit DGAs..
\end{remark}

\begin{corollary}\label{cor: dalg on spf}
   The pullback $(\eta_A^*)_{\dalg}\colon \dalg(\spec(A))=\dalg_A\to \dalg(\spf(A))$ restricts to an on equivalence $\dalg_A^\jcomp.$  
\end{corollary}
\begin{proof}
    This follows exactly as in the proof of Theorem \ref{thm: mod on spf} using the conservativity of the forgetful functor $\dalg_A\to \Mod_A.$
\end{proof}

\begin{remark}\label{remark: mundinger}
After the first version of this paper had been prepared, J. Mundinger informed the author that the results of this section bear some similarity to those occuring in \cite[\S~7.1]{GR_DGIndschemes_ConM610_2014}.
\end{remark}

\begin{remark}\label{remark: functoriality of spf}
It follows from Theorem \ref{thm: mod on spf} and Corollary \ref{cor: dalg on spf} that if $f\colon \spf B\to \spf C$ is a morphism of formal schemes over $A$ induced where the formal direction is along $J$, induced from a morphism of rings $B\to C$. Then the induced functor $f^*\colon \Qcoh(\spf(C))\to \Qcoh(\spf B)$  can, after identifying the source and target with $\Mod_C^\jcomp$ and $\Mod_B^\jcomp$ respectively, be identified with sending $M\mapsto M\widehat{\otimes}_C B$ where $\widehat{\otimes}_C$ is the $J$-completed tensor product. An obviously similar statement holds for the induced morphism $f^*\colon \dalg(\spf(C))\to \dalg(\spf (B))$. Indeed the first can be deduced by the commutativity of the diagram $$
\begin{tikzcd}[cramped]
\Mod_C^{\jcomp} \arrow[d, "f^*"'] & \Mod_C \arrow[d, "f^*"] \arrow[l, "(-)^{\wedge}_J"'] \\
\Mod_B^\jcomp                     & \Mod_B \arrow[l, "(-)^\wedge_J"]                    
\end{tikzcd}$$ and similarly for the second. 
\end{remark}

\begin{remark}\label{remark: definition of formal stack}
    Using Lemma \ref{lem: luries lemma on spf}, we may define a $J$-adic formal $A$-prestack as follows. Let $X\colon \aring_A\to \calS$ be a prestack over $A$. Then we say that $X$ is a $J$-adic formal $A$-prestack if the structure morphism $X\to \spec(A)$ factors through $X\to \spf(A).$
    This definition now implies that by Lemma \ref{lem: luries lemma on spf}, we have $$X=\colim_n X\times_{\spec(A)} \spec(A_n),$$ for some tower $A_i$ satisfying the hypothesis of Lemma \ref{lem: luries lemma on spf}
    In particular we have $$\Qcoh(X)=\lim_n \Qcoh(X\times_{\spec(A)} \spec(A_n)),$$ which heuristically gives a notion of $J$-complete sheaves on $X$. 
    Indeed if $X$ is a stack over $A$-algebras then $X\times \spf(A)$ can be thought of as the $J$-completion of $X$ and 
    $$\Qcoh(X\times \spf(A))=\lim_n \Qcoh(X\times_{\spec(A)} \spec(A_n)),$$ corresponds to pro-systems of complexes $$\widehat{\cF}= \{\cF\otimes_A A\qq (f_1^n,\ldots, f_r^n)\}_{n>0}.$$ 
    A similar statement, of course, also holds for $\dalg(X).$
    This viewpoint will be useful in our in-preparation work \cite{MundingerSahai_inprep}. See Remark \ref{remark: limitations of quasi-affine}.
\end{remark}

\section{Graded derived algebras and comonadicity.}\label{sec: graded dalg and comonads}

\subsection{Preliminaries on graded derived algebras.}
In this section we collect some preliminaries on graded derived algebras needed for our geometric classification theorems.

\begin{recollection}(Local systems of spectra .)\label{rec: loc sys of spectra}
Let $R$ be an $\einf$-ring spectrum and let $X$ be a space. Then, following \cite[\S~3]{AndoBlumbergGepner2018ParametrizedThom} and \cite{Lurie2014Lecture21ATheory}, the $\icat$ of $\Mod_R$ valued local systems on $X$ is defined as $\loc(X,R):=\mathrm{Fun}(X^\op,\Mod_R)$\footnote{Note that as explained in \cite[Remark 3.2]{AndoBlumbergGepner2018ParametrizedThom}, in the $\infty$-category of space there is a canonical equivalence between functors $(-)^\op\to \id$. In particular $\loc(X,R)\simeq \loc(X^\op,R)$ but we stick with the first category as in the applications we are interested in working with decreasing filtrations.}.

For any morphism $f\colon X\to Y$ of spaces, there is an induced pullback functor $f^*\colon \loc(Y,R)\to \loc(X,R)$. As explained in \cite[\S~5.4]{AndoBlumbergGepner2018ParametrizedThom} and \cite[Remark 5]{Lurie2014Lecture21ATheory}, the $\icat$ $\Mod_R$ is presentable and admits all small limits and colimits and therefore, by left and right Kan extensions the functor $f^*$ admits a left adjoint $$f_!\colon \loc(X,R)\rightleftarrows \loc(Y,R)\colon f^*$$ the \emph{proper} pushforward and a right adjoint $$f^*\colon \loc(Y,R)\rightleftarrows \loc(X,R)\colon f_*$$ which is the usual pushforward. 
\end{recollection}

For a space $X$ which admits an $\einf$-group like structure and an $\einf$-ring $R$, there are two monoidal structures on $\loc(X,R)$ which we now recall. 

\begin{recollection}(Pointwise monoidal structure on local systems of spectra.)\label{pointwise monoidal}
 The \emph{pointwise} monoidal two local systems $\cL_1,\cL_2\colon X\to \Mod_R$ to the local system $\cL_1\otimes_X \cL_2:=(\Delta_X)^*(\cL_1\boxtimes \cL_2)$ where $\Delta_X\colon X\to X\times X$ is the diagonal and $\cL_1\boxtimes \cL_2$ is the local system $p_1^*\cL_1\otimes p_2^*\cL_2\in \loc(X\times X,R).$ Explicitly it's value on $x\in X$ is $(\cL_1\otimes_X \cL_2) (x)=\cL_1(x)\otimes_R \cL_2(x).$

With respect to this monoidal structure, for a morphism of space $f\colon X\to Y$, the induced functors $f_!\dashv f^* \dashv f_*$ of Recollection \ref{rec: loc sys of spectra} have the following properties:

\begin{enumerate}
    \item $f^*$ is symmetric monoidal. 
    \item $f_!$ being a right adjoint to a symmetric monoidal functor is op-lax monoidal.
    \item $f_*$ being a left adjoint to a symmetric monoidal functor is lax monoidal.
    \item $f_!$ and $f^*$ satisfy a projection formula  $f_!(f^*(-)\otimes_X(-))\simeq (-)\otimes_Y f_!(-).$
\end{enumerate}

See \cite[Proposition 6.8]{AndoBlumbergGepner2018ParametrizedThom} for a proof. 
\end{recollection}

\begin{recollection}(Day convolution on local systems of spectra over commutative monoids in spaces.)\label{day conv on locsys}
Now let $X$ be an $E_1$-grouplike object in space as in \cite[Definition 5.2.6]{HA}. The multiplication $m\colon X\times X\to X$ endows $X$ with a monoidal structure which is symmetric monoidal when $X$ is $\einf$. It follows from \cite[Corollary 6.13]{AndoBlumbergGepner2018ParametrizedThom} that $\loc(X,R)$ is now endowed with a symmetric monoidal structure given by Day convolution \cite[\S~2.2.6]{HA} with respect to which for any morphism of $\einf$-grouplike objects in space $f\colon X\to Y$, then $f_!$ is symmetric monoidal. 
\end{recollection}

\begin{remark}\label{remark: animated ring loc sys}
    Let $R$ now be an animated ring. Then the notion of $\Mod_R$-valued local systems of Recollection \ref{rec: loc sys of spectra} depend only on the underlying $\einf$-ring $R^\circ$ of \ref{forgetful}. 
\end{remark}

\begin{notation}
    We can now prove the first result that we want. We set $\ZZ^\ds$ be the discrete grouplike $\einf $space obtained by thinking of $\ZZ$ as a discrete space. 
    Then $\loc(\ZZ^\ds,\Mod_R)\simeq \Gr\Mod_R$ and the symmetric monoidal product induced in Recollection \ref{day conv on locsys} induces the usual Day convolution monoidal structure on $\Gr\Mod_R.$
\end{notation}

\begin{recollection}\label{rec: morphisms in contexts}
    Recall following \cite[Definition 4.2.1]{Rak20} we recall that a morphism of derived algebraic contexts $\cC$ and $\cD$ is a colimit preserving right $t$-exact symmetric monoidal functor $F\colon \cC\to \cD$ so that $F(\cC^0)\subset \cD^0.$ Note that since $\C$ and $\cD$ are both presentable $F$ admits a right adjoint $G.$

    As explained in \cite[Remark 4.2.25]{Rak20} these functors induce a colimit morphism $F'\colon \dalg(C)\to \dalg(D)$ which admits a right adjoint (due to the adjoint functor theorem, accessible by the presentability of the source and target) $G'\colon \dalg(D)\to \dalg(C)$ so that the following diagrams commute: 
    
    \begin{equation*}
       \begin{tikzcd}
\dalg(\cC) \arrow[d, "U_C"'] \arrow[r, "F'"] & \dalg(\cD) \arrow[d, "U_D"] \\
\cC \arrow[r, "F"]                         & \cD                        
\end{tikzcd} \text{ and }  
\begin{tikzcd}
\dalg(\cC) \arrow[d, "U_C"'] & \dalg(\cD) \arrow[d, "U_D"] \arrow[l, "G'"'] \\
\cC                          & \cD \arrow[l, "G"']                         
\end{tikzcd}.
    \end{equation*}
    
\end{recollection}

It is now very useful to recall the comonadic version of Barr-Beck-Lurie \cite[Theorem 4.7.0.3]{HA}. Our statement follows the somewhat streamlined presentation in \cite[Theorem 3.3]{galoisgroup}.

\begin{theorem}[Comonadic Barr-Beck-Lurie]\label{barrbeck comonadic}
    Let $F\colon \cC \rightleftarrows \cD\colon G$ be a pair of adjoint functors between $\icats$. Then $F$ is comonadic if and only if 
    \begin{enumerate}
        \item $F$ is conservative.
        \item Given a cosimplicial object $X^\bullet\in \cC$ such that $F(X^\bullet)$ is a split cosimplicial object, then $\Tot(X^\bullet)$ exists in $\cC$ and $F(\Tot(X^\bullet))\to \Tot(F(X^\bullet))$ is an equivalence.
    \end{enumerate}
\end{theorem}

\begin{lemma}\label{monadicity in dalg}
    Let $F\colon \cC\to \cD$ be a morphism of derived algebraic contexts (which automatically admits a right adjoing $G$) such that $F$ is comonadic. Then the induced functor $F'\colon \dalg(\cC)\to \dalg(D)$ (which automatically admits a right adjoint $G'$) is comonadic.
\end{lemma}
\begin{proof}
    We need to check the conditions in Theorem \ref{barrbeck comonadic}. The first condition can be checked after applying $U_D$ in the first diagram in Recollection \ref{rec: morphisms in contexts} in which case it follows from the conservativity of $F\circ U_C$. 
    The second fact can also be checked after applying $U_D$ as the functors $U_C$ and $U_D$ preserve all limits.
\end{proof}

\begin{proposition}\label{pi! is a morphism of contexts}
Let $\pi\colon \ZZ^\ds\to {*}$ be the canonical projection and let
The functor $\pi_!\colon \loc(\ZZ^\ds,\ZZ)\simeq \Gr\Mod_\ZZ\to \loc({*}, \ZZ)\simeq \Mod_\ZZ$ be the induced symmetric monoidal functor of Recollection \ref{day conv on locsys}. Then $\pi_!$ is a morphism of derived algebraic contexts in the sense of Recollection \ref{rec: morphisms in contexts}.
\end{proposition}

\begin{proof}
    First note that $\pi_!$ is colimit preserving. Indeed it is a left adjoint and so preserves all colimits. In fact one can write down a formula for $\pi_!$ evaluated at $E$ for any $E\in \Gr\Mod_\ZZ$ as $\pi_!(E)=\bigoplus_{i\in \ZZ}E(i).$ Indeed this follows from the definition of $\pi_!$ as a left Kan extension over a discrete space.

    The formula then shows that for the neutral $t$-structure on $\Gr\Mod_\ZZ$, the functor $\pi_!$ is $t$-exact and hence right $t$-exact. 

    It remains to verify compact generation. The $\infty$-category $\Mod_\ZZ$ is projectively generated by $\ZZ\in (\Mod_\ZZ)^\heartsuit$ by \cite[Corollary 7.1.4.15.]{HA}. Thus as explained in \cite[Construction 4.3.4]{Rak20} we may generate $\Gr\Mod_\ZZ$ by finite coproducts of $\ZZ(n)$ where $(n)$ denotes $\ZZ$ in grading degree $n.$ Since these also generate $(\Mod_\ZZ)_{\geq0}$ after applying $\pi_!$ as we see from the formula above, we conclude.
\end{proof}

\begin{remark}\label{remark: Moulinos's proof works for animated rings}
  As remarked in Remark \ref{remark: animated ring loc sys}, when $R$ is animated one can also contemplate $\Gr\Mod_R$ and $\Mod_R$ through the lens of $\Mod_R$ valued local systems on $\ZZ^\ds$ and ${*}$ respectively. In this case the proof of Proposition \ref{Moulinos's coaction} works formally over any animated (in fact $\einf$) ring $R$. Indeed this is already observed in \cite[Remark 5.2]{Moulinos2021GeometryOfFiltrations}. We will use this fact in the sequel.
\end{remark}

\begin{corollary}\label{pi! on dalg}
    The functor $\pi_!\colon \Gr\Mod_\ZZ\to \Mod_\ZZ$ lifts to a morphism of derived algebras $\pi_!\colon \Gr\dalg_\ZZ\to \dalg_\ZZ$ where $\Gr\dalg_R:=\dalg(\Gr\Mod_R).$
\end{corollary}
\begin{proof}
    This follows from functoriality of derived algebraic contexts and Proposition \ref{pi! is a morphism of contexts}.
\end{proof}

\begin{remark}(Naturality of $\pi_!$.)\label{naturalityofpi!} As explained in Recollection \ref{rec: morphisms in contexts} the functor $\pi_!\colon \Gr\dalg_\ZZ\to \dalg_\ZZ$ preserves all small colimits and admits a right adjoint $\pi^*$ so that the analogue of the diagrams in \ref{rec: morphisms in contexts} commute. 

\begin{equation}
    \begin{tikzcd}
\Gr\dalg_\ZZ \arrow[r, "\pi_!"] \arrow[d, "U"'] & \dalg_\ZZ \arrow[d, "U"] \\
\Gr\Mod_\ZZ \arrow[r, "\pi_!"]                  & \Mod_\ZZ                
\end{tikzcd}\text{ and } \begin{tikzcd}
\Gr\dalg_\ZZ \arrow[d, "U"'] & \dalg_\ZZ \arrow[d, "U"] \arrow[l, "\pi^*"'] \\
\Gr\Mod_\ZZ                  & \Mod_\ZZ \arrow[l, "\pi^*"]                 
\end{tikzcd}.
\end{equation}
\end{remark}

\begin{notation}
    Let $\mathbf{S}$ be the sphere spectrum. Then let $\mathbf{S}[\ZZ]:=\Sigma_{+}^\infty\ZZ^\ds$. The functor $\Sigma_{+}^\infty\colon \mathcal{S}\to \Sp$ is symmetric monoidal and when evaluated in the discrete space $\ZZ$, gives a biaglebra in $\Sp$. If $R$ is any $\einf$-ring spectrum then $R[\ZZ]:=R\otimes \mathbf{S}[\ZZ]$ and in particular if $R$ is discrete then the latter is just the bicommutative Hopf algebra $R[t,t^{-1}]$  which is the Hopf algebra of functions on $\G_m.$\footnote{We clarify that the point of recalling this notation here is to show that Moulinos's argument for $\mathbf{S}[\ZZ]$ goes through for $R[\ZZ]$ as well by base change.}
\end{notation}
We now recall one of the main result of \cite[Proposition 4.2]{Moulinos2021GeometryOfFiltrations}. Before this we recall how to describe equivalences of monads. 

\begin{remark}[How to detect equivalences of monads?]\label{remark: how to detect equivalences of monads}
    Given an $\icat$ $\cC$ a monad $T$ on $\cC$ is formally a $\mathbf{E}_1$ algebra object of $\End(\cC).$ Let $T\to S$ be a morphism of $\mathbf{E}_1$-algebra objects in $\End(\cC).$ To show that such a morphism is an equivalence we use that fact that the map $\mathrm{Alg}(\End(\cC))\to \End(\cC)$ is conservative. In other words it suffices to verify that the morphism of endofunctors is an equivalence, and the latter can be detected on objects of $\cC.$ An equivalent statement applies in the case of comonads. We will use this without comment in the sequel.
\end{remark}

\begin{proposition}[Moulinos]\label{Moulinos's coaction}
    The functor $\pi_!\colon \Gr\Mod_\ZZ\to \Mod_\ZZ$ is comonadic and identifies the source as $\cmod_{\ZZ[\ZZ]}(\Mod_\ZZ)$
\end{proposition}
\begin{proof}
We sketch Moulinos's proof.
This is an application of the comonadic Bar-Beck-Lurie, recalled for example in \cite[Theorem 4.4]{Moulinos2021GeometryOfFiltrations}. 
The comondadicity is simple to check using the explicit formula for $\pi_!$ described in the proof of Lemma \ref{pi! is a morphism of contexts}.\footnote{We will recall the argument in the proof of Proposition \ref{prop: cts monad mod}.}
It remains to identify the comonad $\pi_!\circ \pi^*\colon \Mod_\ZZ\to \Mod_\ZZ$ with a coaction of $\ZZ[\ZZ]$. One sees first that by point $(2)$ of Recollection $\ref{pointwise monoidal}$ the functor $\pi_!$ is op-lax monoidal for the pointwise monoidal structure and so $\pi_!(1_\ZZ)$ is a Hopf algebra in $\Mod_\ZZ$. Since $1_\ZZ=\pi^*\ZZ$ one sees that $\pi_!(\pi^*\ZZ)=\ZZ[\ZZ]$. 
Then by the projection formula in point $(3)$ of Recollection $\ref{pointwise monoidal}$ one sees that for any $M\in \Mod_\ZZ$ one has $\pi_!(\pi^*(M))=\pi_!(\pi^*(M)\otimes_{\ZZ^\ds}1_\ZZ)= M\otimes_\ZZ \pi_!(1_\ZZ)=M\otimes_\ZZ\ZZ[\ZZ]$ and the diagrams witnessing the $\pi_!\circ \pi^*$-comodule structure on $M$ amount to showing that $M$ admits a coaction of $\ZZ[\ZZ].$
\end{proof}

\begin{corollary}{\label{pi1 is comonadic on dalg}}
    The functor $\pi_!\colon \Gr\dalg_\ZZ\to \dalg_\ZZ$ of Corollary \ref{pi! on dalg} is comonadic and exhibits the source as the $\cmod_{\ZZ[\ZZ]}(\dalg_\ZZ).$
\end{corollary}
\begin{proof}
    In view of Proposition \ref{Moulinos's coaction}, Lemma \ref{pi! is a morphism of contexts}, the first part of the statement follows from Lemma \ref{monadicity in dalg}. It remains to identify the comonad $\pi_!\circ \pi^*\colon \dalg_\ZZ\to \dalg_\ZZ$ with a coaction of $\ZZ[\ZZ].$ First note that $\pi_!(1_\ZZ)=\ZZ[\ZZ]$ and this is a Hopf algebra object of $\dalg_\ZZ$. So one can contemplate the category $\cmod_{\ZZ[\ZZ]}(\dalg_\ZZ).$ It suffices to now identify the functor $\pi_!\circ \pi(-)\simeq (-)\otimes_\ZZ \ZZ[\ZZ].$
    To see this note that from the diagram in Remark \ref{naturalityofpi!} we have the following commutative diagram  
    $$
\begin{tikzcd}
\dalg_\ZZ \arrow[r, "\pi^*"] \arrow[d, "U"'] & \Gr\dalg_\ZZ \arrow[r, "\pi_!"] \arrow[d, "U"'] & \dalg_\ZZ \arrow[d, "U"'] \\
\Mod_\ZZ \arrow[r, "\pi^*"]                  & \Gr\Mod_\ZZ \arrow[r, "\pi_!"]                  & \Mod_\ZZ                 
\end{tikzcd}$$

So for $T\in \dalg_\ZZ$ we have $$U(\pi^*\pi_!(T))\simeq \pi^*\pi_! U(T)=U(T)\otimes_\ZZ \ZZ[\ZZ]=U(T)\otimes_{U(R)} U(\ZZ[\ZZ])=U(T\otimes_\ZZ \ZZ[\ZZ])$$ wherein we used that $\ZZ[\ZZ]$ is a Hopf algebra object of $\dalg_\ZZ$ and $U$ is symmetric monoidal and so we conclude by the conservativity of $U.$
\end{proof}

\begin{remark}\label{remark: remark on diagonalizable coactions}
    It is not hard to generalize the previous constructions to the following. Let $M$ be a discrete abelian group. Then one may consider the group algebra $\Sigma^\infty_{+} M^{\ds}$ and therefore the base change $\ZZ[M]=\ZZ\otimes_{\sph} \Sigma^\infty_{+} M^{\ds}$. This is the commutative bicommutative Hopf algebra of functions on the diagonalizable group scheme $\mathbf{D}(M)$ \cite[\S~2.6]{Tate1997FiniteFlatGroupSchemes}. Now define $\Gr\Mod^M_\ZZ=\Fun(M^\ds,\Mod_\ZZ)$ and $\Gr\dalg^M_\ZZ=\dalg(\Gr\Mod^M_\ZZ)$. Then write $\pi^M\colon M^\ds\to *$ for the canonical projection and $\pi_!^M\colon \Gr\Mod_{\ZZ}^M\to \Mod_\ZZ$ for the induced compactly supported cohomology functor of  Recollection \ref{rec: loc sys of spectra}. Then the direct analogue of Corollary \ref{pi1 is comonadic on dalg} remains true. More precisely, 
    \begin{enumerate}
        
        \item The functor $\pi_!^M\colon \Gr\Mod_\ZZ^M\to \Mod_\ZZ$ is comonadic and exhibits the source as the $\icat$ of comodules over bicommutative Hopf algebra $\ZZ[M]$. 
        \item The morphism of derived algebraic contexts $\pi_!^M\colon \Gr\dalg_\ZZ^M\to \dalg_\ZZ$ is comonadic and exhibits the source as the $\icat$ of comodules in $\dalg_\ZZ$ over the bicommutative Hopf algebra $\ZZ[M].$
    \end{enumerate}
    Indeed this follows directly from replicating the reasoning already present in this section \emph{with only corresponding notational changes.}
\end{remark}

\subsection{Generalisation to graded algebras over animated rings.}\label{ssec: generalisation to graded algebras}

We are now interested in showing that most of the results in the previous subsection work over an arbitrary animated ring $A$. Since $\Mod_A$ is not a context unless $A$ is discrete we have to be a bit careful about how we approach this.

\begin{remark}\label{remark: identify animated with connective dalg}
    Recall from \cite[Remark 4.3.3]{Rak20} that the $\icat$ $$\dalg_\ZZ^\cn:=\dalg_\ZZ\times_{\Mod_\ZZ}(\Mod_\ZZ)_{\geq 0}$$ agrees with animated commutative rings. A similar observation shows that $$\Gr\dalg_\ZZ^\cn:=\Gr\dalg_\ZZ\times_{\Gr\Mod_\ZZ}(\Gr\Mod_\ZZ)_{\geq 0}$$ where the $t$-structure is the neutral one, can \emph{by definition} be treated as graded animated rings.
\end{remark}

\begin{remark}\label{remark: ins0 animated rings}
    Observe that by \cite[Remark 4.3.6]{Rak20} there is a natural functor $\ins^0\colon \dalg_\ZZ\to \Gr\dalg_\ZZ$ viewing a derived algebra in graded degree $0$. This functor has a right adjoint $\ev^0\colon \Gr\dalg_{\ZZ}\to \dalg_\ZZ$, evaluating the $0$th graded piece of a derived algebra so that $\ev^0\ins^0\simeq \id.$ 
    In particular $\ins^0$ is fully faithful. Note that after the identification of animated rings with $\dalg^\cn$ of Remark \ref{remark: identify animated with connective dalg} the functor $\ins^0$ sends $\dalg^\cn\to \Gr\dalg^\cn$ and so we may thus regard an animated ring $A$ as a graded animated ring in $\Gr\dalg_{\ZZ}^\cn.$ 
\end{remark}
\begin{construction}\label{cons: lifting pi! to animated}
Let $A$ be an animated ring.
    We explain why the induced functor $\pi_!\colon \Gr\dalg_\ZZ\to \dalg_\ZZ$ can be lifted to a functor $\pi_!\colon \Gr\dalg_{A}\to \dalg_A$ so that the analogues of the diagram in Remark \ref{naturalityofpi!} commute.

Note that the functor $\pi_!\colon \Gr\dalg_{\ZZ}\to \dalg_\ZZ$ satisfies $\pi_!\circ \ins^0=\id$ by examining the formula explained in the proof of Proposition \ref{Moulinos's coaction}. Now let $A$ be an animated ring considered as a graded animated ring by Remark \ref{remark: ins0 animated rings}. Then we obtain a functor 
$$\pi_!\colon (\Gr\dalg_\ZZ)_{A/}\to (\dalg_\ZZ)_{A/}.$$
After identifying the source and target with  $\Gr\dalg_{A}$ and $\dalg_{A}$ by point $(b)$ of \cite[Notation 4.2.28]{Rak20} we obtain the desired lift 
$$\pi_!\colon \Gr\dalg_{A} \rightleftarrows \dalg_A \colon \pi^*.$$

Note that $\pi_!$ preserves all colimits since colimits of non-empty diagrams are computed in the underlying category $\Gr\dalg_\ZZ$ and $\dalg_\ZZ$ respectively, where they can be computed by the proof of Corollary \ref{pi1 is comonadic on dalg} and the initial objects are preserved by fiat. Since both the source and target are presentable $\icats$, it follows that there is right adjoint $R\colon \dalg_A\to \Gr\dalg_A$ which by generalities on $\icats$\footnote{Let $F\colon \cC\to \cD$ be a functor between presentable $\icats$ which preserves all colimits and therefore admits a right adjoint $G$. Then $F\colon C_{c/}\to D_{F(c)/}$ also preserves all colimits and its adjoint is heuristically given by sending $(F(c)\to d)\mapsto (c\xrightarrow{\eta_c}GF(c)\to G(d))$ where $\eta_c$ is the unit. This is \cite[Proposition 5.2.5.1]{HTT} for the case of overcategories.} is given by the restriction of $\pi^*\colon \dalg_\ZZ\to \Gr\dalg_{\ZZ}$ to the under-categories. 
Thus we have an adjunction 
$$\pi_!\colon \Gr\dalg_A\rightleftarrows \dalg_A\colon \pi^*$$

\end{construction}

\begin{construction}\label{cons: commuting diagrams for animated}
    We have observed in Remark \ref{remark: Moulinos's proof works for animated rings} that if $A$ is an animated ring, then there is a comonadic adjunction 
$$\pi_!\colon \Gr\Mod_A\rightleftarrows \Mod_A\colon \pi^*$$ such that $\pi_!\circ \pi\simeq (-)\otimes_A A[\ZZ]$ as comonads. Now it is clear from the construction of the adjunction in Construction \ref{cons: lifting pi! to animated} that the following diagrams commute

\begin{equation}
    \begin{tikzcd}
\Gr\dalg_A \arrow[r, "\pi_!"] \arrow[d, "U"'] & \dalg_A \arrow[d, "U"] \\
\Gr\Mod_A \arrow[r, "\pi_!"]                  & \Mod_A                
\end{tikzcd}\text{ and } \begin{tikzcd}
\Gr\dalg_A \arrow[d, "U"'] & \dalg_A \arrow[d, "U"] \arrow[l, "\pi^*"'] \\
\Gr\Mod_A                  & \Mod_A \arrow[l, "\pi^*"]                 
\end{tikzcd}.
\end{equation}
\end{construction}

\begin{proposition}\label{prop: monadicity over animated}
    Let $A$ be an animated ring. Then the functor $\pi_!\colon \Gr\dalg_A\to \dalg_A$ of Construction \ref{cons: lifting pi! to animated} is comonadic and exhibits the source as $\cmod_{A[\ZZ]}(\dalg_A).$
\end{proposition}
\begin{proof}
    In light of Construction \ref{cons: commuting diagrams for animated}, this follows \emph{mutatis mutandis} as in the proof of Corollary \ref{pi1 is comonadic on dalg}.
\end{proof}

\begin{remark}\label{remark: coactions for diagonalizable on animated rings}
    Continuing with the thread (and corresponding notation) of Remark \ref{remark: remark on diagonalizable coactions}, we see that analogues of Proposition \ref{prop: monadicity over animated} remains true with \emph{the same proof} giving us the following result 

    Let $A$ be an animated ring, then the functor $\pi_!^M\colon \Gr\dalg^M_A\to \dalg_A$ constructed analogously to Construction \ref{cons: lifting pi! to animated} is comonadic with right adjoint $\pi^{*,M}\colon \dalg_A\to \Gr\dalg^M_A$, and exhibits the source as comodules over the bicommutative Hopf algebra $A[M]=A\otimes_\ZZ\ZZ[M]$ i.e. $\cmod_{A[M]}(\dalg_\ZZ).$

    One may think of $A[M]$ as the corresponding derived ring of functions on the diagonalizable group scheme $\mathbf{D}(M)\times \spec(A).$
\end{remark}

\subsection{Continuous comonadicity.}\label{ssec: continuous comonad}

Our aim now is to prove a continuous analogue of Proposition \ref{prop: monadicity over animated} in the $J$-complete setting of \S~\ref{sec: jcomp setting}. for the reader's convenience we recall the notation therein.

\begin{notation}\label{notation: jcomp redux}
Fix an animated ring $A$ and a finitely generated ideal $J\subset \pi_0(A).$
\end{notation}

Our first goal is to set up a variant of Definition \ref{def: dalgj 1} for graded $A$-modules.

\begin{recollection}\label{rec: def of gradedjcomp}
    We define the $\infty$-category of $J$-complete graded $A$-modules as the full subcategory of $\Gr\Mod_A$ so that the $\gr_i\colon \Gr\Mod_A\to \Mod_A$ factor through $\Mod_A^\jcomp\to \Mod_A.$ Equivalently, one has $$\Gr\Mod_A^\jcomp=\mathrm{Fun}(\ZZ^\ds,\Mod_A^\jcomp).$$

    This makes it compatible with, for example, Notation $6.1.1$ and Notation $\text{D}.12$ of \cite{APC}.
    
\end{recollection}

\begin{definition}\label{def: def of dalggradedjcomp}
    Let $\Gr\dalg_A^\jcomp$ denote full subcategory of $\Gr\dalg$ so that the forgetful functor $\Gr\dalg_A\to \Gr\Mod_A$ factors over $\Gr\Mod_A^\jcomp.$ 
\end{definition}

\begin{remark}\label{lem: not hard pullback}
    It is not too hard to give a pullback definition of $\Gr\dalg_A^\jcomp$ as in Lemma \ref{lem: pullback def dalgj}. We do not do so here. However we remark that the forgetful functor $U_J\colon \Gr\dalg_A^\jcomp\to \grmod_A^\jcomp$ is easily seen to be conservative and limit preserving. The conservativity follows by definition and the preservation of limits follows by noting that the graded analogue of the pullback definition of Lemma \ref{lem: pullback def dalgj} takes place in $\Pr^R$ whence the forgetful functor is a right adjoint.
\end{remark}

\begin{construction}\label{cons: continuous pi1}
We consider the functor $(\pi_!)\cj\colon \Gr\Mod_A^\jcomp\to \Mod^\jcomp_A$ as the composite $(-)\cj\circ\pi_!$. It is clear that this functor preserves colimits as it is the composite of the inclusion $$\Gr\Mod_A^\jcomp\into\Gr\Mod_A \xrightarrow{\pi_!}\Mod_A\xrightarrow{(-)\cj}\Mod_A^\jcomp$$ and colimits in $\Gr\Mod_A^\jcomp$ are computed in $\Gr\Mod_A$ and then completed.
One can also note the preservation of colimits by the following chain of equivalences of mapping spaces

$$\Maps_{\Mod_A^\jcomp}((\pi_!N)\cj,M)\simeq \Maps_{\Mod_A}(\pi_!N,M)\simeq \Maps_{\Gr\Mod}(N,\pi^*M)=\Maps_{\Gr\Mod_A^\jcomp}(N,\pi^*M),$$ where the first equivalence comes from the fact that completion is left adjoint to the inclusion $\Mod_A^\jcomp\into \Mod_A$, the second equivalence comes from the adjunction of Proposition \ref{Moulinos's coaction}, the third equality from the fact that both $N$ and $\pi^*M$ are $J$-complete graded modules, the former by fiat and the latter because $\pi^*M$ inserts $M$ in every graded degree and $M$ is already in $\Mod_A^\jcomp,$ again by fiat.
 Thus we have an adjunction

$$(\pi_!)\cj\colon \Gr\Mod_A^\jcomp\rightleftarrows \Mod_A^\jcomp\colon \pi^*.$$
\end{construction}

\begin{proposition}\label{prop: cts monad mod}
    The adjunction of Construction \ref{cons: continuous pi1} is comonadic and we have an identification of comonads $(\pi_!)\cj\circ \pi^*\simeq (-)\ctensor_A A[\ZZ].$
\end{proposition}

\begin{proof}

    We use the criteria recalled in Theorem \ref{barrbeck comonadic} and follow the ideas in the proof of \cite[Proposition 4.2]{Moulinos2021GeometryOfFiltrations} to verify them.
    We first check conservativity. Note that the composite $$\Gr\Mod_A^\jcomp\into \Gr\Mod_A\xrightarrow{\pi_!}\Mod_A$$ is conservative. So we need to check that $(-)\cj$ doesn't kill anything which is in the image of the composite above. But note that the functor explicitly sends the direct sum $\bigoplus_{i\in \ZZ} M_i$ of modules where each $M_i\in \Mod_A^\jcomp$ to their coproduct in $\Mod_A^\jcomp$. The latter vanishes if and only if each of the $M_i$ is $0.$ This can be seen because for any $i\in \ZZ$ the composite is the identity
    $$M_i\to \bigoplus_{i\in \ZZ} M_i\to \prod_{i\in \ZZ}M_i\to M_i$$ and the $M_i$ are already $J$-complete, so $(-)\cj$ is the identity.
    It remains to show criteria $(2)$ in Theorem \ref{barrbeck comonadic}. We want to show that if $X^\bullet\in \Gr\Mod_A^\jcomp$ is a $\pij$ split cosimplicial object then $X^\bullet$-splits in $\Gr\Mod_A^\jcomp.$ To see this first observe that since $\grmod^\jcomp_A=\mathrm{Fun}(\ZZ^\ds,\Mod_A^\jcomp)$, and limits in functor categories are computed levelwise, it suffices to show that the components $X^\bullet(i)$ for $i\in \ZZ$ split in $\grmod_A^\jcomp.$
    Note that $X^\bullet(i)$ is a retract of $X^\bullet$. Indeed there is a functor $\ev_i\colon \Gr\Mod_A^\jcomp\to \Mod^\jcomp_A$ with left adjoint $\ins^i$. Note that $\ins^i$ is also a right adjoint to $\ev^i.$ Thus we have $$\ins^i \ev^i X^\bullet=X(i) \to X^\bullet \to \ins^i\ev^i X^\bullet=X(i)$$ which gives the retraction.

    Thus $\pij(X^\bullet(i))$ is a retract of $\pij(X^\bullet).$ Since the latter splits, it follows from \cite[Corollary 4.7.2.13]{HA} that the former splits. But $\pij(X^\bullet(i))$ is just $X^\bullet(i)\in \Mod_A^\jcomp$, whence we may split the object $X^\bullet$ by taking the product of $X^\bullet(i)$ in each degree $i\in \ZZ$. 
    Now since any functor preserves split cosimplicial objects (for example,as explained in \cite[Example 3.11]{galoisgroup}, it follows that the canonical morphism $\pij(\Tot(X^\bullet))\to \Tot(\pij(X^\bullet))$ is an equivalence. 

    Lastly it remains to identify the continuous comonad with tensoring with the completed Hopf algebra $(-)\ctensor_A A[\ZZ].$ As explained in Remark \ref{remark: Moulinos's proof works for animated rings} or Construction \ref{cons: commuting diagrams for animated} we know that there's an equivalence of comonads $\pi_!\circ \pi^*(-)\simeq (-)\otimes_A A[\ZZ]$ whence we conclude by completing both the source and the target.\footnote{Note that $\pij$ has the same monidality as $\pi_!$, as $(-)\cj$ is symmetric monoidal.}
\end{proof}

\begin{construction}\label{cons: continuous pij on derived algebras}
   In Construction \ref{cons: lifting pi! to animated} we have constructted an adjunction $$\pi_!\colon \Gr\dalg_A \rightleftarrows \dalg_A\colon \pi^*.$$ 
   Consider the composite $\Gr\dalg_A\xrightarrow{\pi_!}\dalg_A\xrightarrow{(-)\cj}\dalg_A^\jcomp$ where the second arrow is the completion operation of Construction \ref{cons: completion def} and the target is the $\icat$ of Definition \ref{def: dalgj 1}. For the same reason as the mapping space computation done in Construction \ref{cons: continuous pi1}, it follows that there is an adjunction 
   $$\pij\colon \Gr\dalg_A^\jcomp \rightleftarrows \dalg_A^\jcomp\colon \pi^*$$ and the following diagrams commute 

   \begin{equation}
    \begin{tikzcd}
\Gr\dalg^\jcomp_A \arrow[r, "\pij"] \arrow[d, "U_J"'] & \dalg^\jcomp_A \arrow[d, "U_J"] \\
\Gr\Mod^\jcomp_A \arrow[r, "\pij"]                  & \Mod^\jcomp_A                
\end{tikzcd}\text{ and } \begin{tikzcd}
\Gr\dalg^\jcomp_A \arrow[d, "U_J"'] & \dalg^\jcomp_A \arrow[d, "U_J"] \arrow[l, "\pi^*"'] \\
\Gr\Mod^\jcomp_A                  & \Mod^\jcomp_A \arrow[l, "\pi^*"]                 
\end{tikzcd}.
\end{equation}
Indeed the commutativity of the first diagram is a consequence of the commutativity of the first diagram in Construction \ref{cons: commuting diagrams for animated} along with the left adjointability diagram of Corollary \ref{cor: dalgj is right adjointable}, while the commutativity of the second diagram is a consequence of the fact that $\pi^*$ sends complete objects to complete objects.
\end{construction}

\begin{corollary}\label{cor: continuous comonad dalg}
    The adjunction of Construction \ref{cons: continuous pij on derived algebras} is comonadic and there is an isomorphism  of comonads $\pij\circ \pi^*(-)\simeq (-)\ctensor_A (A[\ZZ])^{\wedge}_J.$
\end{corollary}
\begin{proof}
    In light of the commutativity of the diagrams in Construction \ref{cons: continuous pij on derived algebras}, this follows from Proposition \ref{prop: cts monad mod} in the same way as Corollary \ref{pi1 is comonadic on dalg} followed from Proposition \ref{Moulinos's coaction}.
\end{proof}

\begin{remark}\label{remark: continuous coaction diagonalizable schemes}
    Continuing with the ideas of Remarks \ref{remark: remark on diagonalizable coactions} and \ref{remark: coactions for diagonalizable on animated rings}, we  may define for any commutative abelian group $M$, the $\icat$ $\Gr\Mod_A^{M,\jcomp}$ as the $\icat$ $\mathrm{Fun}(M^\ds, \Mod_A^\jcomp)$ which is a direct analogue of Definition \ref{def: def of dalggradedjcomp}. 
    Repeating the reasoning employed in this section, we may also get a comonadic adjunction

    $$(\pi_!^M)^\wedge_{J}\colon \Gr\Mod_A^{M,\jcomp}\leftrightarrows \Mod_A^{M,\jcomp}\colon \pi^{*,M}$$ identifying the continuous comonad $$(\pi_!^M)\circ \pi^{*,M}(-)\simeq (-)\widehat{\otimes}_A A[M].$$
We may now define $\Gr\dalg_A^{M,\jcomp}$ exactly as in Definition \ref{def: def of dalggradedjcomp}.
    Thus we have a comonadic adjunction 
    $$(\pi_!^M)^\wedge_{J}\colon \Gr\dalg_A^{M,\jcomp}\leftrightarrows \dalg_A^{M,\jcomp}\colon \pi^{*,M}$$ identifying the continuous comonad $$(\pi_!^M)\circ \pi^{*,M}(-)\simeq (-)\widehat{\otimes}_A A[M].$$

    Here we view the completed bicommutative Hopf algebra $A[M]^\wedge_J$ as the ring of functions on the formal group scheme $\mathbf{D}(M)\times \spf(A).$ 
\end{remark}

\section{Derived algebras on geometric stacks.}\label{sec: dalg on geomstacks}

In this section our objective is to study the functor $\dalg(-)$ on geometric stacks. Here we use the term geometric somewhat vaguely to mean roughly a stack which admits an faithfully flat representable affine morphism from a representable (formal) stack, but will specialize in the course of this section.

We begin by establishing flat descent for derived algebras.

\subsection{Flat descent for derived algebras}

In this section we show that the functor which sends an animated ring $A\mapsto \dalg_A$ is a flat sheaf valued in $\Calg(\Pr^L).$ This is a relatively simple application of (comonadic) Barr-Beck-Lurie theorem, we chose to prove it since we cannot find an equivalent statement in the literature.

\begin{remark}
    In \cite[Corollary 2.19]{MM24}, the authors prove that the functor sending an animated ring $A$ to the $\infty$-category of $n$-truncated derived $A$-algebras $\dalg_{A,\leq n}$ is a flat hypersheaf. Here $n$-truncations are defined in Construction $2.20$ in \emph{loc.cit.}
    If Postnikov towers of derived algebras over an animated ring converge then flat hyperdescent can be deduced from their result. However, since we did not wish to ruminate on convergence issues, we have decided to record a simpler direct proof of flat descent.
\end{remark}

\begin{question}
    Does the assignment $A\mapsto \dalg_A$ satisfy hyperdescent for the flat topology on animated rings?
\end{question}

To collect the results we need a consequence of the (comonadic) Barr-Beck-Lurie theorem \cite[Theorem 4.7.5.2]{HA} which we will recall below Proposition \ref{prop: criteria for descent} for the reader's convenience. But first we recall the Beck-Chevalley conditions to define the right adjointability conditions we will need. 

\begin{definition}\label{beck-chevalley}(The Beck-Chevalley condition)
A diagram of $\infty$-categories $$
\begin{tikzcd}
\cC \arrow[r, "f"] \arrow[d, "g"'] & \cD \arrow[d, "g'"] \\
\cC' \arrow[r, "f'"]               & \cD'               
\end{tikzcd}$$ commuting up to specified homotopy $\alpha\colon f'\circ g\simeq g'\circ f$ is called \emph{right adjointable} if $f$ and $f'$ admit left adjoints $f_R$ and $f'_R$ so that the right \emph{Beck-Chevalley transformation} $$\beta\colon f'_L\circ g'\to f'_L\circ g'\circ f\circ f_L\mathrel{\overset{\alpha}{\simeq}} f'_L\circ f'\circ g\circ f_L\to g\circ f_L$$ is an isomorphism.
\end{definition}

\begin{example}\label{modules beck-chevalley}
    Let $$
% https://tikzcd.yichuanshen.de/#N4Igdg9gJgpgziAXAbVABwnAlgFyxMJZABgBpiBdUkANwEMAbAVxiRAEEQBfU9TXfIRQBGclVqMWbAELdeIDNjwEiZYePrNWiDgHI5fJYKKj11TVJ3T9XcTCgBzeEVAAzAE4QAtkjIgcEEjCPG6ePoh+AUgATCEgHt4x1FGIAMxxCeGi-oFptlxAA
\begin{tikzcd}
A \arrow[r] \arrow[d] & B \arrow[d] \\
A' \arrow[r]          & B'         
\end{tikzcd}$$ be a pushout diagram of animated rings. Since the forgetful functor $\Theta\colon \aring \to \mathrm{CAlg}_\ZZ^{\einf}$ of \ref{forgetful} commutes with all small colimits by \cite[Proposition 25.1.2.2]{Lur18}, this is also a pushout diagram in $\einf$-rings. Thus the induced diagram $$
\begin{tikzcd}
\Mod_A \arrow[r] \arrow[d] & \Mod_{B} \arrow[d] \\
\Mod_{A'} \arrow[r]        & \Mod_{B'}         
\end{tikzcd}$$
with the arrows implemented via pullbacks is right adjointable by \cite[Lemma D.3.5.6]{Lur18}. Indeed, the right adjoint is given by the obvious forgetful functor.
\end{example}

\begin{lemma}\label{lem: adjointable dalg}
    Consider the pushout diagram of animated rings in Example \ref{modules beck-chevalley}. Then the induced diagram of derived algebras $$
\begin{tikzcd}
\dalg_A \arrow[r] \arrow[d] & \dalg_{B} \arrow[d] \\
\dalg_{A'} \arrow[r]        & \dalg_{B'}         
\end{tikzcd}$$ is right adjointable. 
\end{lemma}

\begin{proof}
This follows immediately from \cite[Remark 4.2.29]{Rak20} but we spell out the details for the convenience of the reader.
For a morphism of animated (in fact, derived) rings $R\to S$, the forgetful functor $\dalg_S\to \dalg_R$ admits a left adjoint given by $-\otimes_R S$. The fact that the Beck-Chevalley conditions of Definition \ref{beck-chevalley} are satisfied can be checked after applying the forgetful functor $\dalg_R\to \Mod_R$ which is conservative. Indeed, in the latter case right adjointability of the diagram in Example \ref{modules beck-chevalley} solves the problem.   
\end{proof}

Our descent statement will take the form of displaying the target $\infty$-category as a limit of cosimplicial $\infty$-categories. For this we recall the comonadic Barr-Beck-Lurie theorem below.

\begin{variant}[Completed adjointability]\label{variant: completed adjointability}
Fix a base ring $A$ and $J\subset \pi_0(A)$ a finitely generated ideal of $A.$ Consider the pushout diagram of animated $A$-algebras given as 
$$% https://tikzcd.yichuanshen.de/#N4Igdg9gJgpgziAXAbVABwnAlgFyxMJZABgBpiBdUkANwEMAbAVxiRACUQBfU9TXfIRQBGclVqMWbAMrdeIDNjwEiZYePrNWiDgHI5fJYKKj11TVJ3T9XcTCgBzeEVAAzAE4QAtkjIgcEEjCPG6ePoh+AUgATCEgHt4x1FGIAMxxCeGi-oFptlxAA
\begin{tikzcd}
R \arrow[r] \arrow[d] & S \arrow[d] \\
R' \arrow[r]          & S'         
\end{tikzcd}.$$

Consider the diagram of $\icats$ where the arrows are given by completed base change 

$$% https://tikzcd.yichuanshen.de/#N4Igdg9gJgpgziAXAbVABwnAlgFyxMJZABgBpiBdUkANwEMAbAVxiRAB12BZaAPU4BWAYwgBbNAH0ASiAC+pdJlz5CKAIzkqtRizaceUCQGV+7YWLRyFIDNjwEiZNVvrNWiDt2gTgUgOSypubiVop2KkQaztSuuh763sBGAUEiIbJaMFAA5vBEoABmAE5iSGQgOBBIavKFJaJl1JVIAEy1IMWliC1NVYgAzO2dDYgaFX2DFLJAA
\begin{tikzcd}
\Mod^\jcomp_R \arrow[r] \arrow[d] & \Mod_S^\jcomp \arrow[d] \\
\Mod_{R'}^\jcomp \arrow[r]        & \Mod_{S'}^\jcomp       
\end{tikzcd}$$ is right adjointable. 
This follows from Example \ref{modules beck-chevalley} because completion is symmetric monoidal. 

Note that as a general remark for any morphism $R\to S$ of $A$-algebras, the forgetful functor $$\Mod_S^\jcomp\to \Mod_R^\jcomp$$ preserves completions. Indeed a $J$-complete derived $R$ module $M$ satisfies 
    $$M=\lim_n M\otimes_R R\qq (f^n_1,\ldots, f^n_r)=\lim_n M\otimes_R R\otimes_A A\qq (f^n_1,\ldots, f^n_r)=\lim_n M\otimes_A A\qq (f^n_1,\ldots, f^n_r)$$ and the forgetful functor preserves limits. 

As a consequence the analogous diagram of completed derived algebras
$$
\begin{tikzcd}
\dalg^\jcomp_R \arrow[r] \arrow[d] & \dalg_S^\jcomp \arrow[d] \\
\dalg_{R'}^\jcomp \arrow[r]        & \dalg_{S'}^\jcomp       
\end{tikzcd}$$
 is right adjointable by the same reasoning as in Lemma \ref{lem: adjointable dalg}, applied in the completed case using conservativity of the forgetful functor $\dalg_A^\jcomp\to \Mod_A^\jcomp$ of Proposition \ref{prop: dalgj is monadic.} .   
\end{variant}

\begin{proposition}\label{prop: criteria for descent}
    Let $\cC^\bullet\colon \mathrm{N}(\mathbf{\Delta^+})\to \cC\mathrm{at}_\infty$ be an augmented cosimplicial $\infty$-category. Set $\cC=\cC^{-1}$ and let $G\colon \cC\to \cC^{0}$ be the obvious functor. Assume that:
    \begin{enumerate}
        \item The $\infty$-category $\cC$ admits geometric realisations of $G$-split simplicial objects and those geometric realisations are preserved by $G$.
        \item For every morphism $\alpha\colon [m]\to [n]$ in $\mathbf{\Delta}_+,$ the diagram $$
\begin{tikzcd}
\cC^m \arrow[d] \arrow[r, "d^0"] & \cC^{m+1} \arrow[d] \\
\cC^n \arrow[r, "d^0"]           & \cC^{n+1}          
\end{tikzcd}$$
        is right adjointable.

        Then the canonical map $\theta\colon \cC\to \lim_{n\in \mathbf{\Delta}} \cC^n$ admits a fully faithful right adjoint. If $G$ is conservative, then $\theta$ is an equivalence.
    \end{enumerate}
\end{proposition}
\begin{proof}
    This is the comonadic variant of the statement proven in \cite[Corollary 4.7.5.3]{HA}. See \cite[Corollary 4.8]{Moulinos2021GeometryOfFiltrations} for an explanation. 
\end{proof}

We will verify that the functor $\dalg_{(-)}\colon \aring\to \catinf$ satisfies the conditions of Proposition \ref{prop: criteria for descent} when evaluated on Cech nerves of faithfully flat maps of animated rings.
 
\begin{theorem}\label{flat descent}
    The  $\catinf$ valued functor which sends an animated ring $A\to \dalg_A$ satisfies descent for the faithfully flat topology on $\aring$.
\end{theorem}

\begin{proof}
We will verify the conditions in Proposition \ref{prop: criteria for descent}. Let $f\colon A\to B$ be a morphism of animated rings. Let $B^\bullet$ be the \v{C}ech nerve of $f$. Then we have a base change functor $\mathrm{coaug}\colon \dalg_A\to \dalg_{B^\bullet}$ sending a derived $A$-algebra $R$ to the algebra $R\otimes_A B^\bullet$ giving us a cosimplicial augmented diagram in $\catinf$. The adjointability conditions in Proposition \ref{prop: criteria for descent} follow from Lemma \ref{lem: adjointable dalg}. Now it remains to check conditions $(1)$ and $(2)$ of the aforementioned propositions. But they can be checked by conservativity of the functor $\dalg_A\to \Mod_A$ via point $(c)$ of \cite[Notation 4.2.28]{Rak20}.

\end{proof}

\begin{remark}\label{dalg is symmetricmonoidal}
    In fact the functor $A\mapsto \dalg_A$ is clearly valued in $\Calg(\Pr^L)$ and the base change functors are symmetric monoidal. Since limits are preserved along the forgetful functors $\Calg(\Pr^L)\to \Pr^L\to \catinf$ we conclude that the flat descent holds as presentably symmetric monoidal categories.
\end{remark}

\begin{remark}\label{lem: dalg evaluated in cech}
   Now the result in Theorem \ref{flat descent} shows that if $X\to Y$ is an effective epimorphism in the $\infty$-topos $\Shv(\aring_{\fpqc})$  so that $Y=\colim(\mathrm{Cech}(X\to Y))$, then $$\dalg(Y)=\lim \dalg(\mathrm{Cech}(X\to Y))$$ where the right hand side is shorthand for a cosimplicial $\infty$-category obtained by evaluating $\dalg$ on terms of the simplicial object $\mathrm{Cech}(X\to Y)$, where $\dalg$ is evaluated on stacks via right Kan extensions as defined in Construction \ref{cons: right kan extension of dalg}.
   The equivalence holds in $\Calg(\Pr^L)$ again by Remark \ref{dalg is symmetricmonoidal}.
   A proof of why flat descent for sheaves valued in presentable $\infty$-categories implies the Cech nerve criteria above, can be found in \cite[Proposition 3.48.]{Yaylali2022NotesDAG} after noting that $\catinf$ is presentable which can be checked by its explicit presentation as a combinatorial simplicial model category by \cite[Proposition 3.1.5.2]{HTT}.
\end{remark}

We also record a corollary which we won't use but is useful to know. This is the derived algebra variant of \cite[Footnote 12]{BS19}. This is the promised point $(6)$ of Goal \ref{goal: purpose of dalgj}.

\begin{corollary}
    Let $A$ be an animated ring and $J\subset \pi_0(A)$ a finitely generated ideal. Then the functor $B\mapsto \dalg^\jcomp_B$ satisfies descent for the $J$-completely flat topology on $\aring_A.$ 
\end{corollary}
\begin{proof}
Write $\mathrm{Cech}(B\to C)$ for the Cech nerve of $B\to C$. We want to show that 

$$\dalg^\jcomp_B=\lim \dalg^\jcomp_{\mathrm{Cech}{(A\to B)}}$$

Fix generators $J=(x_1,\ldots, x_r).$ 

A map of animated $A$-algebras $\varphi\colon B\to C$ is $J$-completely faithfully flat if and only if $\varphi_1\colon B/\!\!/(x_1,\ldots, x_r)\to C/\!\!/(x_1,\ldots, x_r)$ is faithfully flat as explained in \ref{jcompflat}. 

This implies that for each $n\geq 1$ the maps $\varphi_n\colon B\qq (x_1^n,\ldots, x_r^n)\to C\qq (x_1^n,\ldots, x_r^n)$ are faithfully flat by \cite[Lemma 5.2.2]{CesnaviciusScholze2024FlatPurity}.

Thus by flat descent of $\dalg_{(-)}$ it follows that $$\dalg_{B\qq (x_1^n,\ldots, x_r^n)}=\llim \dalg_{\mathrm{Cech}(A\to B)\qq (x_1^n,\ldots, x_r^n)},$$

Note that by Corollary \ref{cor: dalg on spf} we have $\dalg_B^J=\llim_n \dalg_{B/\!\!/(x^n_1,\ldots, x^n_r)}$ 
The induced map  is faithfully flat for each $n\geq 0$. But since limits commute with limits, we conclude.
    
\end{proof}
\subsection{Derived algebras along quasi-affine morphisms.}\label{sec: dalg along qaff}

We now begin with the main classical results we need. Our work here has a mostly utilitarian flavour. We will remark on this as we get to the needed results. 

\begin{construction}\label{cons: lax monoidal right adjoint}
    Let $f\colon X\to Y$ be a morphism of prestacks. Then the functor $f^*\colon \Qcoh(Y)\to \Qcoh(X)$ is symmetric monoidal and preserves all colimits and therefore admits a lax monoidal right adjoint $f_*\colon \Qcoh(X)\to \Qcoh(Y)$, for example by \cite[Proposition 3.2.1]{Lurie2011DAGVIII}.

     It follows that there is a factorization of $f_*$ as 
    \begin{equation}\label{eq: factorisation fOX}
            \Qcoh(X)\xrightarrow{\tilde{f_*}}\Mod_{f_*\cO_X}(\Qcoh(Y))\to \Qcoh(Y).
\end{equation}

 On the other hand, the adjunction $f^*\dashv f_*$ induces a monad $f_*f^*$ acting on $\Qcoh(Y).$ It follows formally that there is also a factorization 

\begin{equation}\label{eq: factorisation ff*OX}
     \Qcoh(X)\xrightarrow{\tilde{\tilde{f_*}}}\Mod_{f_*f^*}(\Qcoh(Y))\to \Qcoh(Y).
\end{equation}

 The counit $f^*f_*\to \id$ induces a morphism of monads $-\otimes_{\cO_Y}f_*\cO_X\to f_*f^*$ whence a forgetful functor 

\begin{equation}\label{eq: final factorisation}
     \Qcoh(X)\xrightarrow{\tilde{\tilde{f_*}}}\Mod_{f_*f^*}(\Qcoh(Y))\xrightarrow{U} \Mod_{f_*\cO_X}(\Qcoh(Y))\to  \Qcoh(Y),
\end{equation}
 where the composite of $U\circ \tilde{\tilde{f_*}}$ is $\tilde{f_*}$ of \ref{eq: factorisation fOX}.

\end{construction}

\begin{remark}
The functor $f_*$ described above is not very well behaved: for example it does not satisfy a projection formula. 
Indeed, one sufficient condition for the morphism $U$ of \ref{eq: final factorisation} to be an equivalence is for the morphism of monads $-\otimes_{\cO_Y}f_*\cO_X\to f_*f^*$ to be an equivalence.
We next isolate a class of morphisms when the projection formula is satisfied. We follow \cite[\S~3]{GaitsgoryRozenblyum2017SAGI}.
\end{remark}

\begin{definition}\cite[\S~3.3]{GaitsgoryRozenblyum2017SAGI}\label{def: quasi-affine}
    A derived scheme is called quasi-affine if it is quasi-compact and admits an open immersion into an affine derived scheme.
\end{definition}

\begin{definition}\label{def: quasi-affine etc}
Let $f\colon X\to Y$ be a morphism of prestacks.
\begin{enumerate}
    \item  We say that $f$ is \emph{schematic} if for any affine derived scheme $\spec(R)\to Y$, the base change $X\times_{Y}\spec(R)$ is representable by a derived scheme.
    \item  We say that $f$ is \emph{quasi-compact} schematic if it is schematic and in addition $X\times_{Y}\spec(R)$ is a quasi-compact.
    \item We say that $f$ is \emph{quasi-affine} if it is schematic and $X\times_{Y}\spec(R)$  is a quasi-affine scheme.
\end{enumerate}
   
\end{definition}

We now collect some results from \cite{GaitsgoryRozenblyum2017SAGI} which are direct consequences of Definition  above\footnote{Note that \cite{GaitsgoryRozenblyum2017SAGI} has a running characteristic $0$ hypothesis, as their intended applications are to geometric Langlands in characteristic $0$. However, the categorical statements in Proposition \ref{prop: GR facts we need} do not need any restriction on the base ring.}. 

\begin{proposition}\label{prop: GR facts we need}
    Let $f\colon X\to Y$ be a quasi-affine morphism of prestacks then 
\begin{enumerate}
    \item The morphisms of monads $-\otimes_{\cO_Y}f_*\cO_X\to f_*f^*$ is an isomorphism,
    \item The morphism $U$ of \ref{eq: final factorisation} is an isomorphism.
    \item The morphism $\tilde{\tilde{f_*}}$ of the factorisation of \ref{eq: factorisation ff*OX} is an equivalence.
    \item For any prestack $g:Z\to Y$, giving a pullback diagram $$% https://tikzcd.yichuanshen.de/#N4Igdg9gJgpgziAXAbVABwnAlgFyxMJZARgBoAGAXVJADcBDAGwFcYkQANEAX1PU1z5CKMsWp0mrdgE0efEBmx4CRcqTE0GLNohAAtOfyVDVFcVqm6OAHWt4AtvAD60gAQHu4mFADm8IqAAZgBOEPZIaiA4EEhkEtrsgYYgIWFIAEw00bGakjogPiA0jPQARjCMAAoCysIgwVg+ABY4yanhiADMWTGImSAl5VU1JroNza25CbqBAORtoR3dUb2RFvk+857cQA
\begin{tikzcd}
X\times_Y Z \arrow[d, "f'"'] \arrow[r, "g'"] & X \arrow[d, "f"] \\
Z \arrow[r, "g"']                            & Y               
\end{tikzcd},$$ the induced (rotated) diagram of $\icats$ 
$$% https://tikzcd.yichuanshen.de/#N4Igdg9gJgpgziAXAbVABwnAlgFyxMJZARgBoAGAXVJADcBDAGwFcYkQAdDgRQGMIAFgAoAGgEoQAX1LpMufIRRli1Ok1bsufQaK54AtvAD6ATQAEALQnTZ2PASLlSKmgxZtEnHv2FWpMkAw7BUcKVTcNTy0fIRNrVRgoAHN4IlAAMwAnCH0kJxAcCCQAZld1DxB0gD0AKhAaRnoAIxhGAAU5e0UQTKwkgRx-DOzcxAAmGkKkMhABGHoodkgwNhtKkenJovGy93YhdIByMVqh9ZySraQJkEaW9s6Qz17+wd3IkCTTtayLxHypogZhEKkIksdvpRJEA
\begin{tikzcd}
\Qcoh(Y) \arrow[d, "g^*"']  & \Qcoh(X) \arrow[l, "f^*"'] \arrow[d, "g'^*"] \\
\Qcoh(Z) \arrow[r, no head] & \Qcoh(X\times_Y Z) \arrow[l, "f'^*"]        
\end{tikzcd}$$ is right adjointable. 
\end{enumerate}
\end{proposition}

\begin{proof}
    \begin{enumerate}
        \item This amounts to the projection formula where this is \cite[Lemma 3.2.4]{GaitsgoryRozenblyum2017SAGI}.
        \item This follows from the previous item.
        \item This is \cite[Proposition 3.3.3]{GaitsgoryRozenblyum2017SAGI}.
        \item This amounts to showing that the canonical morphism $g^*\circ f_*\to f'_*\circ g'^*$ is an isomorphism, where this is  \cite[Proposition 2.2.2.]{GaitsgoryRozenblyum2017SAGI}.
    \end{enumerate}
\end{proof}

\begin{remark}\label{remark: pushforward also work for formal stacks}
    Note that the results of Proposition \ref{prop: GR facts we need} are also valid in the $J$-adic formal setting of Remark \ref{remark: definition of formal stack}. Indeed the statement of Proposition \ref{prop: GR facts we need} is at the level of prestacks and therefore works as well for $J$-completed formal stacks.
\end{remark}

We can now prove the analogs of the above statements for derived algebras. However, before we do so we remark on the limitations of what we are doing and what it would be desirable to do. 

\begin{remark}\label{remark: limitations of quasi-affine}
    Definition \ref{def: quasi-affine etc} works with a notion of \emph{schematic} quasi-affine morphisms. In \cite[Definition 3.1.24]{Lurie2011DAGVIII} defines quasi-affine morphisms of \emph{relatively} Deligne-Mumford stacks. In characteristic $0$, where one can work with $\einf$-rings instead of derived rings, this yields a larger class of morphisms even in the derived setting.
    Moreover in characteristic $0$, one can deduce from \cite[Theorem 2.3]{BhattHalpernLeistner2017TannakaRevisited}, a classification of quasi-affine relatively Deligne-Mumford stacks as compact localizations of commutative algebras in quasi-coherent sheaves. 
    It is natural to wonder whether the theory of derived algebras yields a refinement of these concepts in the setting of derived algebraic geometry. 
    
    This is under investigation by the current author in collaboration with J. Mundinger \cite{MundingerSahai_inprep}
\end{remark}

\begin{lemma}\label{lem: pushforwad preserves dalg}
Let $f\colon X\to  Y$ be a morphism of stacks. Then $f_* \colon \Qcoh(X) \to  \Qcoh(Y)$ admits a natural enhancement to $f^\dalg_*\colon \dalg(X)\to \dalg(Y )$ so that the following diagram commutes

$$
\begin{tikzcd}
\dalg(X) \arrow[r, "f^\dalg_*"] \arrow[d, "U_X"'] & \dalg(Y) \arrow[d, "U_Y"] \\
\Qcoh(X) \arrow[r, "f_*"]                 & \Qcoh(Y)               
\end{tikzcd}
$$
\end{lemma}
\begin{proof}
    This is \cite[Proposition 2.13]{petrov_de_rham}. We recall their proof for the convenience of the reader. First we prove the claim for the functor $f^*\colon \Qcoh(Y)\to \Qcoh(X).$ Note that by naturality of the functors $\L\Sym^n_X\circ f^*=f^*\circ \L\Sym^n_X$, there is an induced colimit preserving functor $f^{*,\dalg}\colon \dalg(Y)\to \dalg(X)$. Indeed the functor preserves colimits because it preserves sifted colimits, which can be detected after passing to $\Qcoh$ and it preserves coproducts, which can be detected after passing to $\Calg.$
    On the other hand, since both $\dalg(X)$ and $\dalg(Y)$ are presentable (by presentability  of the corresponding quasi-coherent categories and \cite[Proposition 4.1.10]{Rak20}), we see that there is a right adjoint $f_*^\dalg\colon \Qcoh(X)\to \Qcoh(Y).$
    It remains to check that $ U\circ f_*\simeq U\circ f_*^\dalg$. To do so, note that the functors in question are the right adjoints of the corresponding functors $\L\Sym\circ f^* \simeq f^{*,\dalg}\circ \L\Sym$.

\end{proof}

\begin{notation}\label{notation: abuse of notation for pushforward}
      To avoid clutter we denote the induced functor $$f_*^\dalg\colon \dalg(X)\to \dalg(Y)$$ by $f_*$ as in our applications the refinement will be clear from context. When we need, for example in proofs, we will revert to the notation of Lemma \ref{lem: pushforwad preserves dalg}.
\end{notation}

We may now begin to prove the desired results of this section. These results will be analogs of Proposition \ref{prop: GR facts we need}.

\begin{proposition}\label{cor: quasi-aff preserves dalg}

    Let $f\colon X\to Y$ be a quasi-affine morphism of prestacks. Then there is an equivalence of $\icats$
    $$\dalg(X)\simeq \dalg(Y)_{f_*\cO_X/}.$$
\end{proposition}
\begin{proof}
For clarity, we revert to the notation employed in the proof of Lemma \ref{lem: pushforwad preserves dalg} i.e. we denote the adjunction on derived algebras as $f^{*,\dalg}\dashv f_*^\dalg. $

First observe that the induced functor $f^\dalg_*\colon \dalg(X)\to \dalg(Y)$ is monadic. Indeed this follows from dual to the reasoning employed in Lemma \ref{monadicity in dalg}, but we spell it out. 

Using the coarse Barr-Beck criteria \cite[Theorem 4.7.0.3]{HA}, we need to show that $f^\dalg_*$ preserves is conservative and  preserves geometric realizations. 

For conservativity, note that $f_*\circ U_X=U_Y\circ f^\dalg_*$ is conservative by point $(3)$ of Proposition \ref{prop: GR facts we need} and conservativity of $U_X$. Since $U_Y$ is conservative, then so is $f_*.$ 

For preservation of geometric realisations, let $X^\bullet\in \dalg(X)$ be a simplicial object and $|X^\bullet|$ be its geometric realisation. We need to compute $f_*^\dalg(|X^\bullet|).$ Since the functor $U_Y$ preserves geometric realisations and is conservative, it suffices to show that $U_Y\circ f_*^\dalg|X^\bullet|$ commutes with geometric realisations. But $U_Y\circ f_*^\dalg=f_*\circ U_X $ and both of the functors on the right hand side preserve geometric realisations. 

Thus we have a canonical identification 
$$\dalg(X)\simeq \Mod_{f_*^\dalg f^{*,\dalg}}(\dalg(Y)).$$

Now note that there is an identification of monads $$(-)\otimes_{\cO_Y}f_*^\dalg\cO_X\simeq f_*^\dalg f^{*,\dalg}$$. Indeed, note that there is an obvious morphism of monads from the left side to the right side, which can be checked to be an isomorphism using conservativity of $U_Y$ using item $(1)$ of Proposition \ref{prop: GR facts we need}.

Thus we have a further equivalence of $\icats $
$$\dalg(X)\simeq  \Mod_{-\otimes_{\cO_Y}f_*^\dalg \cO_X}(\dalg(Y)).$$

It remains to show that $$\Mod_{-\otimes_{\cO_Y}f_*^\dalg \cO_X}(\dalg(Y))\simeq \dalg(Y)_{f^\dalg_*\cO_Y/}.$$

For this note that given any derived algebra $A\in \dalg(Y)$ there is a functor $\dalg(Y)_{A/}$ given by $-\otimes_{\cO_Y}A$ and it has a right adjoint given by the forgetful functor $U\colon \dalg(Y)_{A/}\to \dalg(Y)$. It suffices to show that this functor is monadic. Since the forgetful functor $\dalg(Y)\to \Calg(Y)$ preserves all limits and colimits, the monadicity statement follows from the analogous one for $\Calg(Y).$ 
Since the functors are obviously equivalent (even as $\mathbf{E}_1$-algebras in $\mathrm{End}(\dalg(Y))$ we conclude.
\end{proof}

\begin{remark}
    We view \ref{cor: quasi-aff preserves dalg} as a primitive derived analog of analog of \cite[Theorem 2.3]{BhattHalpernLeistner2017TannakaRevisited}. We expect to examine this in more detail in joint work with J. Mundinger \cite{MundingerSahai_inprep}.
\end{remark}

\begin{proposition}[Base change for derived algebras]\label{prop: derived base change for algebras}
    Let $f\colon X\to Y$ be a morphism of prestacks and $g\colon Z\to Y$ any morphism given a pullback diagram
   $$ \begin{tikzcd}
X\times_Y Z \arrow[d, "f'"'] \arrow[r, "g'"] & X \arrow[d, "f"] \\
Z \arrow[r, "g"']                            & Y               
\end{tikzcd},$$ the induced (rotated) diagram of $\icats$ 
$$% https://tikzcd.yichuanshen.de/#N4Igdg9gJgpgziAXAbVABwnAlgFyxMJZARgBoAGAXVJADcBDAGwFcYkQAdDgRQGMIAFgAoAGgEoQAX1LpMufIRRli1Ok1bsufQaK54AtvAD6ATQAEALQnTZ2PASLlSKmgxZtEnHv2FWpMkAw7BUcKVTcNTy0fIRNrVRgoAHN4IlAAMwAnCH0kJxAcCCQAZld1DxB0gD0AKhAaRnoAIxhGAAU5e0UQTKwkgRx-DOzcxAAmGkKkMhABGHoodkgwNhtKkenJovGy93YhdIByMVqh9ZySraQJkEaW9s6Qz17+wd3IkCTTtayLxHypogZhEKkIksdvpRJEA
\begin{tikzcd}
\dalg(Y) \arrow[d, "g^*"']  & \dalg(X) \arrow[l, "f^*"'] \arrow[d, "g'^*"] \\
\dalg(Z) \arrow[r, no head] & \dalg(X\times_Y Z) \arrow[l, "f'^*"]        
\end{tikzcd}$$ is right adjointable. 
\end{proposition}
\begin{proof}
    In the notation of the proof of Lemma \ref{lem: pushforwad preserves dalg}, it suffices to prove that the induced map 
    $$g^{*,\dalg}\circ f_*^\dalg\to f_*'^{,\dalg}\circ g'^{*,\dalg}$$ is an equivalence. This may checked at the level of modules where this is item $(4)$ of Proposition \ref{prop: GR facts we need}.
\end{proof}

\begin{remark}\label{remark: completed base change holds for formal stacks}
    As explained in Remark \ref{remark: pushforward also work for formal stacks}, we note that Lemma \ref{lem: pushforwad preserves dalg}, Corollary \ref{cor: quasi-aff preserves dalg} and Proposition \ref{prop: derived base change for algebras} work in the $J$-adic setting of Remark \ref{remark: definition of formal stack}. Indeed they only reference the underlying prestack of a formal stack. We give an example making this explicit.
\end{remark}

\begin{example}\label{ex: pushforward on spf}
  Let us give an example of Proposition \ref{prop: derived base change for algebras} in a case of interest. Let $A$ be an animated ring and let $J\subset \pi_0(A)$ be a finitely generated ideal. Let us consider the pushout square of formal $A$-algebra stacks in the sense of Remark \ref{remark: definition of formal stack} induced by a corresponding morphism of $A$-algebras
$$
  % https://tikzcd.yichuanshen.de/#N4Igdg9gJgpgziAXAbVABwnAlgFyxMJZABgBpiBdUkANwEMAbAVxiRAB1240AzACgAinAO5ZYACzo5gnCHgC28AL4B9AEIBhAJQglpdJlz5CKAIzkqtRizadu-bbv0gM2PASJlTl+s1aIOLl5BHT0DN2Mic29qXxsAu2C1UMsYKABzeCJQHgAnCHkkMhAcCCRTMJA8gqLqUqQAJkrqwsQGurLEAGZm-NbzEs6eiiUgA
\begin{tikzcd}
\spf(D\widehat{\otimes}_BC) \arrow[r] \arrow[d] & \spf(C) \arrow[d] \\
\spf(D) \arrow[r]                               & \spf(B)          
\end{tikzcd}$$

Note that such a morphism is a quasi-affine morphism, since, for example, it is base changed from a corresponding morphism of affine schemes along $\spf(A)\to \spec(A).$

Then the induced right adjointability statement of Proposition \ref{prop: derived base change for algebras} gives a diagram $$% https://tikzcd.yichuanshen.de/#N4Igdg9gJgpgziAXAbVABwnAlgFyxMJZABgBpiBdUkANwEMAbAVxiRAB12pGBzAPWCcAVgGMIAWzQBfAPrAAIpwDuWWAAs6OQewh5x8WQCEAwlJBTS6TLnyEUARnJVajFm07cGPGcfOWQGNh4BERk9s70zKyIHFy8fMJikjLyflZBtkSO4dSRbjEe8YkSaDKG5s4wUDzwRKAAZgBOEkhkIDgQSABMFg3N4kiO7Z2IAMy9IE0tiEMdrbmu0SAAtMqqMBpanLpY+nCyxgAEABSK7Crqmto7e0aHxgCUINQMdABGMAwACtbBdiCNLA8NQ4NKTfpIUbUOaILoLKJsVbndaba56AxlQ6pKQUKRAA
\begin{tikzcd}
\dalg^{\jcomp}_{D\widehat{\otimes}_BC} \arrow[d] & \dalg_C \arrow[d] \arrow[l, "-\widehat{\otimes}_C (D\widehat{\otimes}_B C)"'] \\
\dalg^\jcomp_D                                   & \dalg^\jcomp_B \arrow[l, "-\widehat{\otimes}_B D"]                           
\end{tikzcd}.$$

Indeed the horizontal arrows come from Remark \ref{remark: functoriality of spf} and the vertical ones are the natural forgetful functors which preserve complete objects as explained in Variant \ref{variant: completed adjointability}.
\end{example}

\subsection{Derived algebras and the Rees construction.}\label{ssec: rees construction section}

In this section our goal is to explain the functorialities of derived algebras under the Rees construction from \cite[\S~5]{Simpson1997HodgeFiltrationNonabelian}, \cite[\S~2.2.1]{Bhatt22} and \cite[\S~3.1]{Lurie2015RotationInvarianceKTheory}.

We will follow the last reference above (even though it works over the sphere spectrum $\sph$) as the author actually doesn't know a reference which proves the derived statement directly

\begin{remark}(What we mean by the Rees construction).\label{remark: what we mean by rees}
    There seems to be some discrepancy in what gets called the Rees construction in the literature. For example, working over a discrete ring $A$, in \cite{Bhatt22}, Bhatt calls the equivalence $\fil\Mod_R\simeq \Qcoh(\agm)$ the Rees equivalence. This is fine since over a discrete ring the pushforward of $\cO_{\agm}$ along the structure morphism $p\colon \agm\to \bgm$ is almost by definition identified with the $\G_m$-equivariant module $k[t]$ with $t$ in graded degree $1.$ However to do this one has to implicitly fix an equivalence of $\icats$ $\Qcoh(B\G_m)\simeq \Gr\Mod_R.$ In our derived setting it is useful to work as canonically as possible and so to keep the last two categories separate until we canonically identify them in the next section. 

   Therefore, for us, the Rees construction will simply mean an equivalence between filtered modules over a filtered ring and graded modules over its Rees construction (which we will explain in the sequel). The equivalence with stack theory will be explained in the next section.
\end{remark}.

\begin{recollection}[Lurie's Rees construction]\label{rec: luries rees construction} We recall Lurie's Rees construction from  \cite[3.1]{Lurie2015RotationInvarianceKTheory}. We work universally over the sphere spectrum $\sph$.  

Note that there's an evident inclusion of $\icats$ $\ZZ^\ds\into \ZZ^\op$\footnote{The source has the same objects as the target. So the morphism is given by sending objects to themselves and the morphisms to identity morphisms.} which induces a functor\footnote{We preemptively apologise to the reader for using the notation $\res$ and $\rees$ together to mean different things.} $$\res\colon \fil\Sp\to \Gr\Sp.$$

The unit $\sph\{0\}$ gets sent to the graded object $\sph[t]$ which in degree $n$ is given by $\sph$ if $n\leq 0$ and $0$ otherwise. More generally a filtered spectrum $X^\bullet$ is sent to the $\sph[t]$-module $\mathrm{Res}(X^\bullet)$ which in graded degree $n$ is given by $X^nt^{-n}$. 

By left Kan extensions, the functor $\res$ has a left adjoint $\splt\colon \Gr\Sp\to \fil\Sp$ sending a graded spectrum $X^\bullet$ to the filtered object $\splt(X)^\bullet$ so that $\splt(X)^i=\bigoplus_{j\geq i} X^j.$ Since $\splt$ is clearly symmetric monoidal, it follows that $\res$ is lax monoidal whence preserves commutative algebras. 

Identifying the source of $\res$ with $\Mod_{\sph\{0\}}(\fil\Sp)$ there is an induced functor $$\rees\colon \Mod_{\sph\{0\}}(\fil\Sp)\to \Mod_{\sph[t]}(\Gr\Sp)$$ which Lurie shows in \cite[Proposition 3.1.6]{Lurie2015RotationInvarianceKTheory} is in fact a symmetric monoidal equivalence.

If $R$ is any $\einf$-ring then tensoring the above equivalence with $\Mod_R$\footnote{Or in fact any other $\cC\in \Calg(\Pr^L_{\st})$.} and using \cite[Theorem 4.8.5.16]{HA}, one obtains an equivalence

\begin{equation}\label{eq: rees at filgraded}
  \rees\colon \fil\Mod_R\simeq \Mod_{R[t]}(\Gr\Mod_R)  
\end{equation}
The same equivalence holds when $R$ is an animated ring and $R[t]$ may be identified with the polynomial ring $R[t]=R\otimes_{\ZZ}\ZZ[t]\simeq R\otimes_{\sph} \Sigma_{+}^\infty(\mathbf{N}).$

\end{recollection}

\begin{remark}\label{remark: base change on rees}
    
In the equivalence of \ref{eq: rees at filgraded} one heuristically writes a filtered $R$-modules $M^\bullet$ as $$\rees(M^\bullet)=\bigoplus_{i\in \ZZ}M^it^{-i}$$ and thinks of it as a module over the polynomial ring $R[t].$

The equivalence in \ref{eq: rees at filgraded} has the following features.

\begin{enumerate}
    \item We have the base change functors $R[t]\to R$ obtained by setting $t=0$, or equivalently thinking of $R$ in graded degree $0$ and one obtains the object $\rees(M^\bullet)\otimes_{R[t]}R\in \Gr\Mod_R$ which one thinks of as the associated graded of $M^\bullet.$
    \item Along the specialisation $R[t]\to R[t,t^{-1}]$ where the source is the graded ring with $R$ in all degrees, one obtains the module $\rees(M^\bullet)\otimes_{R[t]}R[t,t^{-1}]$ which inverts the action by $t$ on $\rees(M^\bullet)$ and one recovers $$\colim(\rees(M^\bullet)\xrightarrow{t} \rees(M^\bullet)\xrightarrow{t}\rees(M^\bullet)\xrightarrow{t}\ldots)\simeq \colim(\ldots \to M^i\to M^{i-1}\to M^{i-2}\to \ldots ),$$ and the right hand side is equivalent to $\mathrm{und}(M^\bullet)\in \Mod_R$, the underlying module of $M^\bullet$.
\end{enumerate}
\end{remark}

\begin{remark}\label{remark: all base changes preserve tensors}
    
Note that the specializations above are all implemented by taking the graded tensor product in $\Mod_{R[t]}(\Gr\Mod_R)$. Therefore if one were to apply the above equivalences to a filtered $\einf$-ring $E^\bullet$, then one obtains that the following three a priori modules 
\begin{enumerate}
    \item $\rees(E^\bullet),$
    \item $\rees(E^\bullet)\underline{\otimes}_{R[t]} R,$ and
    \item $\rees(E^\bullet)\underline{\otimes}_{R[t]}R[t,t^{-1}],$
\end{enumerate}

all are $\einf$-rings in their respective categories.
A similar remark shows preservation of derived algebra structures.
\end{remark}

Our goal now is to show that the Rees construction of Recollection \ref{rec: luries rees construction} and the base changes in Remark \ref{remark: base change on rees} are appropriately functorial in derived algebras.

\begin{remark}\label{remark: compact generators preserved on sph[t]}
We remark on a light conundrum. 
    For $n\in \ZZ$, there are functors $\ins^n\colon \Sp \to \fil\Sp$ sending a spectrum $X$ to the filtered spectrum $\ins^n(X)$ which is $X$ in filtered degree $m$ for all $m\leq n$ and zero otherwise. Note that $\sph$ projectively generates $\Sp$ by \cite[Corollary 7.1.4.15]{HA} and $\fil\Sp$ is projectively generated by finite coproducts of $\ins^n(\sph).$ 
    The morphism $\gr^*\colon \fil\Sp\to \Gr\Sp$ preserves projective generators. However the morphism $\res\colon \fil\Sp\to \Gr\Sp$ recalled in the first part of Recollection \ref{rec: luries rees construction}  doesn't. Indeed $\ins^0\sph$ gets sent to the graded polynomial ring $\sph[t]$ which is large (and not compact) inside $\Gr\Sp$. On the other hand, once one passes to modules over $\sph[t]$ in the target, the projective generators $\ins^n\sph$ become (graded shifted) free over $\sph[t]$ and indeed are preserved as projective generators. A similar remark holds after base change of $\ins^n$ along $\Mod_R$ for any $\einf$-ring $R.$ 
\end{remark}

\begin{lemma}\label{lem: res is contexts}
    Working over the discrete ring $\ZZ$ the equivalence $$\rees\colon \fil\Mod_\ZZ\simeq \Mod_{\ZZ[t]}(\Gr\Mod_\ZZ)$$ is an equivalence of derived algebraic context.
\end{lemma}
\begin{proof}
We need to check the conditions in Recollection \ref{rec: morphisms in contexts}.
    The functor $\rees\colon \filmod_\ZZ\to \grmod_\ZZ$ is an equivalence and therefore presrves all colimits. 
    The functor is evidently $t$-exact for the neutral $t$-structure on the source and target and sends compact generators of $\filmod_\ZZ$ to those of $\grmod_\ZZ$ (other than being an equivalence, the latter can also be seen from Remark \ref{remark: compact generators preserved on sph[t]}. 
    That it is symmetric monoidal is checked in the course of the proof of \cite[Proposition 3.1.6]{Lurie2015RotationInvarianceKTheory}, whence we conclude.
    \end{proof}

\begin{lemma}\label{lem: undercategory base change.} 
    There is a canonical equivalence $\dalg(\Mod_{\ZZ[t]}(\grmod_\ZZ))\simeq \dalg(\Gr\Mod_\ZZ)_{\ZZ[t]/}$ induced by restriction of scalars $\Mod_{\ZZ[t]}(\grmod_\ZZ)\to \Gr\Mod_\ZZ$.
\end{lemma}
\begin{proof}

The restriction of scalars $\Mod_{\ZZ[t]}(\grmod_\ZZ)\to \Gr\Mod_\ZZ$ is tautologically monadic with monad given by the graded tensor product $\ZZ[t]\underline{\otimes}_{\ZZ}-$. 
It follows that the induced functor $\dalg(\Mod_{\ZZ[t]}(\grmod_\ZZ))\to \dalg(\grmod_{\ZZ})$ is monadic since the forgetful functor to modules is conservative and preserves all sifted colimits (whence totalisations can be checked on underlying module categories). The monad identifies with $\ZZ[t]\underline{\otimes}_{\ZZ}-$, but this is precisely the monad for the adjunction $\dalg(\grmod_{\ZZ})_{\ZZ[t]/}\to \dalg(\Gr\Mod_\ZZ).$ 
\end{proof}

\begin{corollary}\label{cor: res gives derived algebras}
    The functor of Construction \ref{rec: luries rees construction} induces an equivalence $\rees\colon \fil\dalg_\ZZ\simeq \Gr\dalg_{\ZZ[t]}$
\end{corollary}
\begin{proof}
    In light of Lemma \ref{lem: undercategory base change.}, this follows from Lemma \ref{lem: res is contexts}.
\end{proof}

\begin{recollection}\label{rec: filtered derived algebras}
    Consider the $\icat$ of filtered $\ZZ$-modules $\fil\Mod_\ZZ$ with the neutral $t$-structure. Then $\fil\Mod_\ZZ$ is a derived algebraic context and one defines filtered derived rings as $\fil\dalg_{\ZZ}:=\dalg(\fil\Mod_\ZZ)$. The connective filtered derived algebras are defined as $$\fil\dalg_\ZZ^\cn:=\fil\dalg_{\ZZ}\times_{\fil\Mod_\ZZ}(\fil\Mod_{\ZZ})_{\geq 0}$$ and, in keeping with Remark \ref{remark: identify animated with connective dalg} we dub this $\icat$ the $\icat$ of filtered animated rings. 
\end{recollection}

\begin{remark}\label{remark: rees connective sent to connective}
    The restrion of the functor $\rees$ of Corollary\ref{cor: res gives derived algebras} $\rees\colon \fil\dalg_\ZZ^\cn\to \Gr\dalg_{\ZZ[t]}$ induces an equivalence with $\Gr\dalg_{\ZZ[t]}^\cn\into \Gr\dalg_{\ZZ[t]}$, with the source defined as in Remark \ref{remark: identify animated with connective dalg}. Indeed this follows from the $t$-exactness of $\rees.$ In particular why the fully faithful morphism $\ins^0\colon \dalg^\cn_\ZZ\to \fil\dalg^\cn_{\ZZ}$ we may view animated rings as filtered rings which under $\rees$ get sent to the graded polynomial ring $A[t].$
\end{remark}

\begin{remark}\label{cons: induced dalg on animated for rees}
    The morphism of Corollary \ref{cor: res gives derived algebras} we observe that for any $T\in \fil\dalg_\ZZ$ there is an induced equivalence $\rees\colon (\fil\dalg_\ZZ)_{T/}\simeq (\fil\dalg_{\ZZ[t]})_{\rees(T)/}.$

\end{remark}

\begin{remark}
    The next lemma is a standard argument which shows up, for example, in the proof of [Proposition 5.1]\cite{Moulinos2021GeometryOfFiltrations}, which we writed down for further reference.
\end{remark}

\begin{lemma}\label{lem: equivalence of operads}
    Let $F\colon \cC^\otimes\to \cD^\otimes$ be an equivalence of symmetric monoidal $\icats$. Let $A\in \Calg(\cC)$ and $F(A)\in \Calg(D)$ be its image. Then there is a symmetric monoidal equivalence of $\infty$-operads $$F\colon \Mod_A(C)^\otimes\simeq \Mod_{F(A)}(\cD)^\otimes.$$
\end{lemma}
\begin{proof}
    By \cite[Theorem 4.5.3.1.]{HA} there is a coCartesian fibration $p\colon \Mod(\cC)^\otimes\to \Calg(\cC)\times \mathrm{Fin}_*$ and similarly for $\cD.$ There is a commuting diagram of coCartesian fibration where both the bottom and top horizontal arrows are equivalences. 
\[\begin{tikzcd}[cramped]
	{\Mod(\cC)^\otimes} & {\Mod(\cD)^\otimes} \\
	{\Calg(\cC)\times \mathrm{Fin}_*} & {\Calg(\cD)\times \mathrm{Fin}_*}
	\arrow[from=1-1, to=1-2]
	\arrow[from=1-1, to=2-1]
	\arrow[from=1-2, to=2-2]
	\arrow[from=2-1, to=2-2]
\end{tikzcd}\]

Then by taking the fiber above $A\times \mathrm{Fin}_*$ and $F(A)\times \mathrm{Fin}_*$ we win. 
\end{proof}

\begin{remark}\label{remark: module over module is module}
    Note that for any filtered derived ring $S$ in $\fil\dalg_{\ZZ}$ and a derived algebra $T$ over $S$ i.e. $T\in \fil\dalg_S$ we have an equivalence $\Mod_T(\Mod_S(\fil\Mod_{\ZZ}))\simeq \Mod_T(\fil\Mod_\ZZ).$ Indeed this follows because the statement only depends on the underlying $\einf$-rings where it follows from \cite[Corollary 3.4.1.9]{HA}. We will use this remark implicitly in the sequel.
\end{remark}

\begin{lemma}\label{lem: equivalence of modules}
Let $A$ be an animated ring and let $T^\bullet\in \fil\dalg_A$. Then the induced functor $\Mod_{T^\bullet}(\fil\Mod_A)\to \Mod_{\mathrm{Rees}(T^\bullet)}(\Gr\Mod_A)$ is a symmetric monoidal equivalence.
\end{lemma}
\begin{proof}
    By definition (see point $(a)$ of \cite[Notation 4.2.28]{Rak20}) the statement we want to conclude depends only on the underlying $\einf$-rings of all the displayed rings in the statement in the lemma. Thus we can conclude from Lemma \ref{lem: equivalence of operads}.
\end{proof}

\begin{construction}\label{cons: commuting for rees}
    In the situation of Remark \ref{cons: induced dalg on animated for rees} we observe that for any $T\in \fil\dalg_\ZZ$ there is a forgetful functor $U\colon \fil\dalg_{T}\to \Mod_T(\fil\Mod_{\ZZ})$ and similarly a forgetful functor $U\colon \Gr\dalg_{\rees(T)}\to \Mod_{\rees(T)}(\Gr\Mod_{\ZZ[t]}).$ By the equivalences of Remark \ref{cons: induced dalg on animated for rees} and Lemma \ref{lem: equivalence of modules} we have that there is a commuting diagram 
\[\begin{tikzcd}[cramped]
	{\fil\dalg_{T}} && {\Gr\Mod_{\rees(T)}} \\
	{\Mod_T(\fil\Mod_{\ZZ})} && {\Mod_{\rees(T)}(\Gr\Mod_{\ZZ[t]})}
	\arrow[from=1-1, to=1-3]
	\arrow["U"', from=1-1, to=2-1]
	\arrow["U", from=1-3, to=2-3]
	\arrow[from=2-1, to=2-3]
\end{tikzcd}\]

\end{construction}

\begin{remark}\label{remark: commutation is canonical.}
    The commutation in Construction \ref{cons: commuting for rees} is in fact canonical. Let $\cC$ be a derived algebraic context. Then there's a forgetful functor $\dalg(\cC)\to \cC$ which factors over $\Calg(\cC)\to \cC$. These functors are natural with morphisms of contexts $F\colon \cC\to \cD$. Now passing to undercategories we obtain canonical functorial commuting diagrams $\dalg_A(\cC)\to \dalg_{F(A)}(\cD)$ and with $\Calg_A(\cC)\to \Calg_A(\cD)$. Now there's a functorial morphism $\Calg_A(\cC)\to \Calg_A(\cD)$ to the arrow $\Mod_A(\cC)\to \Mod_{F(A)}(\cC)$. Composing the arrows we obtain the diagram in Construction \ref{cons: commuting for rees}.
\end{remark}

We now end with discussing how to work with filtered derived $J$-complete algebras. We will be brief as essentially the same notions discussed in the $J$-complete setting \S~\ref{sec: jcomp setting} and graded $J$-complete setting \ref{ssec: continuous comonad} work. We use the Notation \ref{notation: jcomp} so that $A$ is an animated ring $J\subset \pi_0(A)$.

We make the following definition in order to be compatible with \cite[Notation D.12]{APC}

\begin{recollection}\label{rec: filteredjcomp}
    We let $\filmod^\jcomp_A$ denote the full subcategory of $\filmod_A$ spanned by filtered modules over $A$ so that $\ev^i\colon \filmod_A\to \Mod_A$ factors over $\Mod_A^\jcomp.$
\end{recollection}

\begin{definition}
    We define the $\icat$ of filtered derived $J$-complete algebras as the full subcategory $\fil\dalg_A^\jcomp$ of $\fil\dalg_A$ so that the functor $U\colon \fil\dalg_A\to \filmod_A$ factors over $\filmod_A^\jcomp.$ 
\end{definition}

\begin{remark}
    We can again prove the relevant properties of $\fil\dalg_A^\jcomp$ as in \S~\ref{sec: jcomp setting} using the pullback definition of Lemma \ref{lem: pullback def dalgj}. The only things we need is that the induced forgetful functor $U^J\colon \fil\dalg_A^\jcomp\to \filmod_A^\jcomp$ is conservative and limit preserving. See Remark \ref{lem: not hard pullback} for how to prove this.
\end{remark}

\begin{lemma}\label{lem: rees presrves complete object}
    The functor $\rees\colon \fil\Mod_A\to \Mod_{A[t]}(\Gr\Mod_A)$ preserves $J$-complete modules.
\end{lemma}
\begin{proof}
    This is obvious from the definition since a $J$-complete filtered module $M^\bullet$ gets sent to a graded module $\rees(M)$ whose graded pieces agree with the filtered pieces of $M^\bullet$. Whence completeness is preserved. 
\end{proof}

We next implicitly use Remark \ref{remark: rees connective sent to connective}.

\begin{corollary}\label{cor: rees preserves complete dalg}
    The functor $\rees\colon \fil\dalg_A\to \Gr\dalg_{A[t]}$ preserves $J$-complete derived algebras.
\end{corollary}
\begin{proof}
    The condition can be checked at the level of $A$-modules. Thus this follows from Lemma \ref{lem: rees presrves complete object}.
\end{proof}

\begin{construction}\label{cons: lem the rees construction commutes for derived algebra}
  The following diagram commutes.
\[\begin{tikzcd}[cramped]
	{\fil\dalg^\jcomp_{A}} && {\Gr\Mod^\jcomp_{A[t]}} \\
	{\fil\Mod^\jcomp_A} && {\Gr\Mod^\jcomp_{A[t]}}
	\arrow[from=1-1, to=1-3]
	\arrow["U"', from=1-1, to=2-1]
	\arrow["U", from=1-3, to=2-3]
	\arrow[from=2-1, to=2-3].
\end{tikzcd}\]
Indeed the commutativity can be checked from Construction \ref{cons: commuting for rees} and Lemma \ref{lem: rees presrves complete object} and Corollary \ref{cor: rees preserves complete dalg}.

\end{construction}

\subsection{Geometric classification theorems for derived algebras on formal classifying stacks.}
In this section, our goal is to prove some basic classification theorems for derived algebras on geometric stacks.\footnote{Theorem \ref{theorem: graded dalg bgm} is stated without proof as \cite[Remark B.5]{LahotiManam2025CohomologyRingStacks}.  In any case, we need the compatibility with Moulinos's result \cite[Proposition 4.1]{Moulinos2021GeometryOfFiltrations}. }

We begin with the case of $B\GG_m$ over arbitrary animated rings.
We first recall \cite[Proposition 4.2]{Moulinos2021GeometryOfFiltrations}. 

\begin{proposition}\label{prop: moulinos for bgm}
    Let $A$ be an animated ring and let $B\G_m$ be the classifying stack of $\G_m=\spec (A[\ZZ])$ over $A.$ Then there is a symmetric monoidal equivalence
    $$\Qcoh(B\G_m)\simeq \Gr\Mod_A.$$
\end{proposition}
\begin{proof}
    The proof will be recalled in proving Proposition \ref{theorem: graded dalg bgm}.
\end{proof}

\begin{remark}\label{remark: tensor with bgm}
    Note that Moulinos works over the sphere spectrum $\sph$. To deduce the result for animated rings one could go over his arguments over an animated ring, or first base change his result for $\Gr\Sp$ along $\Sp\to \Mod_\ZZ$ observing the equivalence $$\Gr\Sp\tensor_{\Sp}\Mod_\ZZ\simeq \Gr\Mod_\ZZ$$ and on the left hand side we see that $$\Qcoh(\bgm)\otimes_{\Sp} \Mod_\ZZ=\Qcoh(\bgm)\otimes_{\Qcoh(\spec(\sph))}\Qcoh(\spec(\ZZ))\simeq \Qcoh(\bgm\times_{\spec(\sph)} \spec (\ZZ)),$$ where we used \cite[Corollary 9.4.3.8.]{Lur18} Then for any animated ring $A$ over $\ZZ$, one does the base change again and uses the analogous base change in derived algebraic geometry \cite[Proposition 3.5.3]{GaitsgoryRozenblyum2017SAGI}. We will use this sort of reasoning to get canonicity for results proved by Moulinos using Tannaka duality in derived geometric situations where Tannaka duality is no longer available.
    
    Note further that a similar reasoning will not give the same result for $\dalg.$ This is already highlighted in Remark \ref{remark: base changes for contexts.}.
\end{remark}

\begin{proposition}\label{theorem: graded dalg bgm}
Let $A$ be an animated ring and let $B\G_m$ be the classifying stack of $\G_m=\spec (A[\ZZ])$ over $A.$ Then there is a symmetric monoidal equivalence of $\infty$-categories $\dalg(B\mathbf{G}_m)\simeq \Gr\dalg_A.$
Moreover the following diagram commutes 
\[\begin{tikzcd}[cramped]
	{\Gr\dalg_A} & {\dalg(B\G_m)} \\
	{\Gr\Mod_A} & {\Qcoh(B\G_m)}
	\arrow[from=1-1, to=1-2]
	\arrow["U"', from=1-1, to=2-1]
	\arrow["U", from=1-2, to=2-2]
	\arrow[from=2-1, to=2-2]
\end{tikzcd}\] where the vertical arrows are the obvious forgetful functors and the bottom horizontal arrow is the equivalence of \ref{prop: moulinos for bgm}.
\end{proposition}

\begin{proof}
The commutativity of the diagram is a formal consequence of the method of the proof. So we explain the proof now.

Note that the $\spec(A)\to B\G_m$ is an effective epimorphism and $\mathrm{Cech}(\spec(A)\to B\G_m)$ in degree $n$ is given by $\spec (A[\ZZ])\times_A \ldots \times_A \spec (A[\ZZ])$ where the fiber product is taken $n$-times along with the action and projection map. By Theorem \ref{flat descent} we have

    \begin{equation}\label{totalisation of bgm}
    \dalg(B\G_m)=\mathrm{Tot}(
\cosimp{\dalg_A}{\dalg_{A[\ZZ]}}{\dalg_{A[\ZZ]\otimes A[\ZZ]}})
\end{equation}

Our job now is to \emph{coherently} augment the cosimplicial diagram on the right hand side of the above diagram $\Gr\dalg_R$.\footnote{Here we note that commutativity of diagrams inside $\catinf$ (or any $\infty$-category) is a structure not property.}

Following the proof of \cite[Theorem 4.1]{Moulinos2021GeometryOfFiltrations} we will use the functor $\pi_!\colon \Gr\dalg_A\to \dalg_A$ in order to do so.  Note that in light of Corollary \ref{pi1 is comonadic on dalg}, the functor sends a derived $A$ algebra $T$ to $\pi_!T\in \dalg_A$ equipped with a coaction of $A[\ZZ]$. The cosimplicial presentation of the category in \ref{totalisation of bgm} then shows that in cosimplicial degree $n$ sending 
$$T\mapsto (\pi_!\circ \pi^*)^n\pi_!(T)=A[\ZZ]\otimes_A \pi_!(T)$$ makes the obvious cosimplicial diagram commute since the commutativity amounts to the cobar complex for the monad $\pi_!\pi^*$ on $\dalg_A$

$$
   T\mapsto  \cosimp{\pi_!T}{\pi_!T\otimes_{A}A[\ZZ]}{\pi_!T\otimes_{A}A[\ZZ]\otimes_{A}A[\ZZ]}.
$$

It remains to check using Proposition \ref{prop: criteria for descent} that this augmentation is comonadic. In this case it amounts to checking that $\pi_!$ is conservative, which follows from Corollary \ref{pi1 is comonadic on dalg}, checking that if $X^\bullet\in \Gr\dalg_R$ is a cosimplicial object so that $\pi_!X^\bullet$ splits in $\dalg_R$ then $\Tot(X^\bullet)$ exists and $\pi_!(\Tot(X^\bullet))=\Tot(\pi_!(X^\bullet))$ which is a consequence again of Corollary \ref{pi1 is comonadic on dalg}. 

The only thing (for the comonadicity claim) which remains to check is adjointability condition. But this is a consequence of faithfully flat descent established already in Lemma \ref{lem: adjointable dalg}.

Lastly we note that symmetric monoidality follows from the same argument as that used in Remark \ref{dalg is symmetricmonoidal}.

\end{proof}

\begin{remark}[The underlying derived algebra functor]\label{remark: underlying algebra functor}

Recall that for $\Gr\dalg_A$ one may define an underlying derived algebra functor $\Gr\dalg_A\to \dalg_A$, heuristically sending a graded derived algebra $S\in \Gr\dalg_A$ to $\bigoplus_{i\in \ZZ}S_i.$ Indeed this is just the functor $\pi_!\colon \Gr\dalg_A\to \dalg_A$ of Construction \ref{cons: lifting pi! to animated}. 
The proof of Proposition \ref{theorem: graded dalg bgm} then shows that this functor is implemented by the pullback $\dalg(\bgm)\to \dalg_A$ coming from the effective epimorphism $\spec(A)\to \bgm$.
    
\end{remark}

We next prove a statement which requires the notion of $J$-complete derived algebras developed in \S~\ref{sec: jcomp setting} and the continuous comonadicity developed in  \S~\ref{ssec: continuous comonad}.

\begin{remark}\label{remark: omit verification for qcoh}
    Before turning our attention to the next proposition, we must morally establish an equivalence of $\icats$ $$\Gr\Mod_A^\jcomp\simeq \Qcoh(B\G_m\times \spf(A))$$ where $A$ is an animated ring $J$ is a finitely generated ideal in $\pi_0(A)$. First note that applying Remark \ref{remark: definition of formal stack} via Lemma \ref{lem: luries lemma on spf} we have the following equivalence of $\icats$

    $$\Qcoh(B\G_m\times \spf(A))=\lim_n\Qcoh(B\G_m\times A_n).$$

   Now our definition of $\Gr\Mod_A^\jcomp$ in Definition \ref{def: def of dalggradedjcomp} sets it up so that using Proposition \ref{theorem: graded dalg bgm} we have 
    \begin{align}\label{eq: qcohbgmspf}
        \Qcoh(B\G_m\times \spf(A))&=\lim_n\Qcoh(B\G_m\times A_n)=\lim_n \Gr\Mod_{A_n}\\ &=\lim_n \Fun(\ZZ^\ds, \Mod_{A_n})=\lim_n \prod_{\ZZ} \Mod_{A_n}\notag\\
        &=\prod_{\ZZ} (\lim_n \Mod_{A_n})=\prod_{\ZZ} \Mod_A^{\jcomp}\notag\\
        &=\Gr\Mod_A^{\jcomp}.
    \end{align}

    Here we used that limits commute with limits and in the equality in the third line used our Theorem \ref{thm: mod on spf}. Of course the key non-formal ingredient is indeed our Theorem \ref{thm: mod on spf} without which we do not know how to justify the equivalence. For the $p$-adic case see \cite[Example A.13]{AntieauKrauseNikolaus2023PrismaticDeltaRings}.

    However this method doesn't make it clear that the equivalence of \ref{eq: qcohbgmspf} is compatible with the continuous comonadicity results of Proposition \ref{prop: cts monad mod}. Indeed it is important that the equivalence be compatible with the continuous comonad for the verification for derived algebras. However, we leave this to the reader and turn our attention to derived algebras. As the methods are literally identical with $\dalg$ replaced with $\qcoh$, not much is lost.
\end{remark}

\begin{remark}\label{remark: grdalg is not formal}
   Note that the statement in the next proposition i.e. Proposition \ref{prop: spf graded dalg bgm} cannot be formally proved in the same method as that of Remark \ref{remark: omit verification for qcoh}. The issue is that graded derived $A$-algebras are not just the functor category $\Fun(\ZZ^\ds, \dalg_A)$. Indeed objects of $\Gr\dalg_A$ may not have any ring structure on their positive graded pieces while the category $\Fun(\ZZ^\ds, \dalg_A)$ have a ring structure on all graded pieces.
\end{remark}

\begin{proposition}\label{prop: spf graded dalg bgm}
    Let $A$ be an animated ring and $J\subset \pi_0(A)$ a finitely generated ideal. Then there is a symmetric monoidal equivalence $$\dalg(B\G_m\times \spf A)=\Gr\dalg_A^{\jcomp}$$ where the target is as in Definition \ref{def: def of dalggradedjcomp}. Further the following diagram commutes \[\begin{tikzcd}[cramped]
	{\Gr\dalg_A^\jcomp } & {\dalg(B\G_m\times \spf(A))} \\
	{\Gr\Mod_A^\jcomp} & {\Qcoh(B\G_m\times \spf(A))}
	\arrow[from=1-1, to=1-2]
	\arrow["U_J"', from=1-1, to=2-1]
	\arrow["U_J", from=1-2, to=2-2]
	\arrow[from=2-1, to=2-2]
\end{tikzcd}\] 
\end{proposition}

\begin{proof}
By Theorem \ref{flat descent} and Corollary \ref{cor: dalg on spf} we have a totalisation diagram 

    \begin{equation}\label{totalisation of spf bgm}
    \dalg(B\G_m\times \spf(A))=\mathrm{Tot}(
\cosimp{\dalg_A^{\jcomp}}{\dalg_{A[\ZZ]}^\jcomp}{\dalg^\jcomp_{A[\ZZ]\ctensor A[\ZZ]}})
\end{equation}

The arrows in this diagram are given by completed base changes by Remark \ref{remark: functoriality of dalgj} whence we must once again the augment the cosimplicial object on the right over $\Gr\dalg_A^\jcomp.$ We may now use the functor $\pij$ of Construction \ref{cons: continuous pij on derived algebras} whence the continuous cobar complex for the continuous comonad of Corollary \ref{cor: continuous comonad dalg}, makes the augmentation coherent. It remains to witness the comonadicity of the augmentation to verify the condition in Proposition \ref{prop: criteria for descent}, but other than the adjointability condition, the rest were checked in the proof of Propositon \ref{prop: cts monad mod}. 

The adjointability is now checked Variant \ref{variant: completed adjointability}.
\end{proof}

\begin{remark}[Underlying completed algebra functor]\label{remark: underlying completed algebra functor}
We have the completed analog of Remark \ref{remark: underlying algebra functor}. Namely, by Construction \ref{cons: continuous pij on derived algebras}, there is a functor  $$\pij\colon \Gr\dalg_A^{\jcomp}\to \dalg_A^\jcomp$$ sending a graded complete $A$-algebra $S$ to $\widehat{\bigoplus}_{i\in \ZZ}S_i$. By the proof of Proposition \ref{prop: spf graded dalg bgm}, we see that this functor is implemented by (completed) pullback $\dalg(\bgm\times \spf(A))\to \dalg(\spf(A))$ coming from the effective epimorphism $\spf(A)\to \bgm\times \spf(A).$ 
\end{remark}

\begin{remark}[Analogs of Proposition \ref{prop: spf graded dalg bgm} for diagonalizable group schemes.]\label{remark: classifying stacks diagonalizable group schemes}
   In this remark we continue the thread of Remarks \ref{remark: remark on diagonalizable coactions}, \ref{remark: coactions for diagonalizable on animated rings} and \ref{remark: continuous coaction diagonalizable schemes}. We will sketch how to prove Proposition \ref{prop: spf graded dalg bgm} for diagonalizable group schemes.

   Let $M$ be a commutative abelian group and $\mathbf{D}(M)$ the associated diagonalizable commutative group scheme obtained by the procedure of taking the spectrum of the bicommutative Hopf algebra $\ZZ[M]$ \cite[\S~2.6]{Tate1997FiniteFlatGroupSchemes}. Let $A$ be an animated ring and $J$ a finitely generated ideal of $\pi_0(A).$
   Then consider the formal classifying stack $B\DD(M)\times \spf(A)$. 
   There is of course an effective epimorphism $$\spf(A)\to B\DD(M)\times \spf(A)$$ and the Cech nerve in degree $n$ is given by $\spf(A[M])\times_{\spf(A)}\ldots \times_{\spf(A)} \spf(A[M])$, with the product taken $n$-times and the transition function given by action and coaction maps. We leave it to the interested reader to establish an equivalence of $\icats$ $$\Qcoh(\spf(A)\times B\DD(M)))\simeq \Gr\Mod_A^\jcomp$$ and we explain how to do so for $\dalg$. Note that either an explicit check or an analog of Remark \ref{remark: omit verification for qcoh} works.
   As the techniques are exactly the same, this is not much of an omission.

   By flat descent of $\dalg$, Theorem \ref{flat descent} and using Remark \ref{lem: dalg evaluated in cech} we may write

   $$\dalg(B\DD(M)\times \spf(A))=\Tot(\cosimp{\dalg_A^\jcomp}{\dalg_{A[M]}^\jcomp}{\dalg^\jcomp_{A[M]\widehat{\otimes}_A A[M]}})$$

    We may augment $\Gr\dalg_A^{M,\jcomp}$ on the cosimplicial $\icat$ via $(\pi_!^M)^{\wedge}_{J}$, the functor explained in Remarks \ref{remark: remark on diagonalizable coactions} and \ref{remark: continuous coaction diagonalizable schemes}, and using verbatim the reasoning of Proposition \ref{prop: spf graded dalg bgm}, we may test the hypothesis of Proposition \ref{prop: criteria for descent} to establish an equivalence of $\icats$

    $$\Gr\dalg_A^{M,\jcomp}\simeq \dalg(\spf(A)\times B\DD(M))).$$

\end{remark}

\subsection{Geometric classification theorems for derived algebras on formal filtered stacks.}

We now turn our attention to the promised classification of derived algebras on filtered (formal) stacks.

We will now explain how to use the constructions in \S~\ref{ssec: rees construction section} to manufacture filtered stacks (i.e. those stacks over $A$ whose structure morphism factors over $\agm\to \bgm$) over $A$ from filtered animated algebras over $A.$

\begin{construction}[Manufacturing filtered stacks from filtered animated algebras.]\label{cons: coaction of Az}
Let $T\in \fil\dalg_A^\cn$ i.e. a filtered animated ring as defined in Recollection \ref{rec: filtered derived algebras}. Then by Remark \ref{remark: rees connective sent to connective} we get a graded animated ring $\rees(T)\in \Gr\dalg^\cn_{A[t]}.$ After transporting $\rees(T)$ along the comonadic adjunction of Proposition \ref{prop: monadicity over animated} we obtain that $\pi_!\rees(T)$ is a comodule over the bicommutative Hopf algebra $A[t][\ZZ]$.

Note that the coaction is witnessed by the cobar complex
\begin{equation}\label{eq: cobar complex}
    \cosimp{\pi_!\rees(T)}{\pi_!\rees(T)\otimes_{A[t]}A[t][\ZZ]}{\pi_!\rees(T)\otimes_{A[t]}A[t][\ZZ]\otimes_{A[t]}A[t][\ZZ]}.
\end{equation}

Now we may forget along the functors $\dalg_{A[t]}\to \dalg_{A}$ and noting the equivalence of comonads

$$\pi_!\circ\pi^*\simeq (-)\otimes_{A[t]}A[t][\ZZ]\simeq (-)\otimes_{A[t]}A[t]\otimes_{A}A[\ZZ]\simeq (-)\otimes_A A[\ZZ],$$

we observe that the cobar complex in Equation \ref{eq: cobar complex}, after forgetting to $\dalg_A$ is equivalent to 

\begin{equation}\label{eq: 2nd cobar complex}
    \cosimp{\pi_!\rees (T)}{\pi_!\rees(T)\otimes_{A}A[\ZZ]}{\pi_!\rees(T)\otimes_{A}A[\ZZ]\otimes_{A}A[\ZZ]}.
\end{equation}

As explained in Proposition \ref{pi! is a morphism of contexts} and Remark \ref{remark: rees connective sent to connective} we observe that $\pi_!\rees(T)$ is still an animated connective ring (i.e. the operations described in this construction don't add negative homotopy groups) and all the terms showing up in \ref{eq: 2nd cobar complex} are clearly animated commutative $A$-algebras (noting the bicommutative Hopf algebra structure on $A[\ZZ]$).

Thus in the $\infty$-topos $\Shv((\aring_{A})_{\fpqc})$ the cobar complex for the coaction of $A[\ZZ]$ on $\pi_!\rees(T)$ gives us a simplicial diagram of sheaves whose geometric realisation is defined as the quotient stack $\spec(\pi_!\res(T))/\GG_m$ i.e. the colimit of the diagram

$$\spec(\pi_!\res(T))/\GG_m:=|\simpII{\spec(\pi_!\rees(T))}{\spec(A[\ZZ])\times\spec(\pi_!\rees(T))}|.$$

Note that this diagram lives over the diagram

$$B\GG_m:=|\simpII{\spec(A)}{\spec(A[\ZZ])\times\spec(A))}|$$ and so one obtains a graded stack $p^\gr_T\colon \spec(\pi_!\res(T))/\GG_m\to \bgm.$

\end{construction}

\begin{remark}[\label{remark: coaction is same}]
   As an example of Construction \ref{cons: coaction of Az}, if $T=A$ itself considered as filtered along $\ins^0\colon \dalg_A^\cn\to \fil\dalg_A^\cn$ we get a coaction on the polynomial ring $A[t]$. Then the stack constructed in Construction \ref{cons: coaction of Az} agrees with the stack $\agm\times_{\spec(\ZZ)}\spec(A).$ That the coaction described in Construction \ref{cons: coaction of Az} is the same as the usual one coming from the Cech nerve of $\A^1\to \agm$, as defined intrinsically in spectral algebraic geometry over $\sph$, was checked by Moulinos in the course of the proof of \cite[Proposition 5.1]{Moulinos2021GeometryOfFiltrations}. Therefore, for an animated ring $A$, our definition of $\agm$ is via Construction \ref{cons: coaction of Az}.
   We denote the induced structure morphism $p\colon \agm\to \bgm$.
\end{remark}

We now explain why the stacks of Construction \ref{cons: coaction of Az} in fact deserve the name of \emph{filtered} stacks.

\begin{remark}[Stakcs factoring over $\agm$ are filtered stacks.]\label{graded also gives filtered}
    Note that in Construction \ref{cons: coaction of Az} there is a canonical morphism $p_T^\fil\colon \spec(\pi_!\rees(T))/\G_m\to \agm$ which comes from the structure map $\spec(\pi_!\rees(T))\to \spec(A[t])$. Thus after passing to stacks we obtain a factorisation of $$p^\gr_T\colon\spec(\pi_!\rees(T))/\G_m\to \bgm$$ via $$p\colon \agm\to \bgm,$$ thus giving us a \emph{filtered} stack.
\end{remark}

The next remark explains how to use Proposition \ref{prop: moulinos for bgm} to identify module categories on filtered stacks.

\begin{remark}\label{remark: how to identify image of algebra object}
    We already know that there is a symmetric monoidal equivalence $\Gr\Mod_A\simeq \Qcoh(B\G_m)$ (for example by the tensor argument in $\prl$ advocated in Remark \ref{remark: tensor with bgm}) or just by going over Moulinos's argument leading to Proposition \ref{prop: moulinos for bgm} again. Let $E\in \Calg(\Gr\Mod_A)$ and let $\cE\in \Calg(\Qcoh(\bgm))$ denote its image under the equivalence. Then it follows for general operadic reasons from Lemma \ref{lem: equivalence of operads} that there's a symmetric monoidal equivalence $$\Mod_{E}(\Gr\Mod_A)\simeq \Mod_{\cE}(\Qcoh(B\G_m)).$$ If $\cE$ is of the form $f_*\cO_X$ for some quasi-affine morphism $f\colon X\to B\G_m$ then by Proposition \ref{prop: GR facts we need}, we see that $$\Mod_E(\Gr\Mod_A)\simeq \Qcoh(X).$$
    If $E$ was of the form $\rees(T)$ for some $T\in \fil\dalg_A$ then using Lemma \ref{lem: equivalence of modules} we can conclude that $$\fil\Mod_T\simeq \Qcoh(X).$$

    This is the format of the argument we will use next for the graded stacks from Construction \ref{cons: coaction of Az}. In particular for $T\in \fil\dalg_A^\cn$ it will suffice to show that the stack $p_T^\gr\colon \spec(\pi_!\res(T))/\GG_m\to \bgm$ has the property that under the equivalence of Proposition \ref{prop: moulinos for bgm}, the image of the algebra $\rees(T)$ and $(p_T^\gr)_*\cO$ agree.
\end{remark}

The next proposition identifies that the underlying derived algebra of $(p_T^\gr)_*\cO$ is in fact $\pi_!\rees(T).$ Here, by underlying derived algebra, we will use the notion of Remark \ref{remark: underlying algebra functor}.

\begin{proposition}\label{prop: gr sends correct algebra}
Let $T\in \fil\dalg_A^\cn$. Then the equivalence $\Gr\dalg_A\simeq \dalg(\bgm)$ of Proposition \ref{theorem: graded dalg bgm} sends $$\rees(T)\mapsto (p_T^\gr)_*\cO.$$
\end{proposition}

\begin{proof}

First note that there is a tautlogical pullback coming from the trivial torsor on $\spec(\pi_!\rees(T))/\G_m$

% https://q.uiver.app/#q=WzAsNCxbMSwwLCJcXHNwZWMoQSkiXSxbMSwxLCJcXGJnbSJdLFswLDEsIlxcc3BlYyhcXHBpXyFcXHJlZXMoVCkpL1xcR19tIl0sWzAsMCwiXFxzcGVjKFxccGlfIVxccmVlcyhUKSkiXSxbMCwxLCJxIl0sWzIsMSwicF5cXGdyX1QiLDJdLFszLDIsInFfVCIsMl0sWzMsMCwicF9UIl1d
\[\begin{tikzcd}[cramped]
	{\spec(\pi_!\rees(T))} & {\spec(A)} \\
	{\spec(\pi_!\rees(T))/\G_m} & \bgm
	\arrow["{p_T}", from=1-1, to=1-2]
	\arrow["{q_T}"', from=1-1, to=2-1]
	\arrow["q", from=1-2, to=2-2]
	\arrow["{p^\gr_T}"', from=2-1, to=2-2]
\end{tikzcd}\]

which gives a diagram on $\dalg$

% https://q.uiver.app/#q=WzAsNCxbMCwxLCJcXGRhbGcoXFxzcGVjKEEpKSJdLFsxLDEsIlxcZGFsZyhcXGJnbSkiXSxbMSwwLCJcXGRhbGcoXFxzcGVjKFxccGlfIVxccmVlcyhUKSkvXFxHX20pIl0sWzAsMCwiXFxkYWxnKFxcc3BlYyhcXHBpXyFcXHJlZXMoVCkpKSJdLFsyLDEsIihwX1ReXFxncilfKiJdLFszLDAsIihwX1QpXyoiLDJdLFsyLDMsIihxX1QpXioiLDJdLFsxLDAsInFeKiJdXQ==
\[\begin{tikzcd}[cramped]
	{\dalg(\spec(\pi_!\rees(T)))} & {\dalg(\spec(\pi_!\rees(T))/\G_m)} \\
	{\dalg(\spec(A))} & {\dalg(\bgm)}
	\arrow["{(p_T)_*}"', from=1-1, to=2-1]
	\arrow["{(q_T)^*}"', from=1-2, to=1-1]
	\arrow["{(p_T^\gr)_*}", from=1-2, to=2-2]
	\arrow["{q^*}", from=2-2, to=2-1]
\end{tikzcd}\] where we used base change for $\dalg$ as explained in Proposition \ref{prop: derived base change for algebras}. 
By identifying the top left category on the second diagram with $\dalg_{\pi_!\rees(T)}$ and the bottom left one with $\dalg_{A}$ we see that $$(p_T)_*\circ (q_T)^*\cO=q^*(p_T^\gr)_*\cO$$ where $\cO$ is the structure sheaf of $\spec(\pi_! \rees(T))/\G_m.$  But the left hand side is just $\pi_!\rees(T)$ and by Remark \ref{remark: underlying algebra functor}, this agrees with the underlying algebra of $(p_T^\gr)_*\cO.$
The action on the former is given by the cobar complex of Construction \ref{eq: 2nd cobar complex}. The action on the latter is given by descent along the cover $\spec(A)\to \bgm.$ But an examination of the complex in the proof of Proposition \ref{theorem: graded dalg bgm} shows that these are the same actions. Thus we conclude.

\end{proof}

\begin{corollary}\label{cor: simpson for filtered algebras}
Let $T\in \fil\dalg_{A}^\cn.$ Then there is a symmetric monoidal equivalence of $\icats$
$$\fil\Mod_T\simeq  \Qcoh(\spec(\pi_!\rees (T)/\G_m).$$
\end{corollary}

\begin{proof}
    The equivalence in Proposition \ref{prop: gr sends correct algebra} is stronger than what is required to implement Remark \ref{remark: how to identify image of algebra object}. Indeed we conclude that $\rees(T)\mapsto (p_T^\gr)_*\cO$ as commutative algebras so we can apply the conclusion in Remark \ref{remark: how to identify image of algebra object}. 
\end{proof}
We are now ready to prove classification theorems for derived algebras in quasi-coherent sheaves on stacks which are quasi-affine over $B\G_m$.

\begin{remark}
Let $f\colon X\to \bgm$ be a quasi-affine morphism in the sense of Definition \ref{def: quasi-affine etc}.
It follows from Lemma \ref{lem: pushforwad preserves dalg} and Theorem \ref{theorem: graded dalg bgm} that $f_*\cO_X$ is naturally a derived algebra over $B\G_m$ and so a graded derived algebra over $R$. It follows from Corollary \ref{cor: quasi-aff preserves dalg} that one can understand $\dalg(X)$ as the $\icat$ $\dalg(B\G_m)_{f_*\cO_X/}.$    
\end{remark}

\begin{corollary}\label{cor: simpson for derived algebras on filtered algebras}
   Let $T\in \fil\dalg_{A}^\cn.$ There is a symmetric monoidal equivalence of $\icats$
   $$\fil\dalg_T\simeq \dalg(\spec(\pi_!\rees (T)/\G_m)$$ so that the following diagram commutes 
   
\[\begin{tikzcd}[cramped]
	{\fil\dalg_{T}} & {\dalg(\spec(\pi_!\rees(T))/\G_m)} \\
	{\fil\Mod_T} & {\Qcoh(\spec(\pi_!\rees(T))/\G_m)}
	\arrow[from=1-1, to=1-2]
	\arrow["U"', from=1-1, to=2-1]
	\arrow["U", from=1-2, to=2-2]
	\arrow[from=2-1, to=2-2]
\end{tikzcd}\] where the bottom equivalence is of Corollary \ref{cor: simpson for filtered algebras} and the top one is the diaplayed one in the statement of this claim.
\end{corollary}
\begin{proof}
The commutativity of the diagram is a formal consequence of the proof.

        Note that if $f\colon X\to \bgm$ is a  quasi-affine morphism, then by Theorem \ref{cor: quasi-aff preserves dalg} it follows that $\dalg(X)\simeq \dalg(\bgm)_{f_*\cO_X/}.$ We also know by Proposition \ref{theorem: graded dalg bgm} that $\Gr\dalg_A\simeq \dalg(\bgm) $. Now if $T\in \fil\dalg_A^\cn$ then the equivalence in Proposition \ref{prop: gr sends correct algebra} shows that $$(\Gr\dalg_A)_{\rees(T)/}\simeq \dalg(\bgm)_{(p_T^\gr)_*\cO/}$$ is an equivalence. From the equivalence in Remark \ref{cons: induced dalg on animated for rees} we learn that $$\fil\dalg_T\simeq \dalg(\bgm)_{(p_T^\gr)_*\cO/}\simeq \dalg(\spec(\pi_!\rees (T)/\G_m)).$$

\end{proof}

\begin{example}
    Let $A$ be an animated ring. Then Corollary \ref{cor: simpson for derived algebras on filtered algebras} implies an equivalence of $\icats$
    $$\fil\dalg_A\simeq \dalg(\agm).$$
\end{example}

\begin{example}\label{eg: knygalg}
    Let $k$ be a perfect field of characteristic $p$. Let $W(k)$ be the ring of Witt vectors of $k$ given the $p$-adic filtration. Then we can treat $p^\bullet(W(k))\in \fil\dalg_{W(k)}^\cn.$ Then Corollary \ref{cor: simpson for filtered algebras} implies 
    $$\fil\dalg_{p^\bullet W(k)}=\dalg(W(k))_{p^\bullet W(k)/}\simeq \dalg(\spec(W(k)[u,t]/(ut-p)/\G_m).$$
\end{example}

We now explain the continuous analogues of these statement. We are in the setting of Notation \ref{notation: jcomp} so that $A$ is an animated ring and $J\subset \pi_0(A)$ is a finitely generated ideal.

\begin{notation}\label{notation: complete fildalg cnnective} We introduce some notation to streamline the sequel
\begin{enumerate}
    \item      Let $(\dalg_A^\jcomp)^\cn:=\dalg^\jcomp_A\times_{\dalg_A}\dalg_A^\cn$ be the $\icat$ of connective $J$-complete derived $A$-algebras. Similarly we define $(\Gr\dalg_A^\jcomp)^\cn$ and  $(\fil\dalg_A^\jcomp)^\cn$.
    \item The Rees construction preserves $\rees\colon(\fil\dalg_A^\jcomp)^\cn\to (\Gr\dalg_A^\jcomp)^\cn$. Indeed this follows from the $t$-exactness of Remark \ref{remark: rees connective sent to connective} and Lemma \ref{lem: rees presrves complete object}. 
\item The functor $\pij$ of \S~\ref{ssec: continuous comonad} is right $t$-exact since $\pi_!$ is $t$-exact and completion $(-)\cj$ is right $t$-exact (see \cite[Remark 4.2.17]{Lur11}) and the fact that completion can be computed at the level of modules by Lemma \ref{lem: adjoint is completion}). In particular $\pij\colon (\Gr\dalg_A^\jcomp)^\cn\to (\dalg_A^\jcomp)^\cn$ is well defined.

\end{enumerate}
\end{notation}

\begin{construction}[Manufacturing $J$-formal filtered stacks from $J$-complete filtered animated algebras.]\label{cons: manufacturing completed graded stacks.}
Let $T\in (\fil\dalg_A^\jcomp)^\cn$ as in point $(1)$ of Notation \ref{notation: complete fildalg cnnective}. By point $(2)$ and $(3)$ of Notation \ref{notation: complete fildalg cnnective} we obtain a connective $J$-complete derived algebra $\pij\rees(T)\in (\dalg_{A[t]}^\jcomp)^\cn$ which has a continuous coaction of $A[t][\ZZ].$ By running the same argument in Construction \ref{cons: coaction of Az} we  get an analogue of the cobar complex of \ref{eq: 2nd cobar complex}, withnessing the continuous action of $A[\ZZ]$ on the continuous monad $\pij\circ \pi^*(-)\simeq \ctensor_A A[\ZZ]$ given as 

\begin{equation}
    \cosimp{\pij\rees(T)}{\pij\rees(T)\ctensor_A A[\ZZ]}{\pij\rees(T)\ctensor_A A[\ZZ]\ctensor_A A[\ZZ]}
\end{equation}

Once again we may look at the $\infty$-topos $\Shv((\aring_A)_{\fpqc})$ and we get a cobar complex given as follows 

$$\spf(\pij\rees(T))/\G_m:=|\simpII{\spf(\pij\rees(T))}{\spf(A[\ZZ])\times \spf(\pij\rees(T))}|$$

and this lives over the diagram 

$$\bgm\times \spf(A):=|\simpII{\spf(A)}{\spf(A)\times \spf(A[\ZZ])}|$$ and so we obtain a structure morphism $p_T^\gr\colon \spf(\pij\rees(T))/\G_m\to \bgm\times \spf(A)$. We have engaged in minor abuse of notation here.

\end{construction}
\begin{remark}\label{lem: recognising cts algebra}
We explain an analogue of Remark \ref{remark: how to identify image of algebra object}.
    Let $f\colon X\to B\G_m\times \spf(A)$ be a quasi-affine morphism. Then $X$ is automatically formal for the $J$-adic topology (e.g. by Remark \ref{remark: definition of formal stack}). We know the symmetric monoidal equivalence $$\Gr\Mod_A^\jcomp\simeq \Qcoh(\bgm\times \spf(A))$$ explained in Remark \ref{remark: omit verification for qcoh}. Let $E\in \Calg(\Gr\Mod_A^\jcomp)$ and $\cE\in \Calg(\Qcoh(\bgm\times \spf(A))$ be its image under the aforementioned equivalence. Then again for operadic reasons of Lemma \ref{lem: equivalence of operads} we have an equivalence of $\icats$

    $$\Mod_{E}(\Gr\Mod_A^\jcomp)\simeq \Mod_{\cE}(\Qcoh(\bgm\times \spf(A)).$$ Again if $\cE=f_*\cO_X$ for some quasi-affine $f\colon X\to \bgm\times \spf(A)$ we obtain an equivalence by Proposition \ref{prop: GR facts we need}
    $$\Mod_{E}(\Gr\Mod_A^\jcomp)\simeq \Qcoh(X),$$ and if $E=\rees(T)$ for some $T\in \fil\Mod_{T}^\jcomp$ we obtain an equivalence $$\fil\Mod_T^\jcomp\simeq \Qcoh(X).$$
    
\end{remark}

\begin{remark}[Stacks factoring over $\agm\times \spf(A)$ are filtered formal stacks]
    An analog of the reasoning in Remark \ref{graded also gives filtered}, shows that for all the stacks $p_T^\gr \spf(\pij(\rees(T)))\to \bgm\times \spf(A)$ admit a factorization $$p_T^\fil\colon \spf(\pij(\rees(T)))\to \agm\times \spf(A)$$ giving them the structure of \emph{filtered formal stacks.} 
\end{remark}

\begin{proposition}\label{prop: cts algebra sended correctly}
    Let $T\in (\fil\dalg^\jcomp_A)^\cn$. Then the equivalence $\Gr\dalg_A^\jcomp\simeq \dalg(\bgm\times \spf(A))$ of Proposition \ref{prop: spf graded dalg bgm} sends 
    $$\rees(T)\mapsto (p_T^\gr)_*\cO.$$
\end{proposition}
\begin{proof}
    This is proved in the same way as Proposition \ref{prop: gr sends correct algebra}. However, for completeness, we explain some of the details as we will omit the same in the proofs of Corollaries \ref{cor: cts main theorem} and \ref{cor: cts main theorem for dalg}. The trivial torsor gives a diagram 
\[\begin{tikzcd}[cramped]
	{\spf(\pij\rees(T))} & {\spf(A)} \\
	{\spf(\pij\rees(T))/\G_m} & \bgm\times \spf(A)
	\arrow["{p_T}", from=1-1, to=1-2]
	\arrow["{q_T}"', from=1-1, to=2-1]
	\arrow["q", from=1-2, to=2-2]
	\arrow["{p^\gr_T}"', from=2-1, to=2-2]
\end{tikzcd}\]

The adjointability gives us a diagram

\[\begin{tikzcd}[cramped]
	{\dalg(\spf(\pij\rees(T)))} & {\dalg(\spf(\pij\rees(T))/\G_m)} \\
	{\dalg(\spf(A))} & {\dalg(\bgm\times \spf(A))}
	\arrow["{(p_T)_*}"', from=1-1, to=2-1]
	\arrow["{(q_T)^*}"', from=1-2, to=1-1]
	\arrow["{(p_T^\gr)_*}", from=1-2, to=2-2]
	\arrow["{q^*}", from=2-2, to=2-1]
\end{tikzcd}\]    

Following the reasoning of Proposition \ref{prop: gr sends correct algebra} and using Remark \ref{remark: underlying completed algebra functor} we learn that the underlying algebra of $p_T^\gr\cO$ is indeed $\pij(\rees(T)).$ Note here that the base changes are automatically completed for the same reason as explained in Example \ref{ex: pushforward on spf}.

The action on both can again be checked to be equivalent, using the cobar complexes of Proposition \ref{prop: spf graded dalg bgm} and Construction \ref{cons: manufacturing completed graded stacks.} respectively.

One may worry that the formal stack $\agm\times \spf(A)$ might have a different action coming from the Cech nerve of $\mathbf{A}^1\times \spf(A)\to \agm\times \spf(A)$ than the one in Construction \ref{cons: manufacturing completed graded stacks.}, but both these actions are base changed  along $\spf(A)\to \spec(A)$ from the decompleted (i.e. usual) action wherein this is checked by Moulinos as explained in Remark \ref{remark: coaction is same}.
\end{proof}

\begin{corollary}\label{cor: cts main theorem}
Let $T\in (\fil\dalg_A^\jcomp)^\cn$. Then there's a symmetric monoidal equivalence of $\icats$ 
$$\fil\Mod_T^\jcomp\simeq \Qcoh(\spf(\pij\rees(T)/\G_m)$$
\end{corollary}
\begin{proof}
    This follows \emph{mutatis mutandis} from Proposition \ref{prop: cts algebra sended correctly} as Corollary \ref{cor: simpson for filtered algebras} followed from Proposition \ref{prop: gr sends correct algebra}.
\end{proof}

\begin{corollary}\label{cor: cts main theorem for dalg}
Let $T\in (\fil\dalg_A^\jcomp)^\cn$. Then there's a symmetric monoidal equivalence of $\icats$ 
$$\fil\dalg_T^\jcomp\simeq \dalg(\spf(\pij\rees(T)/\G_m)$$ and the following diagram commutes  \[\begin{tikzcd}[cramped]
	{\fil\dalg^\jcomp_{T}} & {\dalg(\spf(\pi_!\rees(T))/\G_m)} \\
	{\fil\Mod^\jcomp_T} & {\Qcoh(\spf(\pi_!\rees(T))/\G_m)}
	\arrow[from=1-1, to=1-2]
	\arrow["U"', from=1-1, to=2-1]
	\arrow["U", from=1-2, to=2-2]
	\arrow[from=2-1, to=2-2]
\end{tikzcd}\] 
\end{corollary}
\begin{proof}
This follows \emph{mutatis mutandis} like Corollary \ref{cor: simpson for derived algebras on filtered algebras} from the continuous analogues of the statements occuring in the proof there.
\end{proof}

\begin{example}\label{example main thm.}
    Let $(A,I)$ be an orientable prism and let $d\in I$ be a global trivialisation for the Cartier divisor $I\in \Pic(A)$. Then by Corollary \ref{cor: cts main theorem for dalg}, we have an equivalence of $\icats$ 
    $$\fil\dalg_{\fil_{d}A}^{(p,d)\text{-comp}}\simeq \dalg(\spf(A[u,t]/(ut-d)/\G_m),$$ where we used the computation that 
    $$\rees(\fil_dA)\simeq A[u,t]/(ut-d)$$ as graded algebras.

\end{example}

\subsection{$t$-structures for derived algebras on stacks and miscellaneous compatibilities.}

In this section, which is perhaps the least satisfactory of all of the sections, we explain some facts about $t$-structures in the formal setting.

\begin{recollection}\label{rec: t-structure on geometric stacks}
    Recall that a stack $X\colon (\aring)_{\fpqc}\to \calS$ is geometric if there is a faithfully flat\footnote{We warn the reader that faithfully flat here is \emph{stronger} than an effective epimorphism for the flat topology. A representable schematic morphism $X\to Y$ of prestacks is faithfully flat if and only if for any affine scheme $\spec(R)\to Y$ the base change $X\times_Y\spec(R)\to \spec(R)$ is faithfully flat in the sense of derived geometry.} affine $p\colon \spec(R)\to X$ for some animated ring $R$. Lurie shows in \cite[Proposition 3.4.17]{Lurie2011DAGVIII} that this implies that the canonical $t$-structure on $\Qcoh(X)$ has the property that $\cF\in \Qcoh(X)$ is  (co)connective if and only if $p^*\cF$ is (co)connective in $\Mod_R$. Note that, strictly speaking, Lurie's proof is given in the context of spectral algebraic geometry while our work is in derived algebraic geometry. However, an examination of the proof shows that \emph{literally} the same proof works in our context. 
\end{recollection}

\begin{remark}\label{remark: stacks of filtered are all geometric}
    The stacks of Construction \ref{cons: coaction of Az} are all geometric stacks in this sense.
\end{remark}

\begin{definition}\label{def: connectivity for derived algebras}
    Let $X\colon \aring\to \calS$ be any prestack. We define  as the fiber product $$\dalg^\cn(X):=\dalg(X)\times_{\Qcoh(X)}(\Qcoh(X))_{\geq 0}$$ and similarly we define 
    $$\dalg^\ccn(X):=\dalg(X)\times_{\Qcoh(X)}(\Qcoh(X))_{\leq 0}.$$
\end{definition}

\begin{proposition}\label{prop: texactness of derived algebras}
    Let $X$ be a geometric stack. Let $p\colon \spec(R)\to X$ be a faithfully flat surjective representable morphism from an affine derived scheme exhibiting the geometricity of $X.$ Then a derived algebra $\cA\in \dalg(X)$ is in $\dalg^\cn(X)$ (resp. $\dalg^\ccn(X)$) if and only if $p^*\cA\in \dalg_R^\cn$ (resp. $p^*\cA\in \dalg_R^\ccn).$ 
\end{proposition}

\begin{proof}
    This follows immediately from Definition \ref{def: connectivity for derived algebras} and Recollection \ref{rec: t-structure on geometric stacks}.
\end{proof}

\begin{remark}\label{remark: moulinos's generic and special fiber maps}
    In \cite[\S~6]{Moulinos2021GeometryOfFiltrations}, the author proves that there are unique maps $\bgm\to \agm$ (the inclusion of the special fiber) and $\G_m/\G_m\to \agm$ (the inclusion of the generic fiber) which induce maps corresponding to taking the graded of a filtered object and the underlying object of the filtered object respectively. His proof uses Tannaka duality in the sense of \cite{Lurie2011DAGVIII} and \cite{BhattHalpernLeistner2017TannakaRevisited}. 
    There is no hope of replicating his results naively in the derived setting as the next example shows.
\end{remark}

\begin{example}\label{example: no derived tannaka}
Let $R$ and $S$ be two $\einf$-rings. Then by \cite[Proposition 7.1.2.7]{HA} there is an equivalence of mapping spaces 
$$\Maps_{\Calg(\Sp)}(R,S)\simeq \Maps_{\Calg(\Pr^L_{\st})}(\Mod^\otimes_R,\Mod_S^{\otimes}).$$
In particular the symmetric monoidal structure on $\Mod_R$ is too little information to remember maps between animated rings. While rectifying this is work in progress \cite{MundingerSahai_inprep}, we turn to other means to get analogs of Remark \ref{remark: moulinos's generic and special fiber maps}. This was already hinted at in Remark \ref{remark: tensor with bgm}.
\end{example}

\begin{lemma}\label{lemma: canonical maps even in derived geometry}
Let $A$ be an animated ring. Then there are canonical maps\footnote{We explicitly do not claim any uniqueness upto contractible ambiguity of the maps of stacks.}
\begin{enumerate}
    \item $j\colon \spec(A)\simeq \G_m/\G_m\to \agm\times \spec(A) $ which after the identification of the corresponding quasi-coherent categories with $\Mod_A$ and $\fil\Mod_A$ correspond to sending $\fil^\bullet X\in \fil\Mod_A$ to $X(-\infty):=\colim_i \fil^i X$.
     \item $i\colon B\G_m\times \spec(A)\to \agm\times \spec(A) $ which after the identification of the corresponding quasi-coherent categories with $\Gr\Mod_A$ and $\fil\Mod_A$ correspond to sending $\fil^\bullet X\in \fil\Mod_A$ to $\bigoplus_{i\in \ZZ} \fil^i X$.
\end{enumerate}
\end{lemma}
\begin{proof}
We will only prove $(1)$ and the proof of $(2)$ is similar.
    We may first work over $\ZZ$. Then this is clear from \cite[\S~6]{Moulinos2021GeometryOfFiltrations} since for classical stacks Tannaka duality indeed applies. 

    We thus have a map $j\colon \spec(\ZZ)\to \agm$ so that $$j^*\colon \Qcoh(\agm)\to \spec(\ZZ)$$ has the intended effect. In particular the following diagram commutes $$
    % https://tikzcd.yichuanshen.de/#N4Igdg9gJgpgziAXAbVABwnAlgFyxMJZABgBpiBdUkANwEMAbAVxiRAB12BFAYwgAsAFJzoBzALYBKEAF9S6TLnyEUARnJVajFm068Bw9nDQwehgFrnJ0uQux4CRMqs31mrRB3YAzLA04AstAA+pyWsvIgGPbKROou1G46noEhYeaymjBQovBEoN4AThDiSGQgOBBI6lrubABWAHoAVBEFxaWI5ZVIAEzUDHQARjAMAAqKDioghVii-DggidoeXtjiMACObSBFJX3UPYgAzMt1KT5+jZxDTAwMMDgAGpzidGhwlQAET4IAtJwsGBvDgAJ42SJ7To1I6nWrJNZYDbbGQUGRAA
\begin{tikzcd}
\Qcoh(\agm) \arrow[r, "j^*"] \arrow[d, "\simeq"']         & \Qcoh(\spec(\ZZ)) \arrow[d, "\simeq"] \\
\fil\Mod_\ZZ \arrow[r, "\fil^\bullet X\mapsto X(-\infty)"] & \Mod_\ZZ                             \end{tikzcd}
.$$

We may tensor this with $\Mod_A$ for any animated ring $A$ and use \cite[Proposition 3.5.3]{GaitsgoryRozenblyum2017SAGI}.
\end{proof}

\begin{corollary}\label{cor: can map even for derived algebras}
Let $A$ be an animated ring. Then there are canonical maps
\begin{enumerate}
    \item $j\colon \spec(A)\simeq \G_m/\G_m\to \agm\times \spec(A) $ which after the identification of the corresponding derived algebra categories with $\dalg_A$ and $\fil\dalg_A$ correspond to sending $\fil^\bullet S\in \fil\Mod_A$ to $S(-\infty):=\colim_i \fil^i S$.
     \item $i\colon B\G_m\times \spec(A)\to \agm\times \spec(A) $ which after the identification of the corresponding derived algebra categories with $\Gr\dalg_A$ and $\dalg_A$ correspond to sending $\fil^\bullet S\in \fil\Mod_A$ to $\bigoplus_{i\in \ZZ} \fil^i S$.
\end{enumerate}
\end{corollary}
\begin{proof}
    We will prove $(1)$ as the proof of $(2)$ is similar. 

    By Lemma \ref{lem: pushforwad preserves dalg} and Lemma \ref{lemma: canonical maps even in derived geometry}, we have a commuting diagram $$% https://tikzcd.yichuanshen.de/#N4Igdg9gJgpgziAXAbVABwnAlgFyxMJZABgBpiBdUkANwEMAbAVxiRAB12pGBzACk50eAW055h8AASc4aGAGM+AQQCUKkAF9S6TLnyEUARnJVajFm07cG-GXMWr1WndjwEiZQ6frNWiDuwAivIQABYC7EKi7OJSdgrKapraIBiu+kTGXtQ+Fv6cwWERUWJYEnDS7LIJjskueu4oZABM3uZ+AQBmWAycALLQAPpKdam6bgbIxq057ZbsA1DDmqYwUDzwRKCdAE4QwkhkIDgQSMZmvmwAVgB6AFSju-uH1CdIzbOX-gCqINQMdAARjAGAAFcYZfw7LA8UI4R57A6Ic5vRAAZk+eRAv2cICeSI+x1O6MxHVuD1x+Per2JABZSfNsBIAI5-EAA4FgiGNEDQ2HwymIpD0olIACsDPy7G6DBunEBTAYDBgOEkAA1OMI6Gg4Cd1XwALScLBgTo4ACeThSVJJosQEouWJkZRgrI0FA0QA
\begin{tikzcd}
\dalg(\agm\times \spec(A)) \arrow[r, "j^*"] \arrow[d, "U"']      & \dalg(\spec(A)) \arrow[d, "U"]                 \\
\Qcoh(\agm\times \spec(A)) \arrow[r, "j^*"] \arrow[d, "\simeq"'] & \Qcoh(\agm\times \spec(A)) \arrow[d, "\simeq"] \\
\fil\Mod_A \arrow[r, "\fil^\bullet X\mapsto X(-\infty)"]         & \Mod_A                                        
\end{tikzcd}$$ where we have engaged in abuse of notation of Notation \ref{notation: abuse of notation for pushforward}.

By naturality of both vertical composite arrows i.e. the naturality in Corollary \ref{cor: simpson for derived algebras on filtered algebras}, we see that the vertical composites on both sides are the forgetful functors. It follows that one has a unique (upto contractible ambiguity) induced arrow 
$$\fil\dalg_A\to \dalg_A$$ which must be the underlying algebra functor sending $\fil^\bullet S\mapsto \colim_i \fil^i S.$
\end{proof}

Now we may turn our attention to the formal story. Therefore we now work only in the setting of notation \ref{notation: jcomp} which we recall. Here $A$ is an animated ring and $J\subset \pi_0(A)$ is a finitely generated ideal.

\begin{lemma}\label{lemma: canonical maps completed derived geometry}
Let $A$ be an animated ring and $J\subset \pi_0(A)$ a finitely generated ideal. Then there are canonical maps
\begin{enumerate}
    \item $j\colon \spf(A)\simeq \G_m/\G_m\times \spf(A)\to \agm\times \spf(A) $ which after the identification of the corresponding quasi-coherent categories with $\Mod^\jcomp_A$ and $\fil\Mod^\jcomp_A$ correspond to sending $\fil^\bullet X\in \fil\Mod_A$ to $X(-\infty):=(\colim_i \fil^i X)^\wedge_J$.
     \item $i\colon B\G_m\times \spf(A)\to \agm\times \spf(A) $ which after the identification of the corresponding quasi-coherent categories with $\Gr\Mod^\jcomp_A$ and $\fil\Mod^\jcomp_A$ correspond to sending $\fil^\bullet X\in \fil\Mod_A$ to $\widehat{\bigoplus}_{i\in \ZZ} \fil^i X$.
\end{enumerate}
\end{lemma}
\begin{proof}
   We will only prove $(1)$ as the proof of $(2)$ is similar. In fact our proof is similar to that of Lemma \ref{lemma: canonical maps even in derived geometry}. Indeed we consider $\agm\to \spec(\ZZ)$ and base change this along $\spf(A)\to \spec(\ZZ)$ noting that $\Qcoh(\agm)$ is compactly generated (hence dualizable), we may again conclude by \cite[Proposition 3.5.3]{GaitsgoryRozenblyum2017SAGI}

\end{proof}

\begin{corollary}\label{cor: can map in the completed derived setting.}
Let $A$ be an animated ring and $J\subset \pi_0(A)$ a finitely generated ideal. Then there are canonical maps
\begin{enumerate}
    \item $j\colon \spf(A)\simeq \G_m/\G_m\times \spf(A)\to \agm\times \spf(A) $ which after the identification of the corresponding quasi-coherent categories with $\dalg^\jcomp_A$ and $\fil\dalg^\jcomp_A$ correspond to sending $\fil^\bullet X\in \fil\Mod_A$ to $X(-\infty):=(\colim_i \fil^i X)^\wedge_J$.
     \item $i\colon B\G_m\times \spf(A)\to \agm\times \spf(A) $ which after the identification of the corresponding quasi-coherent categories with $\Gr\dalg^\jcomp_A$ and $\fil\dalg^\jcomp_A$ correspond to sending $\fil^\bullet X\in \fil\Mod_A$ to $\widehat{\bigoplus}_{i\in \ZZ} \fil^i X$.
\end{enumerate}
\end{corollary}   
\begin{proof}
We sketch the proof of $(1)$ as $(2)$ is proved similarly.
    This is proved using Lemma \ref{lemma: canonical maps completed derived geometry} and the naturality of Corollary \ref{cor: cts main theorem for dalg} as Corollary \ref{cor: can map even for derived algebras} is obtained from Lemma \ref{lemma: canonical maps even in derived geometry}.
\end{proof}

We end this section with a discussion of some $t$-structure issues in the formal setting. 

\begin{remark}
    Recall from Remark \ref{remark: definition of formal stack} that we defined a $J$-adic formal stack as a stack $X\colon \aring_\fpqc\to \calS$ whose structure morphism $X\to \spf(A)$. For the moment, we may define a geometric $J$-adic formal stack as a $J$-adic formal stack $X$, so that there is an $A$-algebra $B$ and a faithfully flat representable affine morphism $p\colon \spf(B)\to X.$

    Unfortunately, this doesn't imply that a complex $\cF\in \Qcoh(X)$ is coconnective if and only if $p^*\cF\in \Qcoh(\spf(B))=\Mod_B^{\jcomp}$ is coconnective for the canonical $t$-structure induced from $\Mod_B$. Indeed the obstruction can be obtained by an examination of the proof of \cite[Proposition 3.4.17]{Lurie2011DAGVIII}. The proof relies on a certain comonad obtained by tensoring with the algebra $C:=B\widehat{\otimes}_X B$ being $t$-exact (i.e. the comonadic version of \cite[Proposition 6.20]{Lur11b}). In our completed setting the comonad is only right $t$-exact and not left $t$-exact as would be required. Therefore an analog of Proposition \ref{prop: texactness of derived algebras} does not hold.

\end{remark}

\section{Forthcoming work.}\label{sec: forthcoming work}

We will now explain some forthcoming work of the author wherein the theory developed in this note will be utilised. 

We start with \cite{Sah25Nyg}. In that paper, given a bounded prism so that $A/I$ satisfies some regularity conditions, like either being regular (example in the case of the Breuil-Kisin prism), or being perfectoid, then there's a good theory of the \emph{relative} Nygaard filtered prismatization. 
Set $\oA:=A/I.$
This theory will send a smooth $p$-adic formal scheme $X/\overline{A}\mapsto (X/\oA)^\nyg$ a $p$-adic formal stack living over $\spf (A[u,t]/(ut-d))/\G_m$ and has the following features (among others)

\begin{enumerate}
\item We have an equivalence in $f_*^\nyg\cO_{(X/\oA)^\nyg}\simeq \rees(\fil_\nyg\Prism_{X/A})$ as derived algebras where the right hand side before the Rees construction is the relative Nygaard filtered prismatic cohomology of \cite{APC},

\item When $X/\oA$ is an affine $\oA$ scheme then there's an equivalence of stacks $(X/\oA)^\nyg\simeq \underline{\spf}(\rees(\fil_\nyg\Prism_{X/A}))$ where the formal spectrum of a derived algebra is defined as a relative analogue of \cite[Construction 1.2]{MM24},
\item When one extends the construction above to derived $\oA$ schemes, then in the \emph{quasi-lci} case \cite{APCII} these stacks remain classical.
\end{enumerate}

The results in \cite{Sah25Nyg} will then be used to study the syntomification of Frobenius liftable schemes in positive characteristic in \cite{Sah25Syn}.

Recasting the theory in terms of delta schemes, we will show the following.

\begin{theorem}
    Let $X/W$  be a derived $\delta$-scheme. Let $(X/W)^{\dR.+}\to \agm$ be the relative de Rham stack of $(X/W)$. Let $X_{p=0}$ denote the mod $p$-reduction of $X/W$ and let $(X_{p=0})^\nyg\to k^\nyg$ be the Nygaard filtered prismatization of $X_{p=0}.$ Then there's an equivalence of $k^\nyg$ stacks 

    $$(X_{p=0})^\nyg\simeq (X/W)^{\dR,+}\times_{\agm} k^\nyg.$$
\end{theorem}

As a consequence we will describe the stack $(X_{p=0})^\mathrm{Syn}$ as the colimit of the following diagram of stacks

\begin{theorem}
    $(X_{p=0})^\mathrm{Syn}$ can be described as the pushout
   
\[\begin{tikzcd}[cramped]
	{(X/W)^{\dR}\coprod X^{p\text{-dR}}} & {(X/W)^{\dR,+}\times k^\nyg} \\
	{(X/W)^{\dR}} & {(X_{p=0})^\mathrm{Syn}}
	\arrow["{j^{\mathrm{HT}},j^{\dR}}", from=1-1, to=1-2]
	\arrow["{(id,F)}"', from=1-1, to=2-1]
	\arrow[from=1-2, to=2-2]
	\arrow[from=2-1, to=2-2]
\end{tikzcd}\].
\end{theorem}

Here  $X^{p\text{-dR}}$ is the $p$-de Rham stack of $X/W$ and is related to recent work of Ogus \cite{Ogus2024CrystallinePrisms}. We will use this observation to describe $\Qcoh((X_{p=0})^\mathrm{Syn})$ in terms of Fontaine-Laffaille theory of \cite{faltings}, mirroring results in \cite{TerentiukVologodskyXu2025PrismaticFGauges} and forthcoming work of Madapusi-Mondal \cite{MadapusiMondal2026GeneralizedFontaineLaffaille}.

\appendix
\section{$\L\Sym$ on prestacks.}\label{appendix: lsym on prestacks}

The ideas presented in this appendix were essentially communicated to the author by Dmitry Kubrak; all mistakes are solely due to the author.

Recall that for a prestack $X\colon \aring\to \mathcal{S}$ we have defined $\dalg(X)$ by right Kan extension in Construction \ref{cons: right kan extension of dalg}. 

For our set theory conventions, see \ref{set theory}.

\begin{lemma}\label{lem: qcoh to dalg is monadic}
    The canonical functor $$U_X\colon \dalg(X)\to \Qcoh(X)$$ is monadic.
\end{lemma}
\begin{proof}
   The functor $U_X\colon \dalg(X)\to \Qcoh(X)$ is induced along a limit which we write out explicitly. 

   The morphisms $\dalg(X)=\lim_{\spec(R)\to X}\dalg_R$ has transition functors for any morphism $R\to S$ given by $-\otimes_R R$ by \cite[Remark 4.2.29]{Rak20} so that the following diagrams commute 
   
   \begin{equation}\label{eq: transition functors for dalg}
      \begin{tikzcd}
\dalg_S \arrow[r, "-\otimes_R R"] \arrow[d, "U_S"'] & \dalg_R \arrow[d, "U_R"] \\
\Mod_S \arrow[r, "-\otimes_R R"]                   & \Mod_R                
\end{tikzcd}\text{  and     }  \begin{tikzcd}
\dalg_S \arrow[r, "-\otimes_R R"]                      & \dalg_R                       \\
\Mod_S \arrow[r, "-\otimes_R R"] \arrow[u, "\L\Sym_R"] & \Mod_R \arrow[u, "\L\Sym_S"']
\end{tikzcd}.
\end{equation}

It follows that $U_X\colon \dalg(X)\to \Qcoh(X)$ is a right adjoint with left adjoint $L_X$ defined as the limits of the left adjoints $\L\Sym_R$ as $R$ ranges over maps $\spec(R)\to X.$

We first check that $U_X$ is conservative. This is easy as the functor is pointwise conservative. We now want to check that $U_X$ preserves sifted colimits to check Barr-Beck-Lurie. Note that pullbacks for both categories coming from the diagrams in \ref{eq: transition functors for dalg} commute with colimits (as they are left adjoints) the commutation of $U_X$ with sifted colimits can be checked pointwise, where it follows from \cite[Notation 4.2.28 (c)]{Rak20}.

\end{proof}

\begin{remark}
    Lemma \ref{lem: qcoh to dalg is monadic} shows that there is a left adjoint $L_X\colon \Qcoh(X)\to \dalg(X)$ with the property that when $X$ is affine, then $U_X\circ L_X=\L\Sym_X$. 
    We want to understand this adjoint $L_X$ more globally as a natural monad $\L\Sym_X$ acting on $\Qcoh(X)$ and for this we will use a refined result of Kubrak-Li \cite{KubrakLiKGn_inprep} 
\end{remark}

\begin{construction}[Kubrak-Li \cite{KubrakLiKGn_inprep}]\label{cons: Kubrakli}

As explained in Remark \ref{remark: two cons of dalg are equivalent}, we know that the assignment $R\mapsto \Mod_R$ has an endo-transformation given by the composite $U\circ \L\Sym$ in the sense that for any morphism $S\to R$ we have 
$$
    U_S(\L\Sym_S(M))\otimes_RR\simeq U_R(\L\Sym_R(M\otimes_R R)).
$$

Thus by unstraightening \cite[Theorem 3.2.0.1]{HTT} one gets an endo-transformation, denoted abusively by $\L\Sym$, of the coCartesian fibration $\Mod\to \aring.$ 

Here by an endo-transformation of $\Mod\to \aring$ we mean all endo-transformations which preserve all coCartesian arrows. We denote this $\icat$ by $\End_{\aring}(\Mod)$. The result of Kubrak-Li in their work \cite{KubrakLiKGn_inprep} is that $\L\Sym$ naturally refines to an $\mathbf{E}_1$-algebra object in $\End_{\aring}(\Mod)$. 

Now let $F\colon J\to \aring$ be any diagram of animated rings and let $\cC$ be the $\icat$ given by $\lim_{j\in J}\Mod_{F(j)}$. 

Then one has an induced monoidal functor\footnote{One may also compute this as the pullback of the coCartesian fibration $\Mod\to \aring$ along the diagram $F\colon J\to \aring$. Indeed there is a natural action of $\L\Sym$ on the $\icat$ of coCartesian sections of $\Mod\times_{\aring}J\to J$.}
$$\Theta_{\cC}\colon  \End_{\aring}(\Mod)\to \End(\cC).$$

Explicitly for any $T\in \End_{\aring}(\Mod)$, the functor $\Theta_{\cC}(T)$ acts on $\cC=\lim_{j\in J}\Mod_{F(j)}$ by acting compatibly by $T_{F(j)}$ on $\Mod_{F(j)}.$ This functor is monoidal because it clearly respects composition.

We denote the image $\Theta_{\cC}(\L\Sym)\in \Alg(\End(\cC))$ of $\L\Sym\in \Alg(\End{\Mod})$ by $\L\Sym_{\cC}$.
It follows now that we may contemplate the $\icat$ $\Mod_{\L\Sym_\cC}(\cC).$

Note that for each $j\in J$, there is a pullback functor $$\eta_{j}^*\colon \Mod_{\L\Sym_\cC}(\cC)\to \Mod_{\L\Sym_{F(j)}}(\Mod_{F(j)})=\dalg_{F(j)}.$$

Set $\dalg(\cC)=\lim_{j\in J} \dalg_{F(j)}$. Then one gets a natural morphism
\begin{equation}\label{eq: pullback monads}
    \eta^*_{\cC}\colon \colon \Mod_{\L\Sym_\cC}(\cC)\to \dalg(\cC).
\end{equation}

When $F:J\to \aring$ is the inclusion of the slice category $\daff_X$ i.e. the category of derived affine schemes with a map to $X$, for some prestack $X$, then $\Mod_{\L\Sym_\cC}(\cC)$ is denoted $\Mod_{\L\Sym_X}(\Qcoh(X)).$ Thus the functor of \ref{eq: pullback monads} is now denoted by 

\begin{equation}\label{eq: pullback monads on stacks}
\eta^*_X\colon \Mod_{\L\Sym_X}(\Qcoh(X))\to \dalg(X).
\end{equation}
    
\end{construction}

\begin{proposition}\label{prop: kubrak prop}
   Keeping the notation of Construction \ref{cons: Kubrakli}, let $F\colon J\to \aring$ be a diagram of animated rings. Then the morphism $$\eta^*_{\cC}\colon \Mod_{\L\Sym_{\cC}}(\cC)\to \dalg(\cC)$$ of \ref{eq: pullback monads} is an equivalence of $\icats.$
   In particular for any prestack $X$, the functor $$\eta_X^*\colon  \Mod_{\L\Sym_X}(\Qcoh(X))\to \dalg(X)$$ of \ref{eq: pullback monads on stacks} is an isomorphism.
\end{proposition}

\begin{proof}
    First observe that the natural forgetful functor $U^{\dalg}_{\cC}\colon \dalg(\cC)\to \cC$ is monadic for the same reason as the proof of Lemma \ref{lem: qcoh to dalg is monadic}. 
    
    Let $T_{\cC}$ denote the induced monad on $\cC.$ Explicitly, it sends $(X(j))_{j\in J}\in \cC$ to the object $(\L\Sym_{F(j)}(X(j))_{j\in J}\in \dalg(\cC)$ along with its coordinate wise induced monadic structure.

   Further note that the map $\eta^*_{\cC}\colon \Mod_{\L\Sym_{\cC}}(\cC)\to \dalg(\cC)$ is also monadic. Indeed, it is clearly conservative and preserves geometric realisations since geometric realisations in both source and target can be computed after forgetting to $\cC.$ We do need to check that it is a right adjoint. For this we check that it preserves limits. But this can be checked by forgetting to $\cC$ and using presentability of both $\Mod_{\L\Sym_{\cC}}(\cC)$ and $\dalg(\cC)$ (in a larger universe) shows that $\eta^*_{\cC}$ admits a left adjoint.

    Now the forgetful functor $U_\cC\colon \Mod_{\L\Sym_{\cC}}(\cC)\to \cC$ factors as $$\Mod_{\L\Sym_{\cC}}(\cC)\xrightarrow{\eta_{\cC}^*}\dalg(\cC)\xrightarrow{U_{\cC}^\dalg}\cC$$ with both functors monadic. 
    Thus it suffices to show that the monad induced by $\eta_{\cC}^*$ is the trivial monad. For this it will suffice to show that if $X\in \cC$ then  $\eta_{\cC}^*(\L\Sym_{\cC}(X))$ agrees with the free object for the monad $T_\cC$ on $\cC$. 

    For this first note that $$\eta_{\cC}^*(\L\Sym_{\cC}(X))=(\L\Sym_{F(j)}X(j))_{j\in J}\in \dalg(\cC).$$

    Thus we compute for any $Y:=(Y(j))_{j\in J}\in \dalg(\cC)$ we have

    \begin{align*}
        \Maps_{\dalg(\cC)}(\eta_{\cC}^*(\L\Sym_{\cC}(X)),Y) &=\Maps_{\dalg(\cC)}(\L\Sym_{F(j)}X(j))_{j\in J}, (Y(j))_{j\in J})\\
        &=\lim_{j\in J}\Maps_{\dalg_{F(j)}}(\L\Sym_{F(j)}X(j), Y(j))\\
        &=\lim_{j\in J}\Maps_{\Mod_{F(j)}}(X(j), Y(j))\\
        &=\Maps_{\cC}(X,Y).
    \end{align*}

\end{proof}

\bibliographystyle{alpha}
\bibliography{references}

\end{document}